\documentclass[11pt]{article}
\usepackage{latexsym}
\usepackage{amsmath,amsthm,amsfonts,amssymb}
\usepackage[T1]{fontenc}
\usepackage[utf8]{inputenc}
\usepackage{aeguill}
\usepackage{url}
\usepackage{graphicx}
\usepackage[a4paper,textwidth=14cm,textheight=25cm,marginparwidth=2.6cm]{geometry}
\usepackage{xspace,marginnote}
\usepackage{export}
\usepackage{hyperref}

\newtheorem{thm}{Theorem}[section]

\newtheorem*{thm*}{Theorem}
\newtheorem{dfn}[thm]{Definition} 
\newtheorem*{dfn*}{Definition}

\newtheorem{cor}[thm]{Corollary}
\newtheorem*{cor*}{Corollary}

\newtheorem{prop}[thm]{Proposition} 
\newtheorem*{prop*}{Proposition} 
\newtheorem*{properties*}{Properties} 
 
\newtheorem{lem}[thm]{Lemma} 
\newtheorem*{lem*}{Lemma}

\newtheorem*{claim*}{Claim} 
 
\newtheorem*{fact*}{Fact}

\newtheorem*{qst*}{Question}

\newtheorem*{pb*}{Problem}

\theoremstyle{remark}
 
\newtheorem*{algo*}{Algorithm} 
\newtheorem*{rem*}{Remark}
\newtheorem{rem}[thm]{Remark}
\newtheorem*{example*}{Example}
\newtheorem{example}[thm]{Example}

\newenvironment{preuve}[1][Preuve]{\begin{proof}[#1]}{\end{proof}}

\newcounter{numEnonceTmpInterne}
\newenvironment{enonce*}[1]{\theoremstyle{plain}\stepcounter{numEnonceTmpInterne}%
\def\a{enoncetmp\alph{numEnonceTmpInterne}}%
\newtheorem*{\a}{#1}\begin{\a}}{\end{\a}}

\makeatletter
\edef\@tempa#1#2{\def#1{\mathaccent\string"\noexpand\accentclass@#2 }}
\@tempa\rond{017}
\makeatother

\newcommand{\es}{\emptyset}
\renewcommand{\phi}{\varphi} 
\newcommand{\m} {^{-1}} 
\newcommand{\eps} {\varepsilon}

\newcommand {\ra} {\rightarrow}

\newcommand {\onto} {\twoheadrightarrow}
\newcommand {\into} {\hookrightarrow}
    
\newcommand{\imp} {\Rightarrow}
\renewcommand{\implies} {\Rightarrow}
\renewcommand{\iff} {\Leftrightarrow}

\newcommand{\ul}[1]{\underline{#1}} 
\newcommand{\ol}[1]{\overline{#1}}

\newcommand{\ie} {i.e.\ }

\newcommand {\calc} {{\mathcal {C}}}   
   
\newcommand {\cale} {{\mathcal {E}}}   
   
\newcommand {\calg} {{\mathcal {G}}}   
\newcommand {\calh} {{\mathcal {H}}}

\newcommand {\calk} {{\mathcal {K}}}   
   
\newcommand {\calm} {{\mathcal {M}}}

\newcommand {\cals} {{\mathcal {S}}}

\newcommand {\bbF} {{\mathbb {F}}}

\newcommand {\bbZ} {{\mathbb {Z}}}   

\newcommand{\grp}[1]{\langle #1 \rangle}

\newcommand{\Out} {{\mathrm{Out}}}

\newcommand{\Aut} {{\mathrm{Aut}}}

\newcommand{\id} {\mathrm{id}}

\usepackage{ stmaryrd }

\newcommand{\Tcan}{T_{\mathrm{can}}}
\newcommand{\Gcan}{\Gamma_{\!\mathrm{can}}}
\newcommand{\tGcan}{\tilde\Gamma_{\!\mathrm{can}}}
\newcommand{\Z}{{\mathbb {Z}}}

\usepackage{pdfsync}
\newcommand{\inc}{\subset}

\newcommand {\R} {{\mathbb {R}}}

\newcommand {\F} {{\mathbb {F}}}

\newcommand{\core}{\mathrm{c}}
\newcommand{\ecore}   {\mathrm{\widetilde c}}      
\newcommand{\bpm}{boundary-preserving map}
\newcommand{\ndbpm}{non-degenerate boundary-preserving map}
\newcommand{\Ndbpm}{Non-degenerate boundary-preserving map}

\makeatletter
\newcommand{\oset}[3][-0.1ex]{%
  \mathrel{\mathop{#3}\limits^{
    \vbox to#1{\kern-2\ex@
    \hbox{$\scriptstyle#2$}\vss}}}}
\makeatother

\newcommand{\tgt}{\mathrel{>_t}}
\newcommand{\tge}{\mathrel{\geqslant_t}}

\newcommand{\gtun}{\mathrel{>_1}}

\newcommand{\bo}{\partial}

\begin{document}

\title{Towers and the first-order theory  of hyperbolic groups}
\author{Vincent Guirardel, Gilbert Levitt, Rizos Sklinos\thanks{V.G.\  conducted this work within the framework of the Henri Lebesgue Center ANR-11-LABX-0020-01. R.S.\ was supported by the National Science Foundation under Grant No. 1953784.}}

\maketitle

\begin{abstract}
  This paper is devoted to the first-order theory of   torsion-free hyperbolic groups. 
 One of its purposes  is to review some   results and to provide precise and correct statements and definitions, as well as some proofs and new results. 
 
A key concept  is that of a tower (Sela) or NTQ system (Kharlampovich-Myasnikov). We discuss them thoroughly.
  
We  state and prove a new general theorem 
which unifies 
several results in the literature:
elementarily equivalent  torsion-free
  hyperbolic groups have isomorphic cores (Sela);
  if $H$ is  elementarily embedded in a  torsion-free  hyperbolic group $G$, then $G$ is a  tower over $H$ relative to $H$ (Perin); free groups (Perin-Sklinos, Ould-Houcine), and more generally  free products of prototypes and free groups, are homogeneous.
 
 The converse to Sela and Perin's  results just mentioned is true. This follows from the solution to Tarski's problem on elementary equivalence of free groups, due independently to Sela and Kharlampovich-Myasnikov, which we treat as a black box throughout the paper.

We  present many examples and counterexamples, and 
 we prove some new model-theoretic results.
We characterize prime models among torsion-free hyperbolic groups, and minimal models among elementarily free groups. 
Using  Fra\"issé's method, we associate to  every torsion-free hyperbolic group $H$ 
a unique homogeneous countable group
$\calm$ in which any hyperbolic group $H'$ elementarily equivalent to $H$ has an elementary embedding.  

In an appendix we give a complete proof of the fact, due to Sela, that towers over a torsion-free hyperbolic group $H$ are $H$-limit groups.

\end{abstract}


\section{Introduction}

Two groups $G,G'$ are \emph{elementarily equivalent} (denoted $G\equiv G'$)
if they satisfy the same first-order sentences in the language of groups $\mathcal{L}:=\{\cdot, ^{-1}, {\bf 1}\}$.
In 1946 Tarski asked whether all non-abelian free groups are elementarily equivalent.
Although the class of free groups had been extensively studied by group theorists and geometers, the existing tools at the time seemed to be inadequate 
and the  question proved very hard to tackle. It was only after more than fifty years, in 2001, that Sela and separately Kharlampovich-Myasnikov answered the question in the positive. Both works required profound tools to be discovered 
and culminated in many voluminous papers that take in total more than a thousand pages. In addition, the techniques developed allowed Sela   to classify torsion-free hyperbolic groups up to elementary equivalence \cite{Sela_diophantine7}. 

One should note that Tarski's question was answered in the following strong way: the chain of non-abelian free groups under the natural embeddings is elementary. This means that not only do non-abelian free groups share the same first-order theory, 
but also elements (or rather finite tuples) of a non-abelian free group $F$ do not change first-order properties when seen as elements of a larger free group $F*F'$: the natural embedding $F\into F*F'$ is elementary. 
As in the case of elementary equivalence, the techniques generalize  
to certain embeddings of  torsion-free hyperbolic groups. 

One purpose of the current paper is to review some of the above results and to provide precise and correct statements and definitions, as well as some proofs and further results. 
A key concept of this work is that of a tower. This notion has been introduced by Sela and is the counterpart of
  Kharlampovich-Myasnikov's NTQ systems (for Non-degenerate Triangular Quasi-quadratic systems). 

It  has been 
used for both the characterization of elementary embeddings and the characterization of elementary equivalence.  
For instance, 
combining results by Sela \cite{Sela_diophantine7} and Perin  \cite{Perin_elementary}, one can almost state that, given a torsion-free hyperbolic group $G$ and a  non-abelian subgroup $H$, the embedding $H\into G$ is elementary if and only if $G$ is a tower over $H$
(precisely: if and only if $G$ is an extended tower over $H$ relative to $H$, see below).

For this to be true, one  must be careful with the definition of a ``tower'' and make additional requirements 
on the tower involved in this statement (see Theorem \ref{thclintro} below);  in particular, the original
definition of hyperbolic $\omega$-residually free towers of \cite{Sela_diophantine1} does not exactly characterize elementarily free groups (and neither do
regular NTQ groups),
as    first found out by the third named author while working on \cite{PeSk_homogeneity}.
As we shall see, these problems come from low-complexity 
surfaces and are  the source of   faulty results in the literature.

 We quickly mention some of the main ideas of the paper, with a detailed overview containing precise statements   appearing in the next section.  All groups $G$ considered in the paper are assumed to be finitely generated and torsion-free.

Following Sela and Perin, we distinguish two kinds of towers  (Section \ref{pr}): simple towers (\cite{Sela_diophantine1,Sela_diophantine7}),
and more general extended towers (\cite{Perin_elementary,Perin_elementary_erratum})   where the base may be ``disconnected''.
The two notions are closely related, as one can transform an extended tower into a simple one by taking a free product with a free group (Proposition \ref{tstab}).

Associated to the notion of a tower is a partial ordering on isomorphism classes of groups: $G\tge H$ if $G$ is an extended tower over $H$.
Following Sela, minimal elements for this ordering are called \emph{prototypes},
and a \emph{core} of $G$ is a prototype $C$ such that $G\tge C$.
 In Section \ref{ordre} we prove the uniqueness of the core (up to isomorphism) for any finitely generated  torsion-free CSA group  $G$. This is due to
Sela \cite{Sela_diophantine7} for hyperbolic groups; we give a different proof.     

 In order to relate towers to model theory, we isolate a result by Sela (Theorem \ref{Tarski} below) which implies the solution to Tarski's problem,
and which we unfortunately have to take as a black box. We refer to it as ``Tarski'' in the paper.
It says that, if  $G$ and $H$ are non-abelian torsion-free hyperbolic groups such that   $G$ is a simple tower
over $H$, then $H$  is  elementarily embedded in $G$.

 In Section \ref{interp} we  state and prove a new general result going in the converse direction (Theorem \ref{lethmintro}  below).
It is stated in terms of partial elementary maps, which can be interpreted as elementary equivalence of two groups in
a language with constants.
It says that a partial elementary maps between torsion-free hyperbolic groups
extends to an isomorphism between their  relative  cores.
This allows us to unify 
several results in the literature:
\begin{itemize}
\item 
 elementarily equivalent  torsion-free
  hyperbolic groups have isomorphic cores (Sela \cite{Sela_diophantine7}, the converse follows from ``Tarski'')
\item   if $H$ is  elementarily embedded in a  torsion-free  hyperbolic group $G$, then $G$ is an extended tower over $H$ relative to $H$ 
  (Perin \cite{Perin_elementary}, here also the converse follows from ``Tarski'')
\item the homogeneity of free groups (due to Perin and Sklinos \cite{PeSk_homogeneity}   and Ould-Houcine
\cite{Houcine_homogeneity}), and more generally of free products of prototypes and free groups
  (Theorem  \ref{thm_homogeneity_intro} below).
\end{itemize}
We devote the full Section \ref{examp} to working out examples and counterexamples.

In the last section we prove some new model-theoretic results.
We characterize prime models among torsion-free hyperbolic groups, and minimal models among elementarily free groups. 
Using Theorem \ref{lethmintro}  and Fra\"issé's method, we associate to  every torsion-free hyperbolic group $H$ 
a unique homogeneous countable group
$\calm$ in which any hyperbolic group $H'$ elementarily equivalent to $H$ has an elementary embedding (see Theorem \ref{universintro}).  Finally, Theorem \ref{lethmintro} allows us to characterize tuples of elements of torsion-free hyperbolic groups with the same first-order properties in terms of towers  (Corollary \ref{type_char}).  

In an appendix we give a complete proof of the fact, due to Sela, that towers over a torsion-free hyperbolic group $H$ are $H$-limit groups.

\tableofcontents

\section{Overview of the paper}

\begin{figure}[ht!]
  \includegraphics{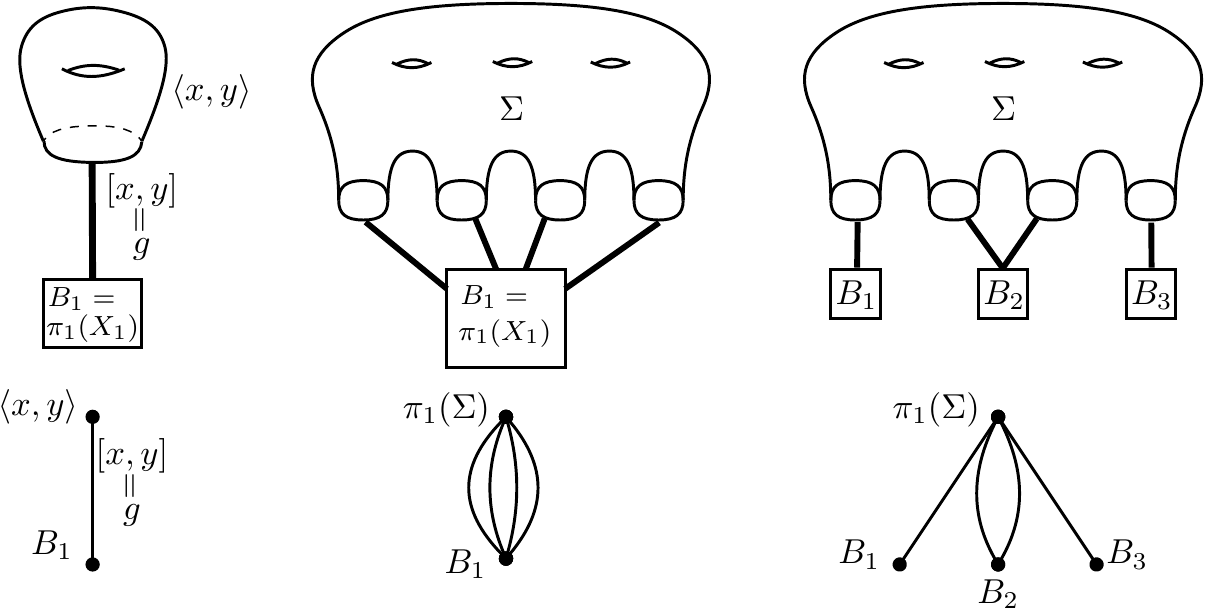}
  \caption{Two simple towers and an extended tower,   with the underlying graphs of groups.}\label{fig_intro}
\end{figure}

\paragraph{Simple towers}

To introduce towers, let us start with an example (see Figure \ref{fig_intro}).
Let $G=B_1*_{g=[x,y]}F(x,y)$ be the cyclic amalgam of a group $B_1$ with a  free group of rank 2, which we view as the fundamental group of a punctured torus; we identify the boundary element $[x,y]$ with some non-trivial $g\in B_1$ (all groups are assumed to be torsion-free, so $g$ has infinite order). 

If $g$ is not a commutator in $B_1$, then clearly the embedding of $B_1$ into $G$ is not elementary: 
the sentence saying that $g$ is a commutator is true or false, depending on whether we   interpret it in $G$ or in $B_1$.

Now suppose that $g$ is a commutator $[a,b]$ in $B_1$. In this case we say that $G$ is a tower over $B_1$;  it follows from ``Tarski'' (Theorem \ref{Tarski})  
that the embedding of $B_1$ into $G$ is elementary if $G$ is hyperbolic.

This example may be generalized, leading to the definition of a simple tower. 
Instead of gluing a once-punctured torus to $B_1$, we may glue any compact   surface $\Sigma$  (possibly non-orientable, with several boundary components).
In topological terms, we are attaching every boundary component $C_i$ of $\Sigma$ to a space $X_1$ with $\pi_1(X_1)=B_1$
along curves representing   non-trivial elements of $B_1$  
(see Figure \ref{fig_intro}).
Van Kampen's theorem expresses the resulting group $ G$ as the fundamental group of a graph of groups $\Gamma$ having two vertices: a central vertex $v$ carrying $\pi_1(\Sigma)$ and a bottom vertex $v_1$ carrying $B_1$; there is one edge  between $v$ and $v_1$ per boundary component of $\Sigma$. 

We shall say that $G$ is a \emph{simple \'etage} over $B_1$ if  $B_1$ is non-abelian and two conditions are satisfied. First, $\Sigma$ must be \emph{non-exceptional}: it must carry a pseudo-Anosov homeomorphism; equivalently, $\Sigma$  is a once-punctured torus or its Euler characteristic satisfies $\chi(\Sigma)\le-2$. 

The second condition is a generalization of $g$ being a commutator in $B_1$ in the example studied above, but it is expressed without referring to a specific element. There are two ways of stating it.  

One may require the existence of a \emph{retraction} $\rho:G\to B_1$ such that $\pi_1(\Sigma)$ has  non-abelian image (in the example, $\rho$ sends $x$ to $a$ and $y$ to $b$  if $g=[a,b]$). One may also require the existence of a map $r$ from $\pi_1(\Sigma)$ to $G$ (such as    the restriction of $\rho$ to $\pi_1(\Sigma)$) which is a \emph{\ndbpm}{}    (called local preretraction in \cite{GLS_finite_index}): it must agree with a conjugation on each boundary subgroup $\pi_1(C_i)$, have  non-abelian image,   not be an isomorphism between  
$\pi_1(\Sigma)$ and some conjugate. 

It is shown in Proposition \ref{equivsimple} that the two requirements are equivalent if $G $ is a   torsion-free CSA group and $B_1$ is not abelian.

If $\Sigma$ is non-exceptional and $\rho$ (equivalently, $r$) exists, with $B_1 $ non-abelian, we say that $G$ (assumed to be CSA)  is a \emph{simple \'etage of surface type} over $B_1$ (we use the French word \'etage rather than the English word floor to emphasize the fact that there are two levels: a single-story house has no \'etage). 

There is another type of \'etage: we say that $G$ is a \emph{simple \'etage of free product type} over $B_1$ if $G=B_1*\Z$.  We mention for completeness another type of \'etage, 
  although it will not be used   in this paper:  
$G$ is a {\em simple \'etage of abelian type}  over $B_1$ if $G$ is an 
  amalgam $G=B_1*_A (A\oplus \bbZ)$,
with ${A}$   a maximal abelian subgroup of $B_1$. Abelian \'etages play a role in obtaining formal solutions/implicit function theorems  
(see \cite{Sela_diophantine2,KhMy_implicit}),
but they are not used in characterizing elementary equivalence and embeddings.

A  {tower} is obtained by stacking  {\'etages} 
on top of each other:   $G$ is a \emph{simple tower} over a subgroup $H$ if there exists a chain of subgroups
$G=G_0>G_1>\dots >G_k=H$ such that $G_i$ is a simple étage over $G_{i+1}$; 
we allow the case $k=0$, so $G$ is always a simple tower over itself.
 We note that $\omega$-residually free towers,  as defined in \cite{Sela_diophantine1}, are precisely
simple towers over a (non-abelian) subgroup which is a    free product of 
cyclic groups and 
 fundamental groups of closed surfaces
of Euler characteristic $\leq -2$ (see Remark \ref{zenbas}).

We can now state the   \emph{main elementary embedding theorem}, which implies the solution to Tarski's problem (we shall refer to it simply as ``Tarski'' and treat it as a black box):

\begin{thm}[``Tarski'', {\cite[Theorem 7.6]{Sela_diophantine7}}]
\label{Tarski}
 Let $G,H$ be non-abelian torsion-free hyperbolic groups. If $G$ is a  \emph{simple} 
 tower
 over  $H$,   then the inclusion $H\into G$ is an elementary embedding.
In particular, $G$ and $H$ are elementarily equivalent.  
 \end{thm}

This statement 
 does not explicitly appear in \cite{Sela_diophantine7} but, as pointed out in \cite{Perin_elementary},
it follows from the arguments used in the proof of Theorem 7.6 of \cite{Sela_diophantine7}.

The exact  converse to Theorem \ref{Tarski} is not true:
in Subsection \ref{nst} we construct a group which is not a simple \'etage but contains a proper elementarily embedded subgroup. 
 The opposite direction in Theorem \ref {Tarski}, which  is due to Perin \cite{Perin_elementary},  
thus has to be formulated differently. 

We shall rephrase Perin's result as follows (see Theorem \ref{thclintro} below):  if $H\into G$ is an elementary embedding,
  there exists $r\geq 0$ such that the free product
$G*\bbF_r$ of $G$ with the free group of rank $r$ is a simple tower over $H$.
If we want to avoid using an auxiliary free group, we have to use   {extended towers}.

\paragraph{Extended towers.}   These towers were introduced by Perin in \cite{Perin_elementary,Perin_elementary_erratum}. 

We defined a simple \'etage of surface type by attaching a surface $\Sigma$ to a single connected space $X_1$ (see   the first two pictures on Figure \ref{fig_intro}). 
We now consider   connected  spaces $X_1,\dots, X_n$ with $n\ge1$, and we attach each boundary component of $\Sigma$ to one of these spaces so as to get a connected space.  

The associated graph of groups  $\Gamma$, which we call a centered splitting (see Definition \ref{centspl}), has bottom vertices $v_i$ for $i=1,\dots,n$, each carrying a group   $B_i=\pi_1(X_i)$; each $v_i$ is  joined to the central vertex $v$ carrying $\pi_1(\Sigma)$ by one or more edges    (see the last picture on Figure \ref{fig_intro}).

If   the surface $\Sigma$ is not exceptional and there exists a \ndbpm{}  $p:\pi_1(\Sigma)\to G$, or  equivalently a retraction
$\rho:G\to \tilde B_1*\dots*\tilde B_n$  with $\rho(\pi_1(\Sigma))$    non-abelian, for some choices of conjugates $\tilde B_i$ of $B_i$,
we say that $G$ is an \emph{extended \'etage of surface type} over $H=\tilde B_1*\dots*\tilde B_n$ (as pointed out in \cite{Perin_elementary_erratum}, the requirements on   the retraction $\rho$ must be modified when $n=1$ and $B_1$ is abelian).
We then define \emph{extended towers} by stacking extended \'etages of surface type and \'etages of free product type.

If no surface of low complexity is involved, any extended tower may be upgraded to a simple tower:
we prove (Proposition \ref{3simple}) that, if $G$ is an extended \'etage over $H$ with the associated surface $\Sigma$ having Euler characteristic $\chi(\Sigma)\le-3$, then $G$ is a simple \'etage over   some subgroup $H'$ isomorphic to $H$.
We also point out (Proposition \ref{tstab}) that, regardless of $\Sigma$, one obtains a simple \'etage after stabilizing: if $G$ is an extended \'etage over $H$, some $G*\bbF_r$ 
is a simple \'etage over a subgroup   $\tilde H\subset G*\bbF_r$ isomorphic to $H$. 

\begin{figure}[ht!]
  \begin{centering}
    \includegraphics{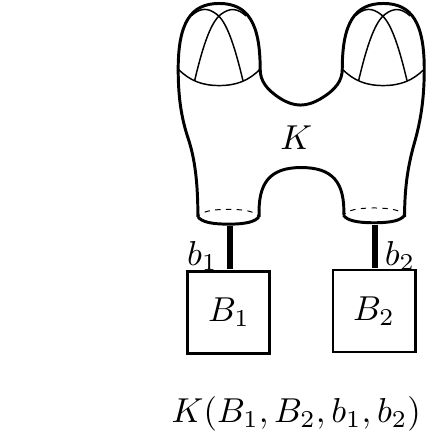}
    \caption{The group $K(B_1,B_2,b_1,b_2)$ constructed by attaching a twice-punctured Klein bottle.  This is an extended étage over $B_1*B_2$
      if $b_1$ and $b_2$ are squares in $B_1$ and $B_2$ respectively.}\label{fig_Klein_intro}
  \end{centering}
\end{figure}

On the other hand, consider as in  Figure \ref{fig_Klein_intro}  and Example \ref{ctrex} a double cyclic amalgam  
$$G=K(B_1,B_2,b_1,b_2)=B_1*_{b_1={g_1}} K *_{ {g_2}=b_2} B_2$$ 
with 
  $K=\langle g_1,g_2,x,y\mid g_1g_2=x^2y^2\rangle$ 
  the fundamental group of a twice-punctured Klein bottle 
 (note that 
it is a non-exceptional surface   since its Euler characteristic is $-2$).
It defines an extended \'etage over $B_1*B_2$ if each $b_i$ is a square  in $B_i$.  
If $B_1$ and $B_2$ have no cyclic splitting,  then $G$ is not a simple \'etage  (see Lemma \ref{pastour}).
  Moreover,   $G$ has no proper elementarily embedded subgroup (see  Proposition  \ref{lebut}):  it  is \emph{minimal}.  Similar examples may be constructed using other surfaces of Euler characteristic  $-2$.

\paragraph{Classifying elementarily embedded subgroups.}
Because of such examples, 
Theorem \ref{Tarski} is not true for extended towers. 
But  it  will be true 
if   we add an 
extra condition. 
One says that an étage is \emph{relative} to a subgroup $H<G$ if $H$ is conjugate into one of the $B_i$'s (\'etage of surface type), into $B_1$ if $G=B_1*\Z$.    
A tower is relative to $H$ if all its étages are (see Definition \ref{relto}).
We can now state the following characterization 
of elementarily embedded subgroups.

\begin{thm}[\cite{Perin_elementary}, see Theorem \ref{thcl}] 
\label{thclintro}
Let $G$ be a non-abelian torsion-free hyperbolic group,  and let $H<G$ be a non-abelian subgroup.
The   following are equivalent:
\begin{enumerate}
\item $H$ is elementarily embedded in $G$; 
\item $G$ is an extended tower over $H$ relative to $H$ (in the sense of Definition \ref{relto});
\item there exists a finitely generated free group $F$ such that $G*F$ is a simple tower over  $H$ (viewed as a   subgroup of $G*F$ in the obvious way).
\end{enumerate}
\end{thm}
$1\implies 2$ is Perin's original statement. 
Deducing 1 from 2 or 3 requires ``Tarski'' (Theorem \ref{Tarski}).

\paragraph{Classifying hyperbolic groups up to elementary equivalence.}
Given torsion-free CSA groups $G$ and $H$, we write $G\tge H$  if $G$ is an extended tower over  a subgroup isomorphic to $H$.
This defines a \emph{partial order} on the set of isomorphism classes, and minimal elements are called \emph{prototypes.} 
In Section \ref{ordre} we   prove that, given a CSA group $G$, there is a unique prototype $\core(G)$ such that $G\tge \core(G)$; it is called the \emph{core} of $G$.  

The difficulty here is to prove uniqueness. In the case of hyperbolic groups, Sela deduces it from ``Tarski'', while our arguments in Section \ref{ordre} are purely group-theoretical. We associate a \emph{minimal base} $M$ to any 2-acylindrical cyclic splitting  
of a group $G$, and we use it to show confluence: if $G_0\tge G_1$ and $G_0\tge G_2$, with $G_0 $ a CSA group, there exists $G_3$ such that $G_1\tge G_3$ and $G_2\tge G_3$.

 The core behaves nicely with respect to free products: $c(G_1*G_2)\simeq c(G_1)*c(G_2)$, and a group is a prototype
if and only if it is a free product of one-ended prototypes.

In Section \ref{elemequiv} we use extended towers to present Sela's classification of torsion-free hyperbolic groups up to elementary equivalence. Note that any finitely generated group which is elementarily equivalent to a hyperbolic group is itself hyperbolic (\cite{Sela_diophantine7,Andre_hyperbolicity}).

We shall deduce from  Theorem \ref{lethmintro} below 
 that, if   two hyperbolic groups $G$ and $G'$ are elementarily equivalent (denoted $G\equiv G'$),  their cores are isomorphic. 
Our arguments do not use ``Tarski'', but ``Tarski'' is needed to prove that 
$G$ is elementarily equivalent to its core, 
 thus obtaining Sela's complete classification.

\begin{thm} [\cite{Sela_diophantine7}, see Theorem \ref{classif}]\label{classifintro}
Two non-abelian torsion-free hyperbolic groups are elementarily equivalent if and only if their cores are isomorphic.
\end{thm}

To be precise, it is not quite true that $G$ is always elementarily equivalent to its core: if $G\equiv \F_n$, its core is the trivial group, which is not equivalent to $G$.  To remedy this, we also define the \emph{elementary core} $\ecore(G)$, which is 
equal to $\core(G)$ if $\core(G)\ne\{1\}$ and  to $\F_2$ if $\core(G)=\{1\}$, so that $G\equiv\ecore(G)$ always holds.
  Note however that $\tilde c(G)$ does not behave as well as $c(G)$ with respect to free products.

Corollary \ref{cor_tarski} gives many characterizations of groups which are elementarily free ($G\equiv \F_2$),   correcting
\cite[Theorem 7]{Sela_diophantine6} and \cite[Theorem 41]{KhMy_elementary}.

 The elementary core
$\ecore(G)$ being  isomorphic to a subgroup of $G$   and  elementarily equivalent to $G$, one may wonder whether it  embeds  elementarily  into $G$.
This is the case if $\ecore(G)$ is one-ended
(Corollary \ref{cor_core_unbout})
 but  
the following result,   based on the groups $G=K(B_1,B_2,b_1,b_2) $ introduced above,  shows that this is not true in general,
thus correcting Theorem 7.6 of \cite{Sela_diophantine7}. 
\begin{thm}[See Theorem \ref{contrex}]\label{contrexintro}
  If $P$ is a prototype with more than one end, there is a torsion-free hyperbolic group $G$ such that $G$ is an extended \'etage over $P$
    (so that $P\simeq\ecore(G)$ is the core of $G$ and  $G\equiv P$), but $P$ does not elementarily embed into $G$.
\end{thm}

 One can get an elementary embedding of the core, however, if one is ready to stabilize (Lemma \ref{coretour}):
  some $G*\bbF_r$ with  $r\geq 0$ has an elementarily embedded subgroup which is isomorphic to $\ecore(G)$.
 
  For example, if      $G=K(B_1,B_2,b_1,b_2) $ with $B_i,b_i$ as above ($b_i$ is a square and $B_i$ has no cyclic splitting), and $P=B_1*B_2$,
 then $P\simeq\ecore(G)$ has no elementary embedding in $G$, but one can show that $B_1*tB_2t\m$ is an elementarily embedded subgroup of $G*\grp{t}$.
 
We also discuss in Section \ref{examp} to what extent $\core(G)$ is unique as a subgroup of $G$, and we study \emph{parachutes}. These are groups having   an arbitrary  
 centered splitting with $n=1$ and $B_1\simeq\Z$.  Some of these groups admit \ndbpm s, but   no retraction from $G$ to $B_1$ may have non-abelian image, so these groups are an exception in the theory. We determine which of these groups are simple \'etages, and which have a proper elementarily embedded subgroup.

\paragraph{Extending partial elementary maps.}  One of the main results of  the paper
is Theorem \ref{lethmintro} below. This result may be viewed as a common generalization of results by Sela and Perin mentioned above,
and it also implies the homogeneity of free groups \cite{PeSk_homogeneity,Houcine_homogeneity}, and  more generally of their free products with prototypes (see Theorem \ref{thm_homogeneity_intro} below). 

We consider hyperbolic groups $G,G'$ with a subgroup  $A\inc G$, and $f:A\to G'$ sending $A$ isomorphically to  a subgroup   $A'\subset G'$.
We say that $f$ is a \emph{partial elementary map} if,
for any
formula $\phi(x_1,\dots,x_n)$ and any $a_1,\dots,a_n$ in $A$, one has $G\models\phi(a_1,\dots,a_n)$ if and only if $G'\models\phi(f(a_1),\dots,f(a_n))$.

 The notion of partial elementary map somehow interpolates between elementary equivalence and elementary embedding:
if $A=\{1\}$, the trivial map is  partial elementary  if and only if  $G\equiv G'$; if $A=G$, a partial elementary  map $f$ defines an elementary embedding of $G$ into $G'$. 

A \emph{core of $G$ relative to $A$} is a subgroup $C\inc G$ containing $A$, such that $G$ is a tower over $C$ relative to $A$, and $C$ is not an \'etage relative to $A$. 
Thus, if $A=\{1\}$, the core of $G$ relative to $A$ is just the core $c(G)$. 

\begin{thm}[See Theorem \ref{lethm}] \label{lethmintro}
  Given $A\inc G$ and $A'\inc G'$, with $G,G'$ torsion-free hyperbolic groups,  let $C$ (resp.\ $C'$) be  a  core  of $G$ relative to $A$ (resp.\ of $G'$ relative to $A'$).
  
Then any partial elementary map $f :A\to G'$ with $f(A)=A'$ extends to an isomorphism $F$ between $C$ and $C'$.
\end{thm}

\begin{rem}
  Using ``Tarski'', one can see that  the converse is also true: if $f$ extends, it is partial elementary
  (see Section \ref{unsens}).
\end{rem}

This theorem is proved in 
  Section \ref{interp}, building on ideas due to Sela, Perin and Sklinos.

 The argument goes as follows  (see Subsection \ref{pfsensfac}).
Using $f$, one can view $G$ and $G'$ as $A$-groups, i.e.\ groups with a specified embedding of $A$;  morphisms between $A$-groups must commute with these embeddings. 
The goal is then to prove that $C$ and $C'$ are isomorphic as $A$-groups.
We introduce a notion of \emph{(strong) engulfing} between two  $A$-groups, which
can be translated into a first-order statement and
satisfies some kind of a  co-Hopf property: if two  $A$-groups engulf in each other, they are isomorphic   (to be precise, we have to define both engulfing and strong engulfing, see details in Subsection \ref{foc}).
The fact that $C$ is the core of $G$   implies that $C$ strongly engulfs in $G$.
Since strong engulfing is first-order and $f$ is a partial elementary map, $C$ strongly 
engulfs in $G'$.
Using the fact that $G'$ is a tower over $C'$, we deduce that $C$  
engulfs in $C'$.
By symmetry of the argument, $C$ and $C'$ engulf in each other and are therefore isomorphic.

 An important step in this proof consists in proving that the inclusion of $C$ into $G$ is a strong engulfing.
 This step is based on  the notion of preretraction introduced by Perin in \cite{Perin_elementary}.
  By way of contradiction, assume 
 that  $C\hookrightarrow G$ is not a strong engulfing. This  directly translates into the existence
of a non-injective preretraction $r:C\ra G$.
In Subsection \ref{prtor} we   give a rather short proof of Propositions 5.11 and 5.12  from \cite{Perin_elementary},
 which use $r$ to show that $C$ is an \'etage, 
thus contradicting the fact  that $C$ is a prototype.

\paragraph{Homogeneity.}
Theorem \ref{lethmintro} can be used to prove the homogeneity
of groups having very restricted tower structures.
Recall that a countable group $G$ is \emph{homogeneous} if, for every pair of tuples $\bar{a}=(a_1,\dots,a_n)$ and $\bar{b}=(b_1,\dots,b_n)$ in $G^n$
satisfying the same first-order formulas, there exists an automorphism of $G$ that sends $\bar{a}$ to $\bar{b}$.
Equivalently,  given $\bar{a}$ and $\bar{b}$ such that the map $a_i\mapsto b_i$ extends to a partial elementary map $f:\grp{\bar{a}}\ra 
 G$, 
there is an extension of $f$ to an automorphism of $G$.

 Given such an $f$, 
Theorem \ref{lethmintro}   says that $f$ extends to an isomorphism between the cores $C$ and $C'$ of $G$ relative to 
$\grp{\bar{a}}$ and $\grp{\bar{b}}$ respectively.
In particular, if $G$ is a prototype, then $C=C'=G$ and it immediately follows that $G$ is homogeneous.
 If $G$ is free, it  cannot be written as an étage of surface type (\cite{Perin_elementary,PeSk_homogeneity}, see also Corollary \ref{paset}), so the only possibilities for the  relative cores $C$ and $C'$ are free factors.
Homogeneity  of free groups  immediately follows. This argument generalizes naturally to free products of free groups and prototypes:
\begin{thm}[see Corollary \ref{cor_homogeneity}, extending \cite{PeSk_homogeneity,Houcine_homogeneity}]\label{thm_homogeneity_intro}
  If a torsion-free hyperbolic group
  $G=P_1*\dots*P_p*\bbF_r$ is a free product of one-ended prototypes  $P_i$ with a free group,
  then $G$ is homogeneous.  
\end{thm}

See \cite{DePe_homogeneity} for a characterization of homogeneous hyperbolic groups.

\paragraph{Primality, minimality, universality.}

A group is \emph{prime}  if it elementarily embeds into any other group elementarily equivalent to it. We show:

\begin{thm}[see Theorem \ref{prime}] \label{primeintro}
 A non-abelian torsion-free hyperbolic group is prime if and only if it is a one-ended prototype.
\end{thm}

A group is \emph{minimal} if it has no proper elementarily embedded subgroup. The classification of minimal torsion-free hyperbolic groups depends on the nature of their core. 

If $G$ is a one-ended prototype, it is minimal and any   minimal finitely generated
group elementarily equivalent to $G$ is isomorphic to $G$. On the other hand,
given two one-ended prototypes $B_1,B_2$ (so that $B_1*B_2$ is also a prototype), 
different choices of  $b_1,b_2$ in the groups $K(B_1,B_2,b_1,b_2)$ constructed above provide different  groups with  core $B_1*B_2$ which are  minimal and elementarily equivalent   to $B_1*B_2$.  More generally:

\begin{thm}[see Theorem \ref{minimal}]\label{minimalintro}
 Let $G$ be a non-abelian torsion-free hyperbolic group. 
\begin{itemize}

 \item If the core of $G$ is one-ended, then $G$ is minimal if and only if  $G$ is a prototype (i.e.\ $G$ is equal to its core).
 \item If the core of $G$ is infinitely-ended, there are infinitely many minimal hyperbolic groups elementarily equivalent to $G$.
\end{itemize}
\end{thm}

Theorem \ref{minimal} also  gives a classification of finitely generated minimal elementarily free groups.

We  
use Theorem \ref{lethmintro} to show that the class of all finitely generated groups elementarily equivalent to a  given torsion-free hyperbolic group is an {\em elementary Fra\"iss\'e class} (see \cite{KMS_fraisse} for this notion and \cite{Fraisse} for the original notion of a Fra\"iss\'e class). 
Using standard arguments, this directly implies:

\begin{thm}[see Theorem \ref{univers}] \label{universintro}
Given any  non-abelian torsion-free hyperbolic group $H$,  
there exists a unique countable  (infinitely generated) group $\calm_H$ elementarily equivalent to $H$ with the following properties:
\begin{enumerate}
\item $\calm_H$ is homogeneous;
\item  there exists an elementary chain of finitely generated subgroups $G_1\prec G_2\prec \dots\prec G_n\prec \cdots$
  such that $\calm_H=\bigcup_{n\geq 1}  G_n$;
\item  all torsion-free hyperbolic groups elementarily equivalent to $H$ elementarily embed into $\calm_H$.  
\end{enumerate}
\end{thm}

\paragraph{Towers are limit groups}

One of the purposes of this paper is to fill gaps in the literature. In an appendix we give a complete proof of the fact, due to Sela \cite{Sela_diophantine7}, that, if $G$ is a tower over a non-abelian torsion-free  hyperbolic group $H$, then $G$ is an $H$-limit group. 

The key step  (Theorem \ref{thm_courbes})  is to prove that, given a homomorphism $f$  from the fundamental group of a non-exceptional surface $\Sigma$ to $H$ which is injective on each boundary subgroup  and has non-abelian image, one may cut $\Sigma$ into subsurfaces such that the restriction of $f$ to the fundamental group of each subsurface is injective. 
One then uses a standard argument based on Baumslag's lemma \cite{Baumslag_residually} to conclude.

\paragraph{Acknowledgments} We thank Pablo Cubides, Fran\c{c}ois Dahmani, Olga Kharlampovich, Chlo\'e Perin, Zlil Sela, Karen Vogtmann for useful conversations.

\section{Preliminaries}

All groups $G$ considered in this paper will be finitely generated and torsion-free.    From Subsection \ref{stt} on,  $G$ will be assumed to be CSA. In Section \ref{elemequiv} we will start discussing model theory, and $G$ will be hyperbolic.

\subsection{Elementary equivalence and   elementarily embedded subgroups}

 We refer to \cite{TentZiegler, Chang_Keisler_model} for basic notions of model theory.
Two groups $G,G'$ are \emph{elementarily equivalent}, denoted $G\equiv G'$, if they satisfy the same first-order  sentences in the language of groups without constants:
such a sentence (in prenex form) consists  of a string
of quantifiers and variables followed by a boolean combination of word equations and inequations in these variables and their inverses.

For finite groups, and for finitely generated abelian groups (but not for finitely generated nilpotent groups), elementary equivalence is the same as isomorphism. 
On the other hand, all non-abelian free groups are elementarily equivalent (``Tarski's problem'').

A subgroup $H\inc G$  is   \emph{elementarily embedded}, denoted $H\preceq_e G$, if, for 
any first-order sentence $\phi$ in the language of groups with constants in $H$, one has $H\models\phi$ if and only if $G\models \phi$.
Equivalently, this means that, for
every first-order formula $\phi(x_1, \ldots, x_n)$ (in the language of groups, without constants) with free variables $x_1,\dots,x_n$, 
and every $n$-uple   $\bar{a}=(a_1,\dots,a_n)$ 
from $H$, we have $H\models \phi(\bar{a})$ if and only if $G\models\phi(\bar{a})$. 
 In model theory, one usually says that $H$ is an elementary subgroup of $G$, but we will avoid this terminology as it may be confusing
in the context of hyperbolic groups.

Intuitively, a subgroup of 
a group is elementarily embedded if the first-order properties of its tuples do not change if we see them 
as tuples of the ambient group   rather than  of the subgroup. Note that, if $H$ is an elementarily embedded subgroup of $G$, then $G$ and $H$ are 
elementarily equivalent, as sentences without constants are in particular sentences with constants in $H$.

\begin{example} 
\label{pasel0}
  The following are   examples of subgroups $H\inc G$ which are 
elementarily equivalent to $G$ but are not elementarily embedded in $G$. 

\begin{itemize}
\item The subgroup of even integers $2\Z$ of the group of integers $\Z$ is not elementarily embedded: the element $2$ 
has a root in $\Z$, but not in $2\Z$.
\item  
Let $\pi_1(S_g)=\langle e_1, \ldots, e_{2g} \mid [e_1,e_2]\cdots [e_{2g-1}, e_{2g}]=1
\rangle$ be
 the fundamental 
group of the closed orientable surface of genus $g$,   with $g\ge2$.  Let $F=\grp{e_1,e_2,e_3,e_4}$. 
 It is a free group  of rank $4$ if $g\ge3$.
 If $g=3$, then  $F$ 
  is not elementarily embedded   because $[e_1,e_2][e_3,e_4]$ is   a  single commutator $[e_6,e_5]$     in $\pi_1(S_g)$ but not in $F$. On the other hand, $F$ is elementarily embedded if $g\ge4$ (see  \cite{Perin_elementary}).
\end{itemize}
\end{example}

The following facts  are easy consequences of 
the definition.

\begin{rem}\label{tm}
If $H\preceq_eG$ and $K$ is conjugate to $H$, then $K\preceq_eG$.

  If $K\preceq_eH$ and $H\preceq_eG$, then $K\preceq_eG$.  

  If $K\preceq_eG$ and $H\preceq_eG$ with $K\subset H$, then $K\preceq_eH$.  
\end{rem}

\subsection{Surfaces}\label{sec_surfaces}

  We denote by $S_{g,b}$ and $N_{g,b}$  respectively the orientable and non-orientable compact surfaces of genus $g$ with $b$ boundary components (we just write $S_g$ or $N_g$ when
there is no boundary).

We will denote by $Q=\pi_1(\Sigma)$ the fundamental group of a compact (possibly non-orientable) surface $\Sigma$. Subgroups of $\pi_1(\Sigma)$ contained up to conjugacy in  the fundamental group of a boundary component of $\Sigma$ will be called \emph{boundary subgroups}.     The maximal boundary subgroups are the fundamental groups of boundary components.

Recall that the Euler characteristic  $\chi(\Sigma) $ equals $2-2g-b$ if $\Sigma$ is orientable, $2-g-b$ if it is not orientable (with $g$ the genus and $b$ the number of boundary components).

We always assume that $Q$ is   not virtually abelian (equivalently, 
$\chi(\Sigma) $ is negative). This rules out the non-hyperbolic surfaces:   disc,   annulus,  M\"obius band,   sphere,    projective plane, torus, Klein bottle. 

\begin{dfn}[Exceptional surfaces] \label{except}
Four  hyperbolic surfaces with $\chi(\Sigma)=-1$ are considered \emph{exceptional} because their mapping class group is ``too small'' (they do not carry pseudo-Anosov diffeomorphisms):  the pair of pants (thrice-punctured sphere) 
  $S_{0,3}$, the twice-punctured projective plane $N_{1,2}$, the once-punctured Klein bottle $N_{2,1}$,  and   the closed non-orientable surface $N_3$ of genus 3 (the once-punctured torus  $S_{1,1}$ is   contained in $N_3$ but is  not exceptional).
\end{dfn}

\begin{figure}[ht!]
  \centering
  \includegraphics{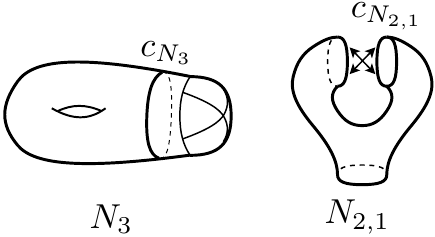} 
  \caption{The   exceptional surfaces with infinite mapping class group, and the corresponding special curves}
  \label{fig_except}
\end{figure}

Two  
exceptional surfaces $\Sigma$ have   infinite mapping class group: 
the once-punctured Klein bottle $N_{2,1}$ and $N_3$, the closed non-orientable surface of genus 3.
Both of them have a \emph{special simple closed curve} $c_\Sigma$ whose isotopy class is invariant under the mapping class group, see Figure \ref{fig_except}.

The punctured Klein bottle $N_{2,1}$ has a unique isotopy class $c_{N_{2,1}}$ of non-separating two-sided simple closed curves whose complement is orientable (hence a pair of pants) \cite[Prop.~A.3]{Stukow_dehn}.

 In the case of $N_3$, the special curve $c_{N_3}$ is the unique (up to isotopy) one-sided simple closed curve    whose complement is orientable \cite[Prop.\ 2.1]{GoMa_homeotopy}. Cutting $N_3$ along  the boundary of a regular neighbourhood of $c_{N_3}$ decomposes $N_3$ into
a once-punctured torus and a M\"obius band, and represents $\pi_1(N_3)$ as an amalgam over 
$\grp{g_{N_3}^2}$ 
(we denote by $g_{N_3}$ a generator of $\pi_1(c_{N_3})$, it is 
well defined up to inversion and conjugacy).

\begin{dfn} [Pinched curve, pinching map] \label{pinc}
Let $p$ be a homomorphism from  $\pi_1(\Sigma)$ to a group. 

 A \emph{pinched curve} is a 2-sided simple closed curve which  is   not null-homotopic  and 
whose fundamental group is   contained in $\ker p$. A   family of pinched curves is a family of disjoint, pairwise non-parallel, pinched curves.

The map $p$ is \emph{pinching} if there is a pinched curve, \emph{non-pinching} if there is none.
\end{dfn}

\subsection{Splittings and trees}
\label{sandt}

Graphs of groups are a generalization of free products with amalgamation and HNN-extensions. 
Recall that a graph of groups   consists of a graph $\Gamma$, the assignment of a group $G_v$ or $G_e$ to each vertex $v$ or non-oriented edge $e$, 
and injections $G_e\to G_v$ whenever $v$ is the origin of an oriented edge (see \cite {Serre_arbres} for precise definitions); the image of $G_e$ is called an incident edge group at $v$. 
This data yields a group, the fundamental group of the graph of groups, with an action   on a simplicial tree $T$ (the Bass-Serre tree).  We usually call the graph of groups $\Gamma$, with the extra data implicit. 

A graph of groups decomposition, or splitting, of a group $G$ is an isomorphism of $G$ with the fundamental group of a graph of groups $\Gamma$. 
The vertex and edge groups $G_v$ and $G_e$ are then naturally viewed as subgroups of $G$. We always assume that $G$ is finitely generated 
and that the graph of groups is minimal (no proper connected subgraph of $\Gamma$ carries the whole of $G$); it follows that $\Gamma$ is a finite graph. 
We allow the trivial splitting, when $\Gamma$ is a single vertex $v$ with $G_v=G$.

A splitting of $G$ yields an action of $G$ on the Bass-Serre tree $T$.   Conversely, an action of $G$ on a tree $T$ defines a quotient graph of groups $\Gamma=T/G$.

We denote by $G_e$ (resp.\ $G_v$) the stabilizer of an edge $e$ of $T$ (resp.\ a vertex $v$); it is conjugate to an edge (resp.\ vertex) group of $\Gamma$.
  We say that  a vertex $v$ of   $\Gamma$  or   $T$   is \emph{cyclic} if $G_v$ is cyclic. 

All splittings considered in this paper will be cyclic splittings: edge groups are cyclic (possibly trivial).  We allow redundant vertices: $\Gamma$ may have a   vertex of valence 2 with both incident edge groups equal to the vertex group;  this   corresponds to vertices of valence 2 in $T$.

A subgroup of $G$ is \emph{elliptic} if it fixes a point in $T$ (equivalently, it is contained in a conjugate of a vertex group of $\Gamma$). The splitting $\Gamma$ and the tree $T$ are \emph{relative to subgroups} $H_1,\dots,H_p$ if each $H_i$ is elliptic. 

If $T,T'$ are trees with an action of $G$, we say that $T$ \emph{dominates} $T'$ (or that $\Gamma=T/G$ dominates $\Gamma'=T'/G$)  if there is a 
$G$-equivariant map from $T$ to $T'$, equivalently if every subgroup which is elliptic in $T$ is also elliptic in $T'$. Two trees (or two splittings) belong to the same \emph{deformation space} if each one dominates the other (equivalently, if they have the same elliptic subgroups).

In particular, $T$ dominates all trees $T'$ obtained by collapsing each edge of $T$ belonging to a given   $G$-invariant set to a point. Such a $T'$ is called a \emph{collapse} of $T$; one passes from $\Gamma=T/G$ to $\Gamma'=T'/G$ by collapsing certain edges and redefining the vertex groups. Conversely, we say that $T$ is a refinement of $T'$. One obtains $\Gamma$ from $\Gamma'$ by blowing up certain vertices $v$ of $\Gamma'$ using a splitting of $G_v$   relative to the incident edge groups.

The (closed) \emph{star} of a vertex $v$ in a graph or a tree is the union of all edges containing $v$. The open star of $v$ is obtained by removing  all vertices $w\ne v$ from the closed star. 

Given an action of  $G$   on  a tree $T$, we say that $T$ is \emph{1-acylindrical near a vertex} $v $ if  $G_e\cap G_{e'}=\{1\}$ whenever $e,e'$ are distinct edges containing $v$.   It is \emph{2-acylindrical} if all segments of length 3 have trivial stabilizer.

\subsection{Grushko decompositions.}\label{gd}

A    finitely generated group $G$ is   either trivial, or  one-ended, or isomorphic to $\Z$, or it has infinitely many ends (equivalently, it is a non-trivial free product); recall that all groups are assumed to be torsion-free.

Any finitely generated group $G$ has a \emph{Grushko decomposition}: it is isomorphic to   a free product $A_1*\cdots*A_n*\mathbb{F}_m$, where each $A_i$ 
is   one-ended 
and $\mathbb{F}_m$ is a free group of rank $m$. Moreover, the numbers $n$ and $m$ are unique, as well as the $A_i$'s up to conjugacy (in $G$) and permutation. 
 We call the $A_i$'s the \emph{Grushko factors} of $G$. 
 
 Decompositions as free products correspond to graph of groups decompositions with trivial edge groups (splittings over the trivial group). If $G$ has
 infinitely many ends, it has 
a 
Grushko decomposition $G=A_1*\cdots*A_n*\mathbb{F}_m$ as above. It may be viewed as the fundamental group 
of a graph of groups with one central vertex $v$ carrying a trivial group, joined to vertices $v_1,\dots, v_n$ carrying $A_1,\dots, A_ n$ by a single edge, with $m$ loops attached to $v$.

If $\calh=\{H_1,\dots,H_p\}$ is a family of non-trivial subgroups of $G$,  we say that $G$ is 
one-ended relative to $\calh$ if it cannot be written as a non-trivial free product with each $H_i$ contained in a  conjugate of one of the factors. 
If $G$ is not one-ended relative to $\calh$, it has a non-trivial
Grushko decomposition relative to $\calh$; it is of the form $G=A_1*\cdots*A_n*\mathbb{F}_m$, with   each $H_i$ contained in some $A_j$ (up to conjugacy), and 
$A_j$ 
one-ended relative to the conjugates of the $H_i$'s which it contains   ($A_j$   free  
  is allowed only 
  if it   contains a conjugate of an $H_i$).
 
\subsection{Surface-type vertices}

Let $\Gamma$ be a splitting of $G$, with Bass-Serre tree $T$.

\begin{dfn} [Surface-type vertex]\label{stype}
A vertex $v$ of $\Gamma$  
is called a \emph{surface-type vertex} if the following conditions hold:
\begin{itemize}
 \item the group $G_v$ carried by $v$ is the fundamental group of a compact   surface $\Sigma$ 	  (usually with boundary, possibly non-orientable),   with   $\chi(\Sigma)<0$;
 
  \item incident edge groups are maximal  boundary subgroups of $\pi_1(\Sigma)$, and this induces a bijection
 between    the set of incident edges and the set of boundary components of $\Sigma$.
\end{itemize}

The vertex $v$ is \emph{exceptional} if $\Sigma$ is exceptional    (see Definition \ref{except}). 

We say that $G_v=\pi_1(\Sigma)$, which we often call $Q$, is a surface-type vertex group, and that $\Sigma$ is a surface of $\Gamma$. If $u$ is any lift of $v$ to 
  $T$, we say that  $u$  
is a surface-type vertex, and its stabilizer $G_u$ (which is conjugate to $Q$) is a surface-type vertex stabilizer.
 \end{dfn}
 
 If $G$ is the fundamental group of a closed hyperbolic surface, it is a surface-type vertex group (with $\Gamma$ the trivial splitting).
 
  Surface-type vertices are QH (quadratically hanging) vertices,  see
  \cite{GL_JSJ}.    
The tree $T$ is 1-acylindrical near any surface-type vertex $u$: if $e$ and $e'$ are distinct edges containing $u$, then $G_e\cap G_{e'}=\{1\}$.

 We will need two standard lemmas about surface groups.

\begin{lem}[compare  Lemma 1.4 of \cite{Sela_diophantine4} and Lemma 3.11 of \cite{Perin_elementary}] \label{Scott}
Let $Q=\pi_1(\Sigma)$ be a surface-type vertex group of a splitting of $G$. Let $J\inc G$ be a one-ended finitely generated group. If $J\cap Q$ is non-trivial and has infinite index in $Q$, then $J\cap Q$ is 
a boundary subgroup of $Q=\pi_1(\Sigma)$, in particular $J\cap Q\simeq\Z$.
\end{lem}

\begin{proof}[Sketch of proof] 
Assume that $J\cap Q$ has infinite index in $Q$, is non-trivial, not a 
boundary subgroup. 
By Scott's Theorem \cite{Scott_almost}, there is a finite covering $\Sigma_0$ of $\Sigma$ such that $J\cap Q$ is the fundamental group of a proper subsurface $\Sigma_A\inc \Sigma_0$. This subsurface contains a boundary component of $\Sigma_0$ because $J$, being one-ended, must have non-trivial intersection with some boundary subgroup of $Q$. It also has a boundary component $C$ which is not parallel to a component of $\partial \Sigma_0$. An essential arc going from $C$ to $C$ within $\Sigma_A$ yields a free splitting of $J$, contradicting one-endedness.
\end{proof}

\begin{lem}\label{lem_pince} 
Let $\Sigma$ be a compact surface. Let $p:\pi_1(\Sigma)\to G_1*G_2$ be a homomorphism to a non-trivial free product such that the image of every boundary subgroup of $\pi_1(\Sigma)$ is contained in a conjugate of $G_1$ or $G_2$, with at least one image contained in a conjugate of $G_1$. If $p$ is non-pinching (Definition \ref{pinc}), its image is contained in a conjugate of $G_1$.
\end{lem}

\begin{proof}[Sketch of proof]
If the conclusion is false, consider a $p$-equivariant map $f$ from the universal covering $\tilde \Sigma$ of $\Sigma$ to the Bass-Serre tree of the amalgam $G_1*G_2$,
 with $f$ mapping each boundary component of $\tilde \Sigma$ to a vertex of the tree. The projection to $\Sigma$ of the preimage of   midpoints of edges contains a pinched curve.
\end{proof}

 \subsection{The cyclic JSJ decomposition of   hyperbolic groups and CSA groups }
 \label{JSJ} \label{sec_modified}

Recall that  a finitely generated group $G$ is hyperbolic if its  Cayley graph (with respect to some, hence any, finite generating set) 
is a hyperbolic metric space: there exists $\delta>0$ such that any point on one side of any geodesic triangle is $\delta$-close to one of the other two sides (see \cite{BH_metric,GhHa}).  
Small cancellation groups and fundamental groups of negatively curved closed manifolds are hyperbolic.
In particular
 free groups, fundamental groups of closed surfaces with negative Euler characteristic, and free products of such groups, are hyperbolic. 
 
 Any torsion-free hyperbolic group    is Hopfian \cite{Sela_Hopf}: it cannot be isomorphic to a proper quotient. 
Any one-ended torsion-free hyperbolic group   is co-Hopfian \cite{Sela_structure}: it cannot be isomorphic to a proper subgroup.

A one-ended torsion-free hyperbolic group $G$ 
does not split over a trivial group, but it may have  splittings with edge groups isomorphic to  $\Z$
and it
has a canonical JSJ decomposition $\Gcan$ 
over cyclic groups \cite{Sela_structure, Bo_cut, GL_JSJ}. 
We mention the features  that will be important for this paper. Similar considerations apply to relative JSJ decompositions,   see \cite{GL_JSJ}.

  A subgroup $H\inc G$ is \emph{universally elliptic} if it is elliptic in every cyclic splitting of $G$.  A tree is universally elliptic if its edge stabilizers are. A JSJ tree is a universally elliptic tree which dominates every universally elliptic tree. In general, only the deformation space of JSJ trees is well-defined, but in the case of one-ended hyperbolic groups there is a canonical JSJ tree $\Tcan$ (see \cite{GL_JSJ}); it is the tree of cylinders of any JSJ tree. It may also be constructed from the topology of the Gromov boundary $\bo G$ of $G$ \cite{Bo_cut}.  It is invariant under automorphisms of $G$.

All edge groups of the graph of groups $\Gcan=\Tcan/G$ are infinite cyclic.  
The graph $\Gcan$ is bipartite, with every edge joining a cyclic vertex (i.e.\ a vertex carrying a cyclic group) to a non-cyclic vertex. 
Non-cyclic vertex groups are either rigid (universally  elliptic)
or surface-type (in the sense of Definition \ref{stype}).
The action of $G$ on the associated Bass-Serre tree $\Tcan$ is 
1-acylindrical near every non-cyclic vertex.

The tree   $\Tcan$
 is universally   compatible: if $T$ is any tree with cyclic stabilizers on which $G$ acts, $\Tcan$ and $T$ have a common refinement $\hat T$ (their lcm, see \cite{GL_JSJ} Section A.5): one may obtain each of the trees $\Tcan$ and $T$ from $\hat T$ by collapsing each edge in some $G$-invariant set to a point.  

More generally, recall that a group is CSA if maximal abelian subgroups are malnormal. 
  Equivalently, a group is CSA if it is commutative transitive and,
whenever $[tat\m ,a]=1$ for some $a\neq 1$, then $t$ commutes with $a$. In particular, torsion-free hyperbolic groups are   CSA.

A 
one-ended torsion-free CSA group has a canonical cyclic JSJ decomposition relative to non-cyclic abelian subgroups  (see Theorem 9.5 of \cite{GL_JSJ}).  

  It has properties similar to the hyperbolic case, but since we consider the cyclic decomposition rather than the abelian one it is not bipartite.
  Each edge  joins an abelian vertex (i.e.\ a vertex carrying an abelian group)  to   a non-abelian vertex, or two   non-abelian vertices. Non-abelian vertex groups are rigid or surface-type,  and no edge joins two surface-type vertices. 
  The Bass-Serre tree is 2-acylindrical.

\begin{figure}[ht!]
  \centering
  \includegraphics{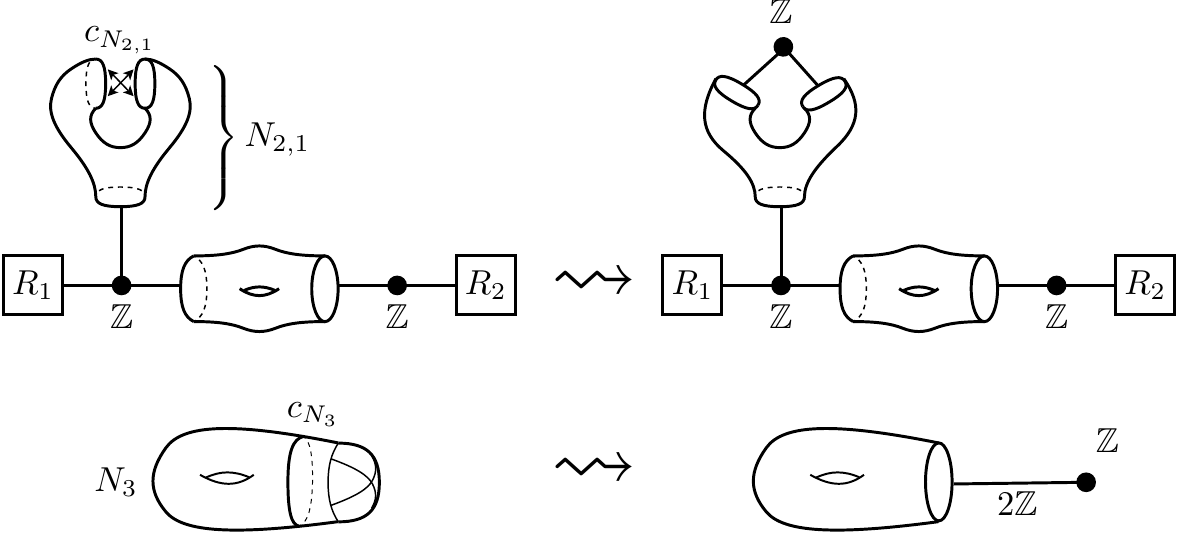}
  \caption{The modified JSJ decomposition $\tGcan$  (with $R_1,R_2$ rigid vertex groups)}\label{fig_JSJ_modifie} 
\end{figure}
 
\paragraph{The modified JSJ decomposition}
 In Section \ref{interp} we will need to modify $\Gcan$ so that all exceptional surfaces
 appearing at surface-type vertices  have finite mapping class group,
 while preserving the 
invariance under automorphisms (see the proof of Lemma \ref{shortening}). 

Recall that the exceptional surfaces with infinite mapping class group are $N_{2,1}$ (the once-punctured Klein bottle) and the closed surface $N_3$ 
  (see Subsection \ref{sec_surfaces}). As explained in \cite[Remark 9.32]{GL_JSJ}, if   $\Gcan$ has vertices which are surface-type with associated surface $N_{2,1}$, one may cut the surfaces along their special curve  $c_{N_{2,1}}$ 
  and thus obtain a splitting $\tGcan$  refining $\Gcan$ (see Figure \ref{fig_JSJ_modifie}).
Since the homotopy class of $c_{N_{2,1}}$ is invariant under the mapping class group of $N_{2,1}$, the splitting $\tGcan$ is still invariant under
the group of automorphisms of $G$.  Similarly, if $G\simeq\pi_1(N_3)$,  we replace $\Gcan$ (which is the trivial splitting) by the splitting dual to the decomposition of $N_3$ into a once-punctured torus and a M\"obius band.

\subsection{Two lemmas about splittings}
The following lemma is folklore (see the proof of  Proposition 4.7 in \cite{GLS_finite_index}).

\begin{lem} 
\label{extid}
Let $A\inc G$, and let $G_v$ be a vertex group of a splitting $\Gamma$ of $A$. Let $p:G_v\to G$ be a homomorphism such that the  restriction of $p$ to any incident edge  group  $G_e\inc G_v$ agrees with conjugation by some element of $G$ (depending on $e$). Then $p$ may be extended ``by the identity'' to $\hat p: A \to G$. The extension $\hat p$ agrees with a conjugation on each vertex group of $\Gamma$ other than $G_v$. \qed
\end{lem}

 The next  lemma will be used in Subsection \ref{prtor}.
It
generalizes the injectivity part of Proposition 6.1 of \cite{Perin_elementary}. The argument is the same as in \cite{Perin_elementary}, we give it for completeness.

\begin{lem}  
  \label{sixunm}
  Let $A$ and $G$ be groups acting on trees    $T_A$, $T_G$  with all edge stabilizers  isomorphic to $\Z$.
  Assume that these trees are bipartite, with vertices of type 0 or 1, and 1-acylindricity holds near vertices of type 1. 
  
  Let $f:A\to G$ be a homomorphism sending any vertex stabilizer of type 0 of $T_A$ injectively into a vertex stabilizer of type 0 of $T_G$, and any vertex stabilizer of type 1   bijectively to a vertex stabilizer of type 1.
  
  If $f$ is not injective, there exist   two non-conjugate vertex stabilizers of type 1 of  $T_A$ with the same image. 
\end{lem}

\begin{proof} 
  By acylindricity, vertex groups of type 1 are not $\Z$ so fix a unique point. Vertex groups of type 0   fix a unique  vertex of type 0 (and possibly vertices of type 1). 
  We denote by $A_v$ the stabilizer of a vertex $v$ of $T_A$.

Given a vertex $v$  of $T_A$, let $\varphi(v)$ be the unique vertex of $T_G$ which has the same type as $v$ and is fixed by   $f(A_v)
$. 
By 1-acylindricity, $\varphi$ maps adjacent vertices   to adjacent vertices.  Assume that  $f$ is not injective. There must then be distinct edges $vu_1$ and $vu_2$ of $T_A$ with $\varphi(u_1)=\varphi(u_2)$. By $1$-acylindricity   and injectivity of $f$ on vertex groups,
 $v$ is of type 0 (and thus $u_1,u_2$ are of type 1). The lemma is proved if $u_1$ and $u_2$ are in different orbits, so assume $u_2=au_1$ with $a\in A$.

Since $f(a)$ fixes $\varphi(u_1)$, which is of type 1, there exists $a'\in A_{u_1}$ such that $f(a')=f(a)$. Replacing $a$ by $aa'{\m}$ lets us assume $u_2=au_1$ with $a\in \ker f$. 
This implies $av\ne v$. 

Let $a_1,a_2$ be non-trivial elements  of $A$ fixing the edges $vu_1$ and $vu_2$ respectively. The element $aa_1a\m$ fixes the edge joining $av$ and $u_2$, so  $\langle aa_1a\m,a_2\rangle$ is not cyclic by 1-acylindricity near $u_2$. On the other hand, $f(aa_1a\m)=f(a_1)$ and $f(a_2)$ both fix the edge joining $\varphi(v)$ to $\varphi(u_1)=\varphi(u_2)$, so generate a cyclic subgroup of $G$. This contradicts the fact that $f$ is injective on $A_{u_2}$.
\end{proof}

\begin{cor}\label{sixunm2}
  Let $G$ be a non-cyclic torsion-free hyperbolic group acting on a tree  $T$  with all edge stabilizers $\Z$. 

  Let $f:G\to G$ be a homomorphism sending any   non-cyclic vertex stabilizer and any edge stabilizer bijectively to a conjugate.
Then $f$ is injective.
\end{cor}

\begin{proof}
We apply the lemma  with $A=G$ and $T_A,T_G$ both equal to
  the tree of cylinders of $T$ (see \cite{GL4} or \cite[Section 7]{GL_JSJ}).
It is bipartite: stabilizers of vertices of type 0 are maximal cyclic and commensurable to an edge stabilizer of $T$,
stabilizers of vertices of type 1 are the non-cyclic vertex stabilizers of $T$.
  The hypotheses of Lemma \ref{sixunm} are satisfied.
\end{proof}

\section{Towers
} \label{pr}

In the beginning of this section, $G$ will be an arbitrary torsion-free finitely generated group. In Proposition \ref{equivsimple}, we will see that subgroups $B<G$ having 
 non-trivial elements $x,y$ such that $x$ commutes with all conjugates of $y$ 
play a special role, and   from then on we shall assume that $G$ is CSA to ensure that such subgroups $B$ are  abelian. 
This is mostly a matter of convenience, as it makes statements simpler. In Subsections \ref{readjsj} and \ref{confl}, however, we will use the canonical cyclic JSJ decomposition of $G$, and this really requires $G$ to be CSA.

\begin{figure}[!ht] 
  \centering
  \includegraphics{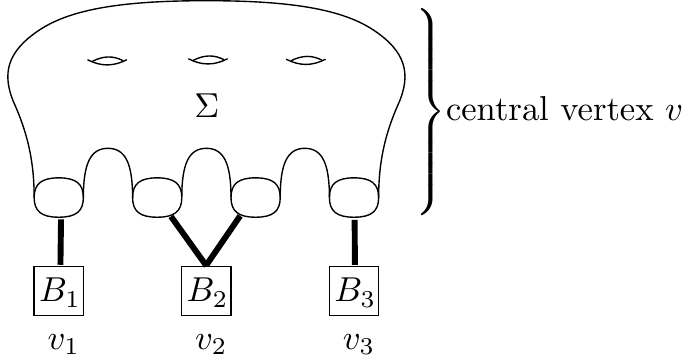}
  \caption{A centered splitting of $G$ with three bottom groups $B_1,B_2,B_3$. Its base is the free product $B_1*B_2*B_3$.}
  \label{fig_centered}
\end{figure}

\subsection{Centered splittings}

\begin{dfn}[Centered splitting] \label{centspl}
A \emph{centered splitting} of $G$ is a  graph of groups decomposition $G=\pi_1(\Gamma)$ such that the vertices of $\Gamma$
are $v,v_1,\dots,v_n$, with $n\ge1$, 
  where $v$ is surface-type 
and every edge joins $v$ to some $v_i$    (see Figure \ref{fig_centered}).  Note that all edge groups are isomorphic to $ \Z$.

The vertex $v$ is called the \emph{central vertex} of $\Gamma$,  the group carried by $v$
is denoted by $Q=\pi_1(\Sigma)$, and $\Sigma$ is the surface associated to $\Gamma$.
The vertices $v_1,\dots,v_n$ are the \emph{bottom vertices}, and we denote by $B_i$ 
 the group carried by $v_i$.   We call $B_i$ a bottom group, and $Q$ is the central group.

The \emph{base} of $\Gamma$ is the abstract free product  $B_\Gamma=B_1*\dots* B_n$.

 The centered splitting $\Gamma$ is \emph{simple} if $n=1$, \emph{non-exceptional} if the surface $\Sigma$ is non-exceptional.
 
 If $\Gamma$ is simple and $B_1\simeq \Z$, we say that $G$ is a \emph{parachute}  (parachutes are studied in Propositions 6.7 and 6.8 of \cite{GLS_finite_index} and in Subsection \ref{parac}).
\end{dfn}

\begin{rem}
   Each $B_i$ is well-defined up to conjugacy in $G$,   but  if $n\ge2$ the subgroup $\grp{B_1,\dots,B_n}$ itself is not well-defined up to conjugacy
(it depends on choices of individual conjugates), and it does not have to be isomorphic to $B_\Gamma$ in general.
See Proposition \ref{equivpassimple}, where particular choices of conjugates are made.
\end{rem}

\begin{rem}
   If $G$ is hyperbolic, the groups $B_i$ are quasiconvex, hence hyperbolic.
\end{rem}

Like all splittings, $\Gamma$ is assumed to be minimal: if $v_i$ belongs to only one edge, the group carried by this edge is a proper subgroup of $B_i$.
  As in every surface-type vertex,   each conjugacy class of boundary subgroups  of $Q$ is carried by  exactly one incident edge 
(all boundary components are used).
We allow redundant vertices: $\Gamma$ may have a cyclic vertex $v_i$ of valence 2 with both incident edge groups equal to the vertex group.

Given any  non-trivial splitting with a surface-type vertex $v$, collapsing all edges not containing $v$ (and subdividing if there is a loop at $v$) yields a centered splitting. 

\begin{rem} \label{malnor} 
By 1-acylindricity near $v$, the family $(B_1,\dots,B_n)$ is malnormal: if $gB_i g\m\cap B_j\neq \{1\}$, then $i=j$ and $g\in B_i$.
In particular, each $B_i$ is malnormal and, if two elements of $B_i$ are conjugate in $G$,   then they are conjugate in $B_i$.
\end{rem}

 See Subsection \ref{sec_grushko} for a few general facts about centered splittings.

\subsection{Boundary-preserving maps, 
pinched quotient}

\begin{dfn}[Boundary-preserving map] \label{locpret}
Let $G$ be a group, and   $Q\simeq\pi_1(\Sigma)$ a subgroup  identified with  the
fundamental group of a compact surface $ \Sigma$.

  A \emph{\bpm}
  (with values in $G$) is a homomorphism $p:Q\to G$ such that, on each boundary subgroup of $\pi_1(\Sigma)$, the map $p$ agrees with the conjugation by some element of $G$.

The \bpm\ is \emph{non-degenerate} if additionally:
\begin{itemize}
\item the image of $p$ is  non-abelian; 
\item   $p$ is not an isomorphism of $Q$ onto some conjugate of $Q$;
\item the surface $\Sigma$ is non-exceptional.
\end{itemize}

If $Q$ is the group carried by a
surface-type vertex $v$ of a splitting $G_1=\pi_1(\Gamma)$ of a subgroup $G_1<G$,
we say that the \bpm\ is \emph{associated to} $\Gamma$, or to $v$  (most often, $G_1$ will be  $G$ or a free factor of $G$).

\end{dfn}

\begin{rem}\label{pasexp}
The first two additional conditions ensure that $p$ carries meaningful information; in particular, it is not the identity  and does not factor through an abelian group.   \Ndbpm s were called   local preretractions in \cite{GLS_finite_index}.

The non-exceptionality assumption is not needed   until Section \ref{elemequiv},
but it is essential for applications to model theory: in particular, ``Tarski'' (Theorem 
 \ref{Tarski})  would not hold without it.
 \end{rem}

\begin{figure}[!ht] 
  \centering
  \includegraphics{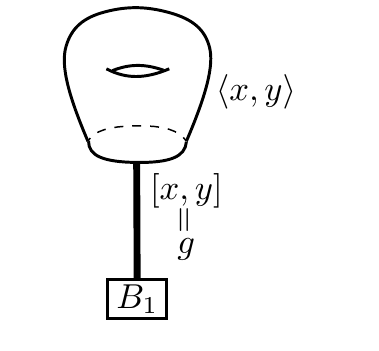}
  \caption{A once-punctured torus glued on  
some non-trivial element   $g\in B_1$.}
\label{torus}
\end{figure}
 
\begin{example}[See Figure \ref{torus}] \label{expt}
Let $G=B_1*_{g=[x,y]}F(x,y)$ be a cyclic amalgam of a group $B_1$ with the free group $F(x,y)$, with $g$ some non-trivial element of $B_1$. The amalgam defines a simple centered splitting $\Gamma$ of $G$, with $\Sigma$ a punctured torus. 

If $g$ is a commutator $[b,b']$ with $b,b'\in B_1$, sending $x$ to $b$ and $y$ to $b'$ defines a \ndbpm\  $p$ associated to $\Gamma$ (Lemma \ref{lemcle4} below implies that there is no such $p$ if $g$ is not a commutator in $B_1$).
\end{example}

\begin{figure}[!ht] 
  \centering
  \includegraphics{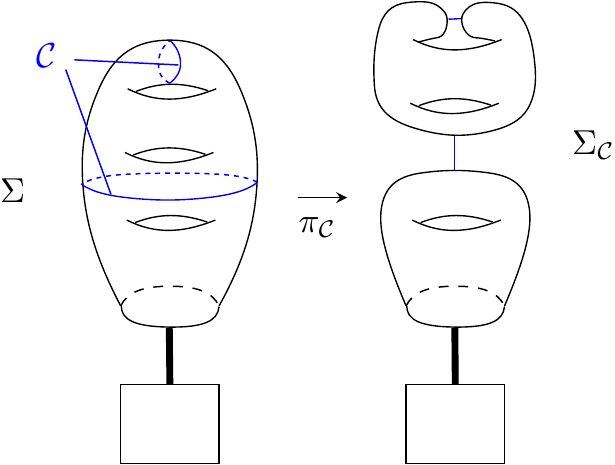}
  \caption{A pinched quotient.}
  \label{fig_pinch}
\end{figure}

Recall (Definition \ref {pinc}) that a curve $C$ is pinched by $p$ if $\pi_1(C)\inc\ker p$.
 In the following definition, we will often take $G_1=G$, or $G_1$ a Grushko factor of $G$.

\begin{dfn}[Pinched quotient, see Figure \ref{fig_pinch} and Definition 5.26 of  \cite{Perin_elementary}] \label{pc}
 Given     a \bpm\ $p:\pi_1(\Sigma)\to G$, let  $\calc$ be a 
 family of pinched curves on $\Sigma$.
  If $G_1$ is a subgroup of $G$ containing $\pi_1(\Sigma)$, the \emph{pinched quotient} of $G_1$ is the quotient $(G_1)_{\calc}$ of $G_1$ obtained by killing all curves in $\calc$.
It can be described as follows.

 For each curve $C\in \calc$, we cut $\Sigma$ along $C$, glue a disc to each of the two boundary curves thus created, and join these two discs by an arc. We get a space $\Sigma_\calc$ consisting of surfaces $\Sigma_\calc^j$ 
 and arcs; these surfaces are   closed or bounded by components of $\partial\Sigma$. 
 
 The surfaces/arcs of $\Sigma_\calc$ correspond to vertices/edges of a graph of groups decomposition $\Delta_\calc$ of $\pi_1(\Sigma_\calc)$ with trivial edge groups.   There is a quotient map   $\pi_\calc: \pi_1(\Sigma)\to \pi_1(\Sigma_\calc)$. 
 The map $p$  factors through $\pi_\calc$ and induces a   map     $p_\calc:  \pi_1(\Sigma_\calc)\to G$.
  
 Now suppose that $p$ is associated to a splitting $G_1=\pi_1(\Gamma)$ of a subgroup $G_1\inc G$.
 Since  $p$ is boundary-preserving, no component of $\partial\Sigma$ is pinched, and we can define a graph of groups $\Gamma_\calc$ by blowing up the vertex $v$ of $\Gamma$ carrying $\Sigma$ using  $\Delta_\calc$. 

 The fundamental group of $\Gamma_\calc$ is isomorphic to the pinched quotient $(G_1)_\calc$.
The quotient map from $G_1$ to $(G_1)_\calc$ is an extension of $\pi_\calc: \pi_1(\Sigma)\to \pi_1(\Sigma_\calc)$ and we also denote it by $\pi_\calc$.   It is injective on vertex groups of $\Gamma$ other than $\pi_1(\Sigma)$.
 \end{dfn}

\subsection{A key lemma}
 \label{kl}
The following key lemma will be used  extensively to study \bpm{}s. 

 \begin{lem} [{\cite[Lemma 3.3]{GLS_finite_index}}]
 \label{lemcle4} 
Let   $p:\pi_1(\Sigma)\to G$ be a \ndbpm\ associated to a centered splitting  $\Gamma$ of $G$.
 Let  $\calc$ be a maximal family of pinched curves on $\Sigma$, and let $\Sigma'$ be a component of the surface obtained by cutting $\Sigma$ along $\calc$.  

 Viewing $\pi_1(\Sigma')$ as a subgroup of $\pi_1(\Sigma)$,
the image of 
 $\pi_1(\Sigma') $ by $p$ is contained in a conjugate of a bottom vertex group $B_i$  of $\Gamma$. 
 \qed
\end{lem}

\begin{rem} \label{vi}
By 1-acylindricity, the lemma implies that all edges of $\Gamma$ associated to 
 boundary components of $\Sigma$ contained in $\Sigma'$  join   the   central vertex to the same bottom vertex, namely the vertex $v_i$ carrying $B_i$.
\end{rem}

The following strengthening of Lemma \ref{lemcle4} will be used later.    Recall that a  subsurface is incompressible if the inclusion induces an injection on fundamental groups.

\begin{lem}[{\cite[Remark 3.4]{GLS_finite_index}}] \label{lemcle2}
Let $\Gamma$ be a centered splitting of $G$. We denote by   $\Sigma$ its surface.
Let $S$ be another compact surface, and 
let $p:\pi_1(S)\to G$ be  a homomorphism such that the image of each boundary subgroup of $\pi_1(S)$ 
is contained in a bottom group of $\Gamma$ (up to conjugacy).

Let  $\calc$ be a maximal family of pinched curves on $S$, and let $S'$ be a component of the surface obtained by cutting $S$ along $\calc$.
  Then, up to conjugacy in $G$,  the image of 
 $\pi_1(S') $ by $p$ is contained in a bottom group of $\Gamma$, or there is an incompressible  subsurface $Z\inc S'$ such that $p$ maps boundary subgroups of $\pi_1(Z)$ into boundary subgroups of $\pi_1(\Sigma)$, and $p(\pi_1(Z))$ is a finite index subgroup of $ \pi_1(\Sigma)$. 
 \qed
 \end{lem}
 
  The next lemma will be used
 in Subsection \ref{sec_choix}.  
 Recall that the closed non-orientable surface of genus 3 contains a 
one-sided simple closed curve    $c_{N_3}$ whose complement is a once-punctured torus (see Section \ref{sec_surfaces});    it is unique up to isotopy. 
We denote by $g_{N_3}$ a generator of $\pi_1(c_{N_3})$  
(well-defined up to  inversion and conjugacy).

 \begin{lem}   \label{lemcle3}
 Let $\Gamma$ and $p$ be as in Lemma \ref{lemcle2}. Assume furthermore that $S$ is exceptional and $\Sigma$ is not.  If the image of $p$ is not contained in a conjugate of a bottom group of $\Gamma$, then  $S$ is   a closed non-orientable surface of genus 3,   $\Sigma$ is a  once-punctured torus, and  
$p(g_{N_3}^2)$ belongs to 
the edge group of $\Gamma$ (up to conjugacy).
 \end{lem}

\begin{proof}
Assume that the image of $p$ is not contained in a conjugate of a bottom group of $\Gamma$ and apply Lemma \ref{lemcle2}.

  We claim that there is no pinching.
 Assume there is. Since $S$ is exceptional,   
$p$ factors through a group isomorphic to $\Z*\Z/2\Z$, $\Z^2*\Z/2\Z$, or   $\pi_1(N_2)*\Z/2\Z$.
Since $G$ is torsion-free, the image of $p$ is abelian or isomorphic to the Klein bottle group $\pi_1(N_2)$. 
This implies that the image of $p$ is
conjugate into a bottom group of $\Gamma$, a contradiction.

The exceptional surface $S$ has Euler characteristic $-1$ and fundamental group $\bbF_2$, so there is an incompressible subsurface $Z\inc S$ such that $p$ induces an isomorphism between $\pi_1(Z)$ and $\pi_1(\Sigma)$, with $\Sigma$ a  once-punctured torus.  
 Up to conjugacy, $p$ maps the boundary subgroups of $Z$
into the boundary subgroups of $\Sigma$. In particular, boundary subgroups of $\pi_1(Z)$ are contained in the derived subgroup 
  of $Z$, and this is possible only if $Z$ is a once-punctured torus. It follows that $S$ is the exceptional closed surface,   the union of the punctured torus $Z$  and a M\" obius band.  
 \end{proof}

 We will also need:
 
 \begin{lem} \label{n=1}
 If $p:Q\to G$ is a non-pinching \ndbpm\ associated to a centered splitting $\Gamma$ of $G$, there is a unique bottom group $B_1$  (i.e.\ $\Gamma$ is simple), 
and the image of $Q$ is contained in a conjugate of $B_1$.

If  $\hat p:G\ra G$ is an extension of $p$ which maps $B_1$ onto a conjugate $B_1^g$, then $B_1^g$
contains the whole image of $\hat p$. In particular, $\hat p$ is not injective.
\end{lem}

Such extensions $\hat p$ exist by Lemma \ref{extid}.

\begin{proof}
The first assertion follows from 
  Lemma \ref{lemcle4} and Remark \ref{vi}. 
  
  By 1-acylindricity, any two distinct conjugates of $B_1$ have trivial intersection. Now consider two adjacent vertices in the Bass-Serre tree $T$ of $\Gamma$. Their stabilizers are conjugate to $Q$ or $B_1$, and their intersection is a cyclic group on which $\hat p$ is injective, so their images by $\hat p$  are contained in the same conjugate of $B_1$. Thus all vertex stabilizers are mapped into the same conjugate.  This conjugate is invariant under conjugation by every element of $\hat p(A)$, and malnormal, so the lemma follows. 
\end{proof}

\subsection{Simple towers} \label{stt}

  Recall that a centered splitting is simple if it has a single bottom group (Definition \ref{centspl}).

\begin{prop} \label{equivsimple}
  Let $G$ be any torsion-free finitely generated group. 
Let $\Gamma$ be a non-exceptional \emph{simple} centered splitting of $G$, with bottom group $B_1$ and surface group $Q=\pi_1(\Sigma)$.  
Assume that
  any two non-trivial elements of $B_1$ have  conjugates which do not commute 
(this holds in particular if $B_1$ is CSA and non-abelian). 
Then the following are equivalent:
\begin{enumerate}
\item
There exists a non-degenerate \bpm\  
$p:Q\to G$. 
\item There exists a retraction $\rho:G\to B_1$ with $\rho(Q)$   non-abelian.
\end{enumerate} 
\end{prop}

\begin{proof}
 For $2\implies1$, we simply let $p$ be the restriction of $\rho$ to $Q$.  We now prove $1\implies2$. 
Denoting by $b$ the number of components of $\partial\Sigma$,
  we may present $G$ as $$G=\langle  \pi_1(\Sigma),  B_1,t_2,\dots,t_b\mid c_1=a_1, t_ic_it_i\m=a_i \ \text{for}\  i=2,\dots, b\rangle$$ where $a_i\in   B_1$ 
  and the $c_i$'s are generators of maximal boundary subgroups $C_i$ of $\pi_1(\Sigma)$,   see Figure \ref{fig_retraction}.

Since $p$ maps $c_1$ to a conjugate, we may assume after composing $p$ with a conjugation that $p(c_1)=c_1$. 
Now assume that  there is no pinched curve.
 By Lemma \ref{lemcle4}, the image of $\pi_1(\Sigma)$ by $p$ is contained in  a conjugate of $B_1$. This conjugate contains  $p(c_1)=a_1$, so must be $  B_1$   (by malnormality, see Remark \ref{malnor}).
 
We define $\rho$ as being the identity on $B_1$, and equal to $p$ on $\pi_1(\Sigma)$. This ensures that $\pi_1(\Sigma)$ has   non-abelian image. We must now define $\rho(t_i)\in  B_1$ in a way which is compatible with the relation $t_ic_it_i\m=a_i $. We know that $p(c_i)$ belongs to $B_1$ and is conjugate to $c_i$, hence to $a_i$, in $G$. 
 By malnormality,
 $p(c_i)$ is conjugate to $a_i$ in $ B_1$, so 
  $p(c_i)=g_ia_ig_i\m$   for some $g_i\in B_1$. We then define $\rho(t_i)=g_i\m$.

  \begin{figure}[ht!]
    \includegraphics[width=\linewidth]{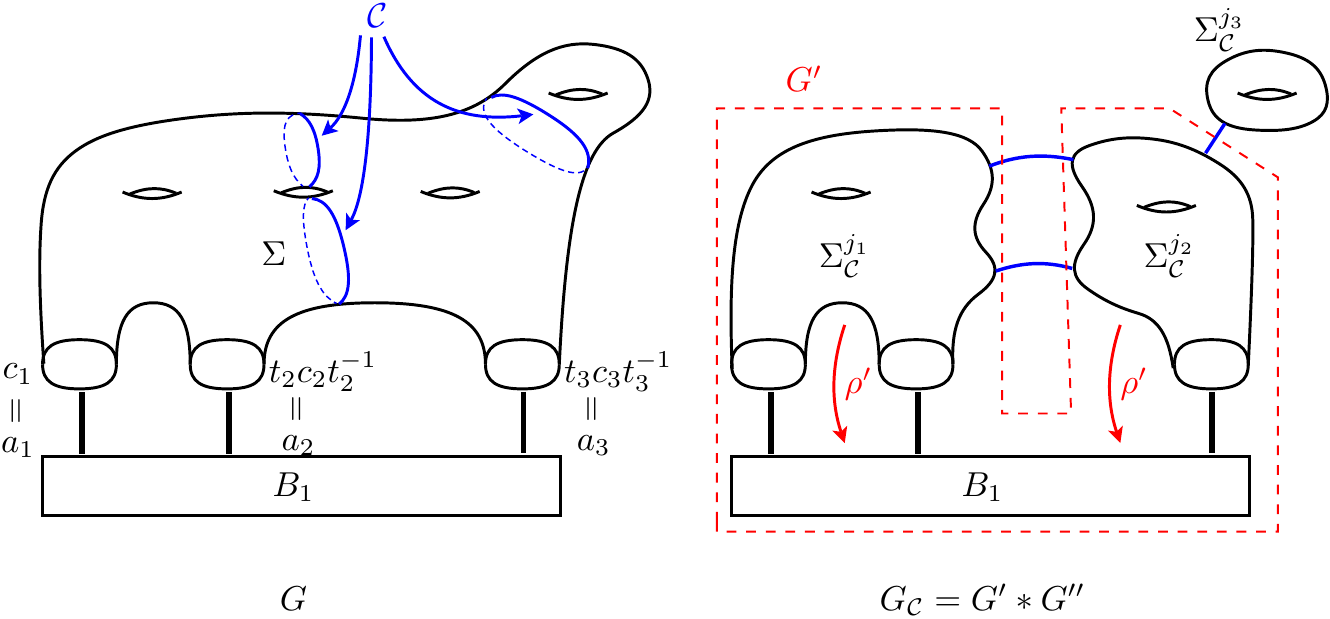}
    \caption{Proof of Proposition  \ref{equivsimple}.}\label{fig_retraction}
  \end{figure}

We now allow pinching (see Figure \ref{fig_retraction}).
We consider a pinched quotient $\pi_\calc:G\to G_\calc$ 
as in Definition \ref{pc} (with $G_1=G$),  with its splitting $\Gamma_\calc$. 
We identify $B_1$ with its image in $G_\calc$.
We shall define a retraction $\rho':G_\calc\to B_1$, and $\rho$ will be $\rho'\circ\pi_\calc$.

Let $G'$ be  the fundamental group  of the maximal subgraph  $\Gamma^\Z_\calc$ of $\Gamma_\calc$ containing no edge with trivial group (the star of the bottom vertex $v_1$). It is a free factor of $G_\calc$, and we can write $G_\calc\simeq G'*G''$, with $G''$   isomorphic to the free product of 
closed surface groups (corresponding to surfaces $\Sigma_\calc^j$ disjoint from $\bo \Sigma$) and a free group $F$ (the topological fundamental group of the graph obtained from $\Gamma_\calc$ by collapsing 
  $\Gamma^\Z_\calc$ 
to a point,  of rank 2 on Figure \ref{fig_retraction}). Note that $G''$  either maps onto $\Z$ or is a (possibly trivial) free product of  cyclic groups of order 2   (fundamental groups of projective planes).

First suppose that  all components of $\bo \Sigma$ belong to the same surface $\Sigma_\calc^{j_0}$.
We can define $\rho'$ on $G'$ in the same way as we defined $\rho$ in the non-pinching case, and   define $\rho'$ arbitrarily on $G''$. We just have to make sure that the image of $\pi_1(\Sigma_\calc)$ is not   abelian. 

This is easy to do, using the assumption on $B_1$, if $G''$ maps onto $\Z$, since $\pi_1(\Sigma_\calc)$ may be identified with $\pi_1(\Sigma_\calc^{j_0})*G''$.  If $G''$    is a free product of  cyclic groups of order at most 2, it is mapped trivially by the map $p_\calc: \pi_1(\Sigma_\calc)\to G$ because $G$ is torsion-free, and the image   of $\pi_1(\Sigma_\calc)$ by $\rho'$ is   non-abelian because it is isomorphic to the image of $\pi_1(\Sigma)$ by $p$.

Now suppose that there are at least two surfaces $\Sigma_\calc^{j_1}$  and $\Sigma_\calc^{j_2}$ containing a component of $\bo\Sigma$. The key difference with the previous case is that we can no longer assume $G''\inc \pi_1(\Sigma_\calc)$, so making the image of $\pi_1(\Sigma_\calc)$   non-abelian is less direct. 
 
We can still define $\rho'$   on  $G'$: 
for each   $\Sigma_\calc^{j}$ containing a component of $\bo \Sigma$, we change the restriction of $p$ to $\pi_1(\Sigma_\calc^{j})$ by a conjugation   so that 
it fixes some boundary subgroup of   $\pi_1(\Sigma_\calc^j)$, just like we ensured $p(c_1)=c_1$ in the non-pinching case, and we define $\rho'$ on the whole of $G'$ as above. We now explain how to use $G''$ to make the image of $\pi_1(\Sigma_\calc)$   non-abelian.

In defining $\rho'$ on $G'$, we have   viewed $\pi_1(\Sigma_\calc^{j_1})$  and $\pi_1(\Sigma_\calc^{j_2})$ as subgroups of $G'$ (vertex groups of $\Gamma^\Z_\calc$). When we view them as subgroups of $\pi_1(\Sigma_\calc)$, hence as vertex groups of the free splitting $\Delta_\calc$  of $\pi_1(\Sigma_\calc)$, they get replaced by conjugates.
More precisely, 
connectedness of $\Sigma$ implies that the free group $F$ appearing as a free factor in   $G''$ is non-trivial, and we can find a primitive element $t$ of $F$ such that  a conjugate of $\pi_1(\Sigma_\calc)$ contains $\pi_1(\Sigma_\calc^{j_1})$ and $t\pi_1(\Sigma_\calc^{j_2})t\m$. 

We may choose $\rho'(t)$ arbitrarily, and we shall define it so that the image of  $\pi_1(\Sigma_\calc)$ by $\rho'$  is   non-abelian.
  Since the images of $\pi_1(\Sigma_\calc^{j_1})$ and $\pi_1(\Sigma_\calc^{j_2})$ by $\rho'$ contain some non-trivial elements $a_1,a_2$,
our assumption on $B_1$ ensures that there exists $x\in B_1$ such that $\grp{a_1,xa_2x\m}$ is non-abelian,
and we can define $\rho'(t)=x$.
We conclude by extending $\rho'$ from $\grp{t}$ to $G''$ arbitrarily.
\end{proof}

  If $G$ is CSA (in particular, if $G$ is hyperbolic), the  hypothesis on $B_1$ simply asks that  $B_1 $ be non-abelian. \emph{From now on, we assume that $G$ is a torsion-free finitely generated CSA group.}

\begin{rem} \label{fact}
Assume that  
   $B_1 $  is abelian 
   and 1 holds. Then $p$  is pinching   by Lemma \ref{lemcle4} and  there exists an epimorphism $\pi'':G''\onto\Z$. Using $\pi''$,  we    can construct a  surjective map $\tilde\rho :G\to B_1*\Z$, equal to the identity on $B_1$, such that $\tilde\rho(Q)$ is  non-abelian. 
    If we let  $t\in G$ be any preimage 
   of a generator of $\Z$ under $\tilde \rho$,
   then $B_1$ and $\grp{t}$ generate their free product in $G$, and we can view $\tilde \rho$ as 
  a retraction from $G$ to  its subgroup $B_1*\grp{t}$.
   This will be used in the  proof of Proposition \ref{equivpassimple}.
\end{rem}

\begin{dfn}[Simple étage]\label{dfn_simple_etage}
  A   CSA group   $G$ is a \emph{simple étage of surface type} over a     non-abelian subgroup $H $ if $G$ has 
  a non-exceptional simple centered splitting $\Gamma$, with bottom group  $B_1=H$, 
  and the equivalent conditions of Proposition \ref{equivsimple} are satisfied. 

  We say that $G$ is a \emph{simple étage of free product type} over   an arbitrary subgroup $H$ if $G$ splits as a free product $G=H*\Z$. 

In either case, we say that $G$ is a \emph{simple étage} over $H$.
\end{dfn}

\begin{rem} \label{pasz}
 We do not allow $H $   to be abelian in the surface-type case (in particular, the splitting defining a parachute as   in Definition \ref{centspl} is not a simple \'etage, though it is an extended \'etage in the sense of Definition \ref{dfn_extended}).
 In a simple étage of free product type, we do allow  $H$   to be   abelian, in particular $G$ may be $\F_2$ or $\Z$.
In some statements (such as Theorem \ref{Tarski}), we  
specifically assume that $H$ is non-abelian also in this case.

\end{rem}

\begin{dfn}[Simple tower]\label{dfn_simple}
A 
CSA group $G$ is a simple tower over a subgroup $H$ if there exists a chain of subgroups
$G=G_0>G_1>\dots >G_k=H$ such that $G_i$ is a simple étage over $G_{i+1}$.
The tower is \emph{trivial} if $k=0$.
\end{dfn}

We allow the case $k=0$, so $G$ is always a simple tower over itself.
On the other hand, not every group is a simple étage. 

\begin{rem}\label{rk_amalgam} 
  If $G$ is a simple tower over $H$, and we have an embedding $H\hookrightarrow H'$,
  then $G*_H H'$ is a simple tower over $H'$.
  In particular, $G*K$ is a simple tower over $H*K$.
\end{rem}

\begin{rem}\label{rk_tour_malnormal}
  If $G$ is a simple tower over $H$, then $H$ is malnormal in $G$ 
(see   Remark \ref{malnor}). 
\end{rem}

\subsection{Extended towers and retractable splittings} \label{ext}

  As mentioned earlier, we   assume from now on  that $G$ is CSA.

  We   want to extend Proposition \ref{equivsimple} to non-simple splittings. The statement given below is more complicated than Proposition \ref{equivsimple}.  One reason is the following. If $\Gamma$ is a simple centered splitting, its base $B_\Gamma=B_1$ is well-defined up to conjugacy as a subgroup of $G$; if $\Gamma$ is not simple, each 
  bottom group $B_i$ is a subgroup of $G$ well-defined up to conjugacy, but   their abstract free product $B_\Gamma$ is not. In terms of logic, ``Tarski''
  (Theorem 
  \ref{Tarski})   about elementary embeddings is valid for simple \'etages, but not   always for extended \'etages (see Theorem \ref{contrex}).

\begin{prop} \label{equivpassimple}
  Let $\Gamma$ be a non-exceptional   centered splitting of $G$, with bottom groups $B_1,\dots,B_n$ and surface group $Q=\pi_1(\Sigma)$.
  The following are equivalent:
\begin{enumerate}
\item
There exists a non-degenerate \bpm{}  
$p:Q\to G$. 
\item There exist:
  \begin{itemize}
  \item conjugates $\tilde B_1,\dots,\Tilde B_n$ of the bottom groups generating their free product
    $\Tilde B_\Gamma=\tilde B_1*\dots*\Tilde B_n$,  
    
    and: 
  \item a retraction $\rho:G\ra \tilde B_\Gamma^e$ with $\rho(Q)$   non-abelian,
     where $\tilde B_\Gamma^e=\tilde B_\Gamma$  if $\tilde B_\Gamma$ is non-abelian, and  $\tilde B_\Gamma^e=\tilde B_\Gamma*\grp{t}$   
     for some non-trivial extra element $t\in G$   if  $n=1$ and $B_1$ 
     is abelian.
  \end{itemize}
\end{enumerate}
\end{prop}

 Proposition \ref{equivpassimple} is equivalent to Proposition \ref{equivsimple} when $\Gamma$ is simple and $B_1$ is non-abelian, but 
it extends it
in two ways: we allow non-simple splittings, and     simple splittings with $B_1$   abelian. 

 We shall say that  $\Gamma$ is retractable if it satifies the equivalent conditions above (see Definition \ref{dfn_retractable}).
 Note that, when $n=1$ and $B_1$ is  abelian, requiring $\rho(Q)$ to be non-abelian forces us to   enlarge $B_\Gamma$.

\begin{rem}\label{Error}
The extra condition when  $n=1$ and $B_1$ is abelian   has been added in the erratum \cite{Perin_elementary_erratum} after the third named author, while working   on \cite{PeSk_homogeneity}, observed that the first condition in Proposition \ref{equivpassimple} does not always imply the second. For example,  $\pi_1(N_4)$ admits a non-exceptional centered splitting  with central group $\pi_1(N_{3,1})$ and a single bottom group $B_1$ isomorphic to $\mathbb{Z}$, and there is an obvious \ndbpm{}
 $p: \pi_1(N_{3,1})\rightarrow \pi_1(N_4)$, but since $B_1$ is abelian no retraction $\rho:\pi_1(N_4)\rightarrow B_1$ with $\rho(\pi_1(N_{3,1}))$ non-abelian can be found.  
 \end{rem}

\begin{proof} 
  First assume that $\Gamma$ is not simple, i.e.\ $n\geq 2$. Let $\pi_\calc:G\to G_\calc$ be  a pinched quotient   (Definition \ref{pc}).
We denote by $v_1,\dots, v_n$ the bottom vertices of $\Gamma$ and also  the corresponding vertices of $\Gamma_\calc$.
Note that 
a vertex    of $\Gamma_\calc$ carrying some $\pi_1(\Sigma_\calc^j)$ cannot be joined to two distinct $v_i$'s by Remark \ref{vi}. We may therefore write $G_\calc$ as a  free product $G_\calc=G'_1*\dots*G'_n*G''$, with $G'_i$ the fundamental group of the star of the vertex $v_i$ in $\Gamma_\calc$. 

As in the proof of Proposition \ref{equivsimple},
we may retract each $G'_i$ onto a subgroup $\tilde B'_i<G'_i$ conjugate to the stabilizer of $v_i$ in $\Gamma_\calc$. 
These retractions extend to a retraction 
$\rho':G_\calc\to \langle \tilde B'_1,\dots,\tilde B'_n\rangle\simeq  B_1 *\dots * B_n\simeq   B_\Gamma$.
We choose  a lift $\tilde B_i\inc G$ which  is mapped isomorphically onto $\tilde B'_i$ by $\pi_\calc$.
Since $\tilde B'_1,\dots \tilde B'_n$ generate their free product, so do $\tilde B_1,\dots,\tilde B_n$, and $\pi_\calc$
induces an isomorphism from  $\langle \tilde B_1*\dots*\tilde B_n\rangle$ to $\langle \tilde B'_1,\dots,\tilde B'_n\rangle$.   

We define $\rho$ by postcomposing  
$\rho'\circ\pi_\calc$ with the inverse of this isomorphism. 
  The image of $Q$ 
is automatically non-abelian because it contains a free group generated by the image of two boundary subgroups of $\pi_1(\Sigma)$ contained  in $B_1$ and $B_2$ respectively (up to conjugacy). 

If $\Gamma$ is simple and $B_1 $   is abelian, then 2 $\imp$ 1 is clear and 1 $\imp$ 2 follows from Remark \ref{fact}.
\end{proof}

Proposition \ref{equivpassimple} explains the terminology in the following definition. 

\begin{dfn}[Retractable surface, retractable splitting] \label{dfn_retractablesurf}\label{dfn_retractable}
  Let   
  $Q$ be a subgroup of $G$   isomorphic to 
  the fundamental group of a non-exceptional surface $\Sigma$. 
We say that  $\Sigma$, and $Q$, are \emph{retractable (in $G$)} if there exists a \ndbpm\ $p:Q\ra G$. 

  A non-exceptional centered splitting $\Gamma$ satisfying the equivalent conditions of Proposition \ref{equivpassimple} 
  is called a \emph{retractable} centered splitting,  and we say that its central vertex is retractable.

  More generally, if a \ndbpm\ $p:Q\ra G$ is associated to a surface-type vertex $v$ of a splitting $\Gamma$ of a subgroup $G_1<G$, we say that $\Gamma$ and  $v$  
are  \emph{retractable in $G$} (when $G_1=G$, we just say retractable).
\end{dfn}

\begin{rem}\label{rem_retractable}
  If   a  cyclic splitting $\Gamma$ is retractable, so is the centered splitting obtained by collapsing all edges not containing the vertex carrying $\Sigma$ (and subdividing if needed).  In particular, $G$ has a retractable centered splitting as soon as some cyclic splitting has a surface-type vertex group $Q$ with a non-degenerate \bpm{}  
$p:Q\to G$.
  
  Also note that retractable centered splittings  have to be  non-exceptional. 
\end{rem}

\begin{rem} \label{ineg}
  Recall (Corollary 3.5 of \cite{GLS_finite_index}) that, if a centered splitting $\Gamma$ is retractable, then $g\ge n_1$, with $g$ the genus of $\Sigma$  and $n_1$ the number of bottom vertices of $\Gamma$ having valence 1. 
\end{rem}

The following definition was introduced by Perin \cite{Perin_elementary,Perin_elementary_erratum}. 

\begin{dfn}[Extended étage, extended tower]\label{dfn_extended}
  If 
$G$ has a non-exceptional centered splitting $\Gamma$ 
satisfying the equivalent conditions of Proposition \ref{equivpassimple}, we say that  $G$ is an \emph{extended étage of surface type} over
 the subgroup $\Tilde B_\Gamma=\tilde B_1*\dots *\tilde B_n$ appearing in assertion 2  (always isomorphic to the base of $\Gamma$,  even if $n=1$ and $B_1$ is abelian).

An \emph{extended étage of free product type} over $H$ is 
a free product $G=H*\Z$ (as in Definition \ref{dfn_simple_etage}).

 A group $G$ is an \emph{extended tower} over a subgroup  $H$ if there exists a sequence of  subgroups
$G=G_0> G_1> \dots > G_k=H$ such that $G_i$ is an extended étage over $G_{i+1}$.
\end{dfn}

As in Definitions \ref{dfn_simple_etage} and \ref{dfn_simple}, we   allow 
$H$ to be abelian in \'etages of free product type, and we allow towers with no \'etage, so that $G$ is always an extended tower over itself. 

\begin{rem} \label{passimp}
   Beware  that  an extended étage of surface type with a single bottom group   $H $  (i.e.\ a simple retractable centered  splitting) 
 does not define   a simple \'etage if $H$ is abelian (see Remark \ref{pasz});   in particular,  a parachute is not always a simple \'etage.
\end{rem}

\begin{rem} \label{quot}
  If $G$ is an   extended tower
 over $H$, there 
  is a retraction from $G$ to $H$, so one may view $H$ both as a subgroup and a   quotient of $H$.  
 In general, however, $H$ is not canonically defined as a subgroup of $G$; for instance, changing the choice of the conjugates $\tilde B_i$'s in Proposition  \ref{equivpassimple} expresses $G$ as an \'etage 
 over isomorphic, but non-conjugate,  subgroups.
   What is canonical is the family of  homomorphisms 
$B_\Gamma\ra G$ which agree with a conjugation on each $B_i$.
Depending on context, one may want to think
of $H$ either as a subgroup of $G$, or as an abstract group, forgetting about its embedding in $G$. Roughly speaking, Sections \ref{ordre} and \ref{elemequiv} are about abstract groups, Section \ref{ee} about subgroups.
\end{rem}

\begin{rem} \label{surfdeux}
When $n=1$ and   $B_1\simeq\Z$    ($G$ is a parachute), 
the existence of the retraction $G\ra \tilde B_1*\grp{t}$ in  Proposition \ref{equivpassimple} is equivalent to the existence of a 
retraction  $\tilde\rho:G*\Z\to B_1*\Z$ such that  $\tilde\rho(Q)$ is not abelian. 
This is the way Perin defines it in   \cite{Perin_elementary_erratum}. 
Note in particular that $G*\Z$ is a simple \'etage of surface type over    $B_1*\Z\simeq\F_2$.
\end{rem}

The following remarks apply to both simple and extended \'etages/towers.

\begin{rem}\label{rk_prodlib}
If $G$ is an   étage over $H$, then $G*G'$ is an étage over $H*G'$ for any $G'$ (replace the bottom group  $B_1$ of $\Gamma$ by $B_1*G'$). 
It follows that, if $G$ is a   tower over $H$ and $G'$ is a   tower over $H'$, then $G*G'$ is a   tower over $H*H'$.
\end{rem}

\begin{rem} \label{zenbas}
In a tower, one can switch étages of surface type and of free product type as follows: 
if $G_{i-1}=G_i*\Z$ and $G_i$ is an  étage of surface type over $G_{i+1}$,
then $G_{i-1}$ is an   étage of surface type over $G'_i=G_{i+1}*\Z$, which is an   étage of free product type over $G_{i+1}$. This allows us to assume that all \'etages of surface type are at the top of the tower.  
\end{rem}

 Proposition \ref{equivpassimple} is about single \'etages. The following is about towers.

\begin{prop}\label{prop_retractions}
  Assume that $G$ is an extended tower over a subgroup $H$.
  
  Then there are subgroups $G=G_0> G_1>\dots> G_k= H'$,
  a (possibly trivial) free group $F\subset H'$ such that $H'=H*F$, and for $i<k$ non-exceptional
   retractable centered splittings   $\Gamma_i$ 
  of $G_i$, and
  retractions $\rho_i:G_i\ra G_{i+1}$ such that:
  \begin{itemize}
  \item the image under $\rho_i$ of the central vertex group $Q_i$ of   $\Gamma_i$  is
      non-abelian;
  \item 
  for $i>0$, the group    $G_{i}$ is the free product of some conjugates of the bottom groups of  $\Gamma_{i-1}$, plus an extra cyclic factor if $i=k$ and  $\Gamma_{k-1}$ is simple with an abelian base.
  \end{itemize}
\end{prop}

\begin{proof}
  Using Remark \ref{zenbas}, starting from a sequence of étages defining the tower, one
    can move  the étages of free product type  down to get 
  a sequence of groups
  $G=G_0, G_1, \dots, G_{k}=H'$, 
  where each $G_i$ is an étage of surface type over $G_{i+1}$ and $H'=H*F$ with $F$ free. 
We conclude by applying  Proposition \ref{equivpassimple}.   If   $G_k$ is abelian, we must enlarge $G_k=H'$ and $F$ by taking their free product with $\grp{t}$ as in Proposition \ref{equivpassimple}.
\end{proof}

\begin{rem}
 Assume  for simplicity that $G$ is hyperbolic (or more generally that all abelian subgroups of $G$ are cyclic).
The subgroups appearing in this proposition are almost the same as those in Definition \ref{dfn_extended}, but not exactly, because of
  parachutes: a parachute is an étage over $\Z$, but in order to have a retraction such that the surface group has  non-abelian  image one must retract to a free group of rank 2.
Conversely, if   there are subgroups $G_i$ as in 
the proposition, then $G$ is an extended tower over $H$,  except possibly if $H=G_k\simeq \F_2$: in this case it may happen that  $G$ is   a tower over a cyclic free factor of $H$, but   not over $H$ itself.
\end{rem}

\begin{figure}[ht!]
  \centering
  \includegraphics{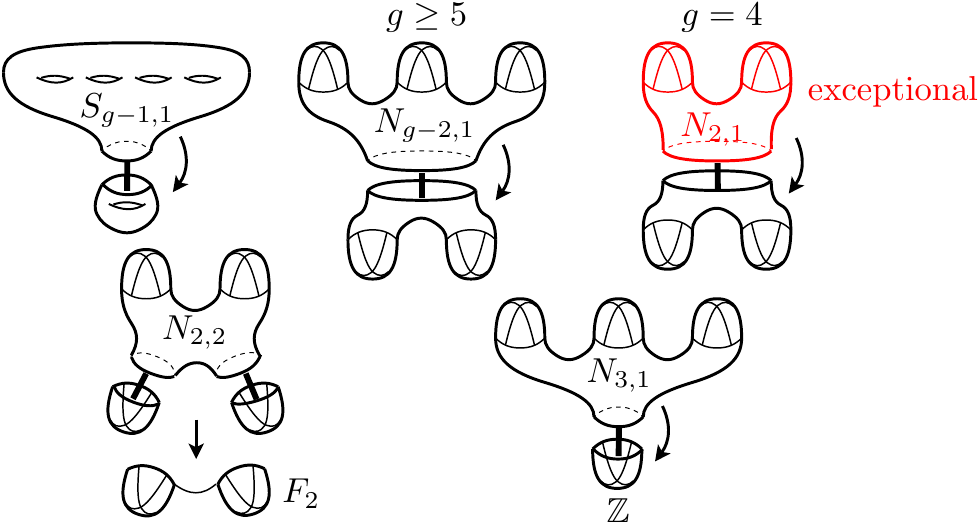}
  \caption{Top part: non-exceptional surfaces except $N_4$ are simple étages over $\bbF_2$ (the top right picture is not a valid étage because $N_{2,1}$ is exceptional). 
  \newline
   Bottom part: $N_4$ is an extended étage in two different ways.}
\label{fig_N4}
\end{figure}

The following lemma implies that a given $G$ cannot be a tower of arbitrary length.

\begin{lem}\label{dccg}
   Let $G$ be a torsion-free CSA group. Denote by $b_1^2(G)$ the first Betti number modulo 2 (the dimension of $\mathrm{Hom}(G,\Z/2\Z)$ over $\Z/2\Z$).
If $G$ is an extended \'etage over $H$, then $b_1^2(G)>b_1^2(H)$.
\end{lem}

\begin{proof}
 The result is clear if the \'etage is of free product type. Assume it is of surface type, associated to a surface $\Sigma$ and a centered splitting $\Gamma$. Since $H$ is a retract of $G$, it suffices to find a homomorphism from $G $ onto $\Z/2\Z$
  vanishing on $H$.
 This is easy to do 
 if $\Sigma$ is not planar or if $\Gamma$ is not a tree of groups.
 In the remaining case ($\Sigma$ is planar and $\Gamma$ is a tree), $\Gamma$ cannot be retractable \cite{GLS_finite_index}: a retractable $\Gamma$ with $\Gamma$ planar has no terminal vertex   (see Remark \ref{ineg}).
\end{proof}

 Recall that  $S_{g,b}$ and $N_{g,b}$ denote the    orientable/non-orientable 
surfaces of genus $g$ with $b$ boundary components, and $S_g=S_{g,0}$, $N_g=N_{g,0}$.

\begin{example} [Closed surface groups, see Figure \ref{fig_N4}]\label{csurf}
\emph{All non-exceptional closed surface groups except $\pi_1(N_4)$ are simple étages over $\bbF_2$, and $\pi_1(N_4)$ is an extended étage over $\Z$  and $\bbF_2$.
}

 We sketch the proof (compare \cite{Perin_elementary, LPS_hyperbolic}).
  Every $S_g$ with $g\geq 2$ is non-exceptional and $\pi_1(S_g)$  is a simple étage over $\bbF_2$ (see Figure \ref{fig_N4}).
For $g\geq 5$,  decomposing $N_g$ into $N_{g-2,1}$ and $N_{2,1}$  makes $\pi_1(N_g)$  a simple étage over $\bbF_2$.  
When $g=4$, this does not yield an étage structure because $N_{2,1}$ is exceptional.

To express $\pi_1(N_4)$ as an extended étage, 
 view  
 $N_4$ as a   twice-punctured Klein bottle $N_{2,2}$ with a M\"obius band  
glued to each boundary component.  
This  yields a retractable splitting of its fundamental group 
$\pi_1(N_4)$ with base $\bbF_2$, so $\pi_1(N_4)$ is an extended \'etage over $\bbF_2$ (the associated retraction maps $\pi_1(N_4)=\langle a_1,a_2,b_1,b_2\mid a_1^2a_2^2=b_1^2b_2^2\rangle$ to $\langle a_1,a_2\rangle$ by sending $b_i$ to $a_i$). 
One may also decompose $N_4$ as the union of 
$N_{3,1}$ and a M\"obius band, so $\pi_1(N_4)$ has a simple retractable splitting with base $\Z$ (but is not a simple \'etage,   see Remark \ref{pasz}). 

The surface $N_3$ is exceptional, its fundamental group is not an extended \'etage. 
\end{example}

\subsection{Relative extended towers}

Recall that, if $G$ is an extended \'etage of surface type over $H$, associated to a   centered splitting $\Gamma$ of $G$,  the bottom groups are by definition elliptic in $\Gamma$ and are well-defined up to conjugacy as subgroups of $G$. On the other hand, if the \'etage is not simple, there is no canonical embedding of $H$ into $G$. 
This motivates the following definition,   which is important for applications.

\begin{dfn}[Relative   extended tower  \cite{Perin_elementary}]\label{relto}
Let $G=G_0>G_1>\dots >G_k=H$ be an extended tower. 
The tower is \emph{relative to a subgroup $A\inc H$} if
     each \'etage of surface type is relative to $A$ in the following sense:   
the group $A$ is contained, up to conjugacy, in a bottom group of the centered splitting defining the étage.
\end{dfn}

\begin{example} \label{subtil}
Any simple tower over $H$ is relative to $H$, or to any subgroup $A\subset H$.
Now suppose that $G$ has a retractable centered splitting  
with two bottom groups $B_1,B_2$ and  base $G_1=B_1*B_2$.
Also suppose that $G_1$ is a   simple  \'etage over $H$. The tower $G>G_1>H$ is relative to $H$ if and only if $H$ is contained in a conjugate of $B_1$ or $B_2$.
\end{example}

\begin{rem} 
  If $G$ is an étage of free product type over $H<G$, i.e. $G=H*\bbZ$, then the  
    \'etage is automatically relative to $A$ for any subgroup $A<H$.
\end{rem}

\begin{rem}\label{prop_retractions_rel}
Proposition \ref {prop_retractions} holds if $G$ is an extended tower over $H$ relative to $A$, with $A$ contained in a conjugate of a bottom group of  $\Gamma_i$ for $i=0,\dots,k-1$.
\end{rem}

\begin{rem} \label{ttrel}
  If $G_i$  as in Definition \ref{relto} is an extended \'etage of surface type over $G_{i+1}$, then $G_{i+1}$ is the free product of the bottom vertex groups of $\Gamma_i$. Thus, 
if $A<H$ is a one-ended subgroup,
the tower is automatically relative to $A$.
More generally, if 
$A$ is 
one-ended relative to a subgroup $A'$, and the tower is relative to $A'$,  then the tower is also relative to $A$.
\end{rem}

\begin{rem}\label{rem_last_simple}
 An  extended \'etage of surface type over $H$ is relative to $H$ if and only if it is simple. 
If $G$ is an extended tower over $H$ relative to $H$, then the last étage   (between $G_{k-1}$ and $G_k$) has to be an  étage of free product type or a simple étage.
\end{rem}

\subsection{Relation with other definitions}\label{sec_defs}
Sela introduced hyperbolic $\omega$-residually free towers in \cite{Sela_diophantine1}.
They are defined as simple towers (involving only étages of surface type) over a (non-abelian)
free product of (maybe 0) non-exceptional closed surface groups and a (possibly trivial) free group.
Because étages of free product type may be moved to the bottom by Remark \ref{zenbas},
one does not change the definition if one allows étages of free product type in this definition.
Moreover, 
since fundamental groups of non-exceptional  surfaces except $N_4$ are simple towers over $\F_2$  (Example \ref{csurf}),
and since $\pi_1(N_4)*\pi_1(N_4)$ and $\pi_1(N_4)*\bbZ$ are simple towers over $\bbF_2$,
$\omega$-residually free towers coincide with simple towers over $\bbF_2$ or over $\pi_1(N_4)$.

 Our extended towers are almost identical to the extended hyperbolic towers defined by Perin in \cite{Perin_elementary_erratum}.

Perin defined hyperbolic floors/\'etages in \cite{Perin_elementary} allowing the possibility of 
  a non-simple centered splitting, \ie  a splitting having  several bottom vertices.
In \cite[Section 4]{Perin_elementary_erratum} Perin alludes to the possibility that Sela also thinks of   non-simple floors/\'etages, 
but this definitely does not appear in print.

As mentioned   in Remark \ref{Error}, Perin's definition needed a fine tuning,   hence the introduction of extended hyperbolic floors in \cite{Perin_elementary_erratum}.    We do not 
use the same retraction as Perin  
in the case of parachutes, but this makes no difference 
 (see Remark \ref{surfdeux}).

 Unlike Perin, we insist that only one surface appears in each \'etage (we only use centered splittings). This does not affect the definition of   towers, though it increases the number of \'etages in a given tower.   We do not use Bass-Serre presentations, but we select conjugates of the $B_i$'s in Proposition  \ref{equivpassimple}.
 
Perin does not consider étages of free product type, but by Remark \ref{zenbas} 
\emph{$G$ is an extended   tower over $H$ in our sense if and only if $G$ is   an extended hyperbolic tower over $H*F$   in the sense of Perin \cite{Perin_elementary_erratum} for some
free subgroup $F$.
}

We  now quickly review the definition of regular NTQ systems and NTQ groups from \cite[Section 7.6]{KhMy_implicit}.
If $\cals=\{S_1(x_1,\dots,x_n),\dots,S_k(x_1,\dots,x_n)\}$ is a set of equations in the variables $x_1,\dots, x_n$ (maybe with constants in $G$),
the coordinate group $G_{R(\cals)}$ of $\cals$ is the largest residually-$G$ quotient of $\grp{G,x_1,\dots,x_n\mid \cals }$,
i.e. the quotient of
$G*\grp{x_1,\dots,x_n}$ by 
the set of words  $r(x_1,\dots,x_n)$ (with constants in $G$) such that 
any  $g_1,\dots, g_n\in G$ satisfying $S_i(g_1,\dots,g_n)=1$ also satisfies $r(g_1,\dots,g_n)=1$.
We will restrict to the case where $G$ is a limit group (over $\bbF_2$).

The building blocks of regular NTQ systems are quadratic systems of equations (corresponding to surface groups),
and empty equations (corresponding to the free product with a free group).
More generally, a TQ system, as defined in Definition 43 of \cite{KhMy_implicit},
also allows commutation equations (corresponding to extensions of centralizers, or abelian étages),
but these are ruled out in the definition of \emph{regular} NTQ systems.
In regular NTQ systems, one requires \cite[Def.~6]{KhMy_implicit} that the quadratic equations correspond to non-exceptional surfaces
and that the equations have a non-commutative solution in $G$ (corresponding to retractions such that the surface groups  have  non-abelian image).
Thus, a group is the coordinate group $G_{R(\cals)}$ of a regular quadratic system of equations over $G$
if and only if it is a simple étage of surface type over $G$.
 A group $L$ is the coordinate group of a regular NTQ system over $G$ (including constants) if and only if it
is a simple tower over $G$. 
We note that, for $G=\bbF_r$ with $r\geq 2$, this class of groups depends on $r$ (all such groups have an epimorphism to $\bbF_r$).
For $r=2$, this class includes the fundamental group of all non-exceptional surfaces except $N_4$.

If a system of equations $\cals$  does not involve constants in $G$, one can additionally define 
a \emph{constant-free} coordinate group $G^*_{R(S)}$  
as the largest residually-$G$ quotient of $$\grp{x_1,\dots,x_n\mid  \cals}.$$
A group $L$ is the constant-free coordinate group of a regular NTQ system without constants over the free group $\bbF_r$
if and only if $L*\bbF_r$ is a simple tower over $\bbF_r$. 
One can easily check that this notion does not depend on $r\geq 2$.
This yields a larger class of groups, containing all non-exceptional surface groups  (including the fundamental group of $N_4$), elementarily free parachutes,
$\bbZ$ and the trivial group.

\section{Some properties of towers} \label{prto}

  In this section we assume that $G$ is CSA, torsion-free, finitely generated.
\subsection{Making an \'etage simple}

Simple towers are better than extended towers; in particular, Theorem \ref{Tarski} about elementary embeddings only applies to simple towers. 
It is therefore desirable to make an \'etage simple. 
  We show that this is   possible if the surface is complicated enough, or if we are prepared to enlarge $G$  
 by taking a free product with a free group (Corollary \ref{ett}).

\begin{prop} \label{3simple}
 Assume that $G$ has a   \emph{non-simple}  
retractable centered splitting $\Gamma$ 
such that 
the associated surface $\Sigma$ has Euler characteristic satisfying $\chi(\Sigma)\le -3$. 

  Then $G$ has a retractable \emph{simple} splitting $\hat \Gamma$ whose base is isomorphic to that of $\Gamma$ (so $G$ is a simple \'etage of surface type).
\end{prop}

The proposition applies unless the non-exceptional surface $\Sigma$ is   a 4-punctured sphere  $S_{0,4}$, a   thrice-punctured projective plane $N_{1,3}$, or a   twice-punctured Klein bottle $N_{2,2}$  
(the twice-punctured torus cannot appear in a  non-simple retractable splitting by    \cite{GLS_finite_index}, see Remark \ref{ineg}). 
 These surfaces cause pathologies that will be studied in Section \ref{examp}.

\begin{figure}[ht!]
  \centering
  \includegraphics[width=\linewidth]{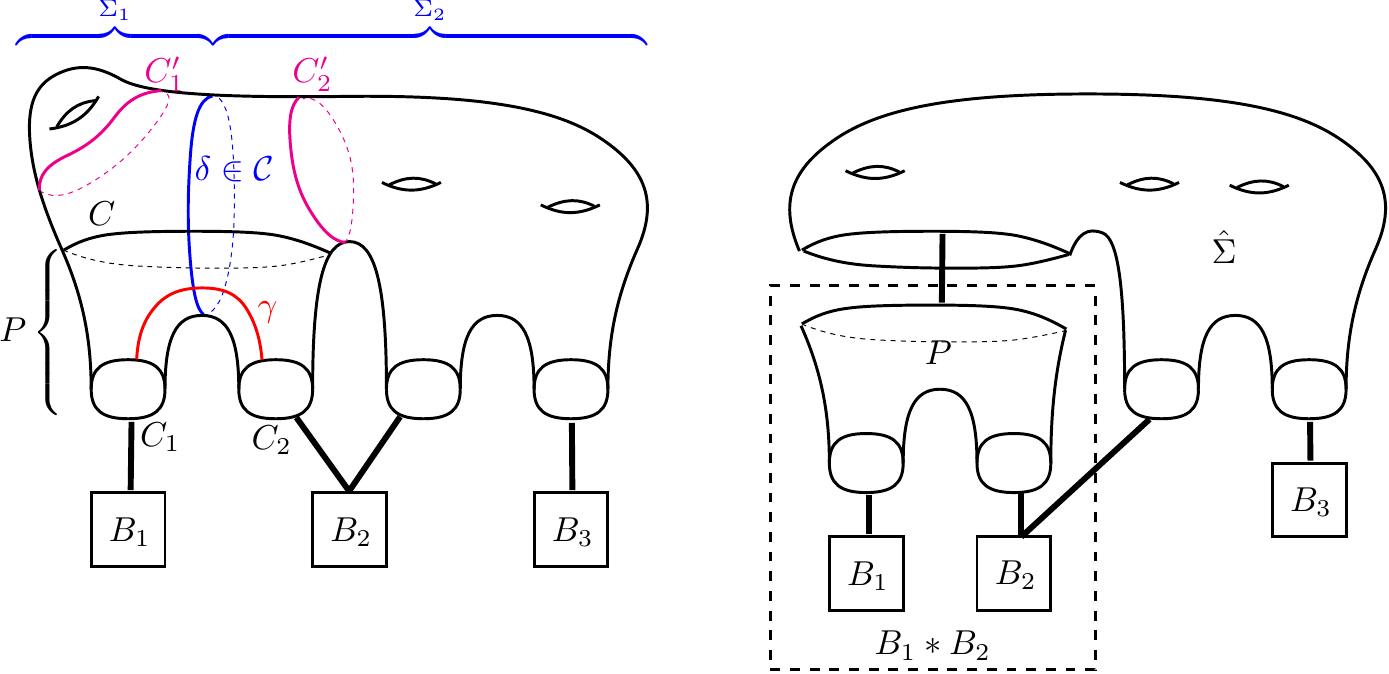}
  \caption{Proof of Proposition \ref{3simple}.}
  \label{fig_3simple}
\end{figure}

\begin{proof}
 Let $B_1,\dots,B_n$ be the bottom groups of $\Gamma$, with $n\ge2$, and $Q=\pi_1(\Sigma)$ the central group.
 Let $p: Q\to G$ be a non-degenerate \bpm. Let $C_1$ and $C_2$ be two boundary components of $\Sigma$ whose fundamental groups are contained in conjugates of $B_1$ and $B_2$ respectively (see Figure \ref{fig_3simple}).
 Using Lemma \ref{lemcle4}, we may find a  (possibly non-maximal) family $\calc$ of pinched curves separating $\Sigma$ into two   connected surfaces $\Sigma^1,\Sigma^2$ containing $C_1$ and $C_2$ respectively  (on Figure \ref{fig_3simple},   the family $\calc$ consists of the single curve $\delta$). 

We choose an arc $\gamma\inc \Sigma$ joining $C_1$ to $C_2$ and meeting $\calc$ exactly once,  and we call $\delta$ the curve of $\calc$ which meets $\gamma$.    
We view $Q$ as $\pi_1(\Sigma, x)$, with a basepoint $x\in \gamma$, and using $\gamma$ to join $C_1$ and $C_2$ to $x$ we can  view $\pi_1(C_1) $ and $\pi_1(C_2) $ as subgroups of $Q$ which are well-defined, not just up to conjugacy (this is similar to the choice of subgroups  $\tilde B_i$ in the proof of Proposition \ref{equivpassimple}).

Let $P\inc \Sigma$ be a pair of pants obtained as a regular neighborhood of $C_1\cup\gamma\cup C_2$, with $\bo  P$ consisting of $C_1,C_2$, and a third curve $C$. 
We shall   define a \bpm\  $\hat p$ from $\pi_1(\hat \Sigma)$ to $G$, where $\hat\Sigma$ denotes the surface obtained from $\Sigma$ by removing $P$.  Note that  $\chi(\hat \Sigma)=\chi(  \Sigma)+1\le-2$, so $\hat\Sigma$ is non-exceptional.

 A regular neighbourhood of $C_1\cup\gamma\cup C_2\cup\delta$ is a $4$-punctured sphere containing $P$;  its boundary consists of $C_1,C_2$ and 
curves $ C'_1,C'_2$ with $C'_i\subset \hat \Sigma\cap\Sigma^i$.
Since $\delta$ is pinched, the image of  $\pi_1(C'_i)$ by $p$ is conjugate to that of $\pi_1(C_i)$, hence contained in a conjugate of $B_i$.

We cannot directly use the restriction of $p$  to $\pi_1(\hat \Sigma)$ because it does not have to  map $\pi_1(C)$ to a conjugate, so we first have to modify $p$. 
Denote by $\Sigma_\calc$ the pinched space (see Definition \ref{pc}), so that $p$ factors through $ p_\calc:\pi_1(\Sigma_\calc)\ra G$. Let $\Sigma^1_\calc,\Sigma^2_\calc$  be the images of $\Sigma^1,\Sigma^2$ in $\Sigma_\calc$, and denote by  $p_\calc^i$ the restriction of $p_\calc$ to $\pi_1(\Sigma^i_\calc)$.
Since $\pi_1(\Sigma^1_\calc)*\pi_1(\Sigma^2_\calc)$  is a free factor
of  $\pi_1(\Sigma_\calc)$, one can   modify $p_\calc$ by composing each $p_\calc^i$ with a  conjugation so that the modified map $p'=p_\calc\circ \pi_\calc$  from $\pi_1(\Sigma)$ to $G$ is a  \bpm\   
which is the identity on $\pi_1(C_1) $ and $\pi_1(C_2)$, hence on $\pi_1(P)$.

The restriction $\hat p$ of $p'$ to $\pi_1(\hat \Sigma)$  is a \bpm. It is associated to a splitting $\hat \Gamma$ having one less vertex group than $\Gamma$: the vertex groups $B_1$ and $B_2$ have been replaced by  a single vertex group isomorphic to $B_1*B_2$, in particular the base of $\hat \Gamma$ is isomorphic to  that of $\Gamma$. 
  The   non-degeneracy of $\hat p$ follows from the fact that its image contains those of the groups $\pi_1(C'_i)$, which are non-trivial and contained in conjugates of $B_i$.

If $\hat \Gamma$ is not simple, we can 
  iterate this process and obtain a \ndbpm\ associated to a simple splitting. There remains to check that the corresponding surface $\Sigma_f$ is non-exceptional when $\chi(\Sigma)\le-3$.  
  
 We denote by $g,b,n$ the genus of $\Sigma$, the number of boundary components of $\Sigma$, and the number of bottom vertices of $\Gamma$  respectively. Retractability of $\Gamma$ implies $g+b\ge2n$ (\cite{GLS_finite_index}, see Remark \ref{ineg}).
   By construction $\Sigma_f$ has genus $g$ and $b-n+1$ boundary components. We assume that $\Sigma_f$ is exceptional (so $\chi(\Sigma_f)=-1$ and $g\le2$), 
   and we deduce that $\chi(\Sigma)=-2$.
  
  First suppose that $\Sigma$ is orientable. This implies  $g=0$  
  since non-planar orientable surfaces are non-exceptional. Then $b\ge 2n$ and 
  $1=| \chi(\Sigma_f) | = b-n-1\ge   n-1$, so $n=2$,  $b=4$, and $\Sigma$ is a 4-punctured sphere.
  
  If $\Sigma$ is non-orientable with $g=1$, we have  $1=| \chi(\Sigma_f) | = b-n\ge   n-1$, so $n=2$, $b=3$, and $\Sigma$ is a   thrice-punctured projective plane.
  
  If $\Sigma$ is non-orientable with $g=2$, we have  $1=| \chi(\Sigma_f) | = b-n+1\ge   n-1$, so $n=b=2$ and $\Sigma$ is a   twice-punctured Klein bottle.
\end{proof}

\begin{figure}[ht!]
  \centering
  \includegraphics{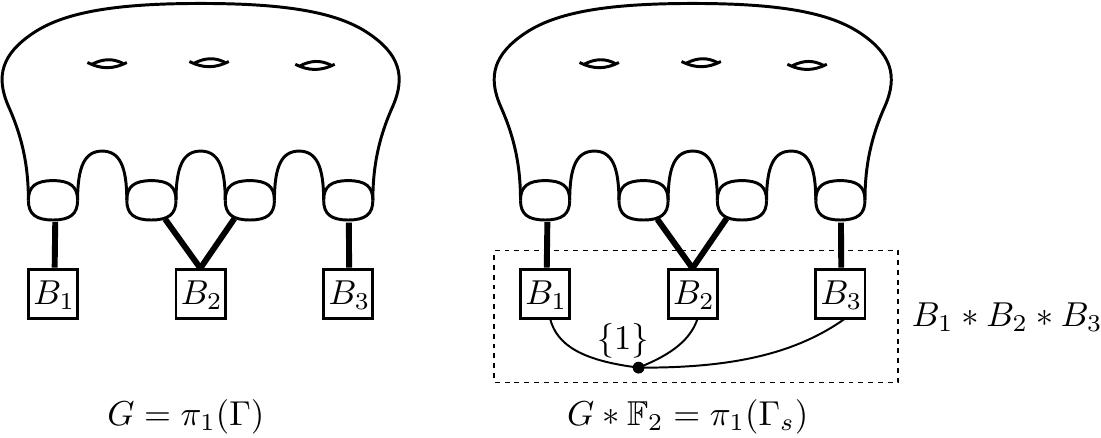}
  \caption{If $G$ is an extended étage with bottom groups $B_1$, $B_2$, $B_3$, then $G*\bbF_2$ is a simple étage over a subgroup isomorphic to $B_1*B_2*B_3$.}
  \label{fig_ajout_arc}
\end{figure}

\begin{prop}[See Figure \ref{fig_ajout_arc}] \label{tstab}
Let $\Gamma$ be a centered splitting of $G$. 
  If 
it is  retractable, 
  there exists a (possibly trivial) 
  free group $F$ and a  {simple} retractable splitting $\Gamma_s$ of  $G*F$,
  which makes $G*F$ a simple \'etage.  
  
  The splittings $\Gamma$ and $\Gamma_s$ are associated to  the same surface.   The base of $\Gamma_s$ is isomorphic to the base of $\Gamma$, except that $\Gamma_s$ has base   $B_1*\Z$ if   $\Gamma$ is simple with an abelian base $B_1$. 
\end{prop}

Conversely, we will see 
  (Corollary 
  \ref{redu}) that $G$ has a retractable splitting whenever some  $G*F$ has one.

\begin{proof}
There is nothing to prove if $n=1$, except if $\Gamma$   has an abelian base   $B_1 $. But then $G*\Z$ is a simple \'etage over $B_1*\Z$.

Now assume $n\ge2$ (see Figure \ref{fig_ajout_arc}). Let $\rho:G\to H=\tilde B_1*
\dots*\tilde B_n $ be a retraction provided by Proposition \ref{equivpassimple}. 
Write $\bbF_{n-1}=\langle t_2,\dots, t_{n}\rangle$  and define $\varphi:H\to G*\bbF_{n-1}$ as being the identity on $\tilde B_1$ and conjugation by $t_i$ on $\tilde B_i$. It induces an isomorphism (also denoted by $\varphi$) between $H$ and a subgroup $H'= 
\tilde B_1 *\tilde B_2^{t_2}*\cdots*\tilde B_{n}^{t_{n}} \inc G*\bbF_{n-1}$, with $\tilde B_i^{t_i}=t_i\tilde B_it_i\m$. 

The group $H'$ is the base of a simple 
centered splitting of $G*\bbF_{n-1}$, whose restriction to $G$ is $\Gamma$. The map 
$\rho':G*\bbF_{n-1}\to H'$ equal to $\varphi\circ\rho$ on $G$ and killing every $t_i$
is a retraction     from $G*\bbF_{n-1}$ to $H'$.
\end{proof}

\begin{rem} \label{bassimp}
If  $\Gamma$  
is not simple,   this construction makes $G*F$ a  simple  \'etage over a  subgroup 
$H'$ which is isomorphic to $H$ but is not contained in $G$.
\end{rem}

\begin{cor} \label{ett}
 If $G$ is an extended tower over a group $H$,  
  there exists a free group $F$ such that $G*F$ is a simple tower over   a subgroup $H'$ isomorphic to $H*\bbF_2$.  If $H$ is not abelian, one can take $H'\simeq H$.
\end{cor}

\begin{proof}
By  Proposition \ref{prop_retractions},
there exists  a tower  $G=G_0> G_1\dots> G_k=H*F'$ where $G_i$ is an extended étage of surface type over $G_{i+1}$ and $F'$ is free.

  Assume first that $H$ is not abelian.
Then no $G_i$ is   abelian, so for all $i<k$ there is a free group $F_i$ such that $G_i*F_i$ is a simple tower   over a subgroup $G'_{i+1}$ isomorphic to $G_{i+1}$.
Applying Remark \ref{rk_prodlib} inductively, one finds a free group $F$ such that $G*F$ is a simple tower over a   subgroup isomorphic to $H*F'$, hence over 
a  subgroup isomorphic to $H$.   

If   $H $ is abelian, then 
 $G*\bbF_2$
is an extended tower over   the non-abelian group $H*\bbF_2$ 
by Remark \ref{rk_prodlib},  and the argument given above applies.
\end{proof}

\begin{rem}\label{rk_tour_rel}
 In Corollary \ref{ett}, 
  assume  that the tower is relative to a subgroup $A<H$,
   and let $H=H_1*\dots *H_k* \F_r$ be a Grushko decomposition of $H$ relative to $A$. At each \'etage, every $H_j$ is contained in a conjugate of a bottom group.
   It follows from the way that $\varphi$ was defined in the proof of Proposition \ref{tstab} that there is an isomorphism $\psi:H\to H'$ whose restriction to each $H_j$ agrees with the conjugation by some element of $G*F$ (depending on $i$);   
in particular $H'$ contains a conjugate of $A$   (and we may choose $H'$ so that it contains $A$).  If the tower is relative to $H$ itself, then $H'$ is conjugate to $H$ in 
$G*F$.

 By Remark \ref{ttrel},
 the tower is automatically relative to $A$ if $A<H$ is one-ended; thus \emph{any 
one-ended $A<H$ is  contained in  $H'$ up to conjugacy.} 
This also holds if $A$ is one-ended relative to a subgroup $A'<A$ and the tower is relative to $A'$.
\end{rem}

The following lemma will be used in Section \ref{interp}.

\begin{lem}\label{lem_malnormal_tour} 
 Assume that $G$ is an extended tower over a subgroup $H<G$  relative to $A<H$, and let $H=H_1*\dots *H_k*\bbF_r$ be the Grushko decomposition of $H$  relative to $A$.

Then  $(H_1,\dots,H_k)$ is a malnormal family in $G$: if $H_i^g\cap H_j\neq \{1\}$, then $i=j$ and $g\in H_i$.
\end{lem}

\begin{proof}
  First assume that the tower $G=G_0> G_1 > \dots> G_n=H$ is simple.
  Then $G_i$ is malnormal in $G_{i-1}$ for each $i\geq 1$  by Remark \ref{malnor}. It follows that $H$ is malnormal in $G$,
  and the lemma follows from the malnormality of the family of Grushko factors of $H$.

  We now turn to the general case.  
  By Corollary \ref{ett} and Remark \ref{rk_tour_rel}, $G*F$ is a simple tower over a group $H'$, and there is an isomorphism  from $H$ or $H*\F_2$ to $H'$ 
  whose restriction to each $H_j$ is inner. The result follows from the case of a simple tower.
\end{proof}

\subsection{Narrowing the range}
\label{rff}

In this subsection we consider a subgroup   $H\inc G$ and a     splitting of $H$. We assume that the splitting  is retractable in $G$   (Definition \ref{dfn_retractable}), and we wish to show that it is retractable in $H$. Our results are based on the following lemma.

\begin{lem} \label{propschloe2} 
Let  $\pi_1(\Sigma)\inc 
H\inc G$, with 
  $H$ malnormal. 
Let  
  $p:\pi_1(\Sigma)\to G$ be a non-degenerate \bpm, 
  and let $\calc $ be a maximal family of pinched curves, with $\calc\ne\es$.  
  
Assume that, whenever $\Sigma_\calc^j$ (as in Definition \ref{pc}) contains a component of $\partial \Sigma$, the image of $\pi_1(\Sigma_\calc^j)$ is contained in a conjugate of $H$. Then there is a pinching non-degenerate \bpm\
$p':\pi_1(\Sigma)\to H$.
\end{lem}

\begin{proof}
 We factor $p$ into $p=p_\calc\circ\pi_\calc$ with $\pi_\calc:\pi_1(\Sigma)\to\pi_1(\Sigma_\calc)$ and $p_\calc:\pi_1(\Sigma_\calc)\to G$ (see Definition \ref{pc}). 
The group $\pi_1(\Sigma_\calc)$ is  isomorphic to a free product whose factors are closed   surface groups, cyclic groups, and the $\pi_1(\Sigma_\calc^j)$'s for which  $\Sigma_\calc^j$ contains a component of $\partial \Sigma$; we fix such an isomorphism. 
We shall modify $p_\calc$ on each free factor, and $p'$ will be $p_\calc\circ\pi_\calc$.

  We can modify $p_\calc$ on each $\pi_1(\Sigma_\calc^j)$, by composing with a conjugation of 
$G$, so that the image of $\pi_1(\Sigma_\calc^j)$ is contained in $H$ (rather than a conjugate);   by malnormality of $H$, the modified  $p_\calc \circ\pi_\calc$
 acts on each boundary subgroup of $\Sigma$ as 
conjugation by some element of $H$.
We can also modify $p_\calc$ on the other free factors so that the image of $\pi_1(\Sigma_\calc)$ is contained in $H$. 
We now have to make sure that this
 image 
  is   non-abelian.

If $\Sigma$ has two boundary components $C_1,C_2$ belonging to   different surfaces $\Sigma_\calc^{j_1} $, $\Sigma_\calc^{j_2} $, the images of  $A_1=\pi_1(C_1)$ and $A_2=\pi_1(C_2)$ are conjugate to subgroups of   $\pi_1(\Sigma)$ (in $G$ hence in $H$ by malnormality);
  we 
can modify $p_\calc$ on $\pi_1(\Sigma_\calc^{j_1} )$
so that  $p_\calc(A_1)$ and $p_\calc(A_2)$ are two non-commuting conjugates of $A_1$ and $A_2$ in $H$
(for instance   by mapping   $A_1$ and $A_2$ to themselves in $\pi_1(\Sigma)\subset H$,  since they do not commute there).

If $\partial \Sigma$ is contained in a single $\Sigma_\calc^j$, and $\pi_1(\Sigma_\calc^j)$ has   abelian image,    we argue as in the proof of Proposition \ref{equivsimple}.  Recall that $G$ is torsion-free and  $\pi_1(\Sigma_\calc)$ has   non-abelian image under the original $p_\calc$. 
Consider the factors in the decomposition of $\pi_1(\Sigma_\calc)$ as a free product. Besides $\pi_1(\Sigma_\calc^j)$, there must be at least a $\Z$ or the group of a closed surface different from $S^2$ and $P^2$. Such a surface group maps onto $\Z$, and we can modify $p_\calc$ to make its image   non-abelian.
\end{proof}

We can now state:

\begin{prop}
\label{cinqdouze2}
Assume that 
$H$ is a free factor of $G$. 
If a splitting of $H$ is retractable in $G$, then it is retractable in $H$.   

\end{prop}

We view 
this as a special case of the following  more precise result, which will be needed later.
 
\begin{lem}\label{cinqdouze}
Let  $\pi_1(\Sigma)\inc   H\inc G$, with 
$H$ a free factor of $G$. 
If   there is a non-degenerate \bpm\ 
$p:\pi_1(\Sigma)\to G$, 
there is a non-degenerate \bpm\  
$p':\pi_1(\Sigma)\to H$. 
Moreover,    $p'$ is pinching or $p'=\iota\circ p$   with $\iota$ an inner automorphism of $G$.

\end{lem}

\begin{proof}
If $p$ is non-pinching, its image is contained in a conjugate of $H$ by Lemma \ref{lem_pince} and we take $p'=\iota\circ p$ for some
inner automorphism $\iota$ of $G$. Otherwise 
we apply Lemma \ref{propschloe2}, noting that
the image of every $\pi_1(\Sigma_\calc^j)$ 
is contained in a conjugate of $H$ by Lemma \ref{lem_pince}.
 \end{proof}

We   now consider the following situation.
Let $\Gamma$ be a retractable  centered splitting of   
$G$, 
and let $B$ be a bottom vertex group.
  Given a surface subgroup $\pi_1(\Sigma )\subset B$, and a \ndbpm\ $p:\pi_1(\Sigma)\ra G$, we wish to construct a \ndbpm\   $p'$ with values in $B$.

  \begin{figure}[ht!]  
    \centering
    \includegraphics{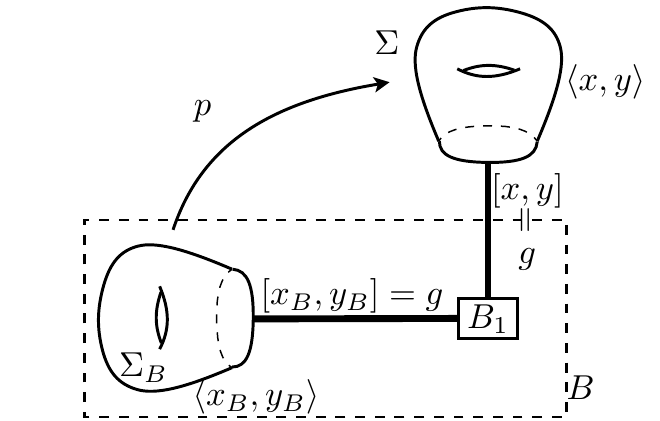}
    \caption{Example \ref{bete}: there is a \ndbpm\ $p:\pi_1(\Sigma_B)\ra G$, but none from $\pi_1(\Sigma_B)$ to $  B$.}
    \label{fig_twin0}
  \end{figure}

  The following simple example shows that  additional assumptions are required.

\begin{example}[see Figure \ref{fig_twin0}] \label{bete}
As in Example \ref{expt}, we consider a group   $B =B_1*_{g=[x_B,y_B]}F(x_B,y_B)$  obtained by attaching a punctured torus $\Sigma_B$ to a group $B_1$, but we now require that   $g$   is \emph{not} a commutator in $B_1$. Construct $G=B*_{g=[x,y]}F(x,y)$ by attaching a second punctured torus $\Sigma$ to $B_1$ in the same way as $\Sigma_B$. 
There is a \ndbpm\   $p:\pi_1(\Sigma_B)\to G$ with image $\pi_1(\Sigma)$, but no \ndbpm\   $p:\pi_1(\Sigma_B)\to B$ since $g$ is not a commutator in $B_1$. 
\end{example}

 \begin{lem} \label{desc}
Let   $\Gamma $ be a retractable centered splitting of $G$, 
 with central vertex group   $\pi_1(\Sigma )$,
and let 
  $B$ be a bottom vertex group. 
Let $\pi_1(\Sigma_{B})<B$ be   a surface subgroup, and let
  $p:\pi_1(\Sigma_{B})\to G$ be a \ndbpm.

 If $p$ satisfies one of the following conditions, there is a \ndbpm\ $p':\pi_1(\Sigma_{B})\to B$: 
\begin{enumerate}
\item no conjugate of the image of $p$  contains a finite index subgroup of $\pi_1(\Sigma )$;
\item $p$ is pinching. 
\end{enumerate}

Moreover,    $p'$ is pinching or $p'=\iota\circ p$   with $\iota$ an inner automorphism of $G$.
\end{lem}

\begin{proof}

Assume 1.
If $p$ does not pinch, then  
 it  
takes values in a conjugate of $B$ by  Lemma  \ref {lemcle2}, 
so we can   take   $p'=\iota\circ p$ for some inner automorphism $\iota$. 
If $p$ is pinching, 
consider a maximal collection $\calc$  of pinched curves on $\Sigma_B$.
Lemma \ref{lemcle2} ensures that   the fundamental group of every connected component 
of $\Sigma_B\setminus \calc$
is mapped into a conjugate of a bottom group of $\Gamma $.
One can then apply Lemma \ref{propschloe2}:
 if such a component contains a component of   $\partial \Sigma_B$,   Remark \ref{malnor}
implies that this bottom group has to be $B$ since $\pi_1(\Sigma_B)\subset B$ and  $p$ agrees with conjugations 
on the boundary subgroups.

If  1 does not hold but $p$ is pinching, we define  $\hat p=\rho\circ p$, with $\rho:G\to 
\tilde B_1*\dots * \tilde B_n$ a retraction associated to $\Gamma $  
provided by Proposition \ref{equivpassimple}. The  image of $\hat p$  is   non-abelian because it is CSA and
contains a finite index subgroup of the non-abelian group  $\rho(\pi_1(\Sigma ))$,  
so $\hat p$ is a pinching \ndbpm\ with values in $\tilde B_1*\dots * \tilde B_n$.
The group $B$ is conjugate to some $\tilde B_i$, and 
we   apply  Lemma 
\ref{cinqdouze} to find $p'$ with values in  
$B$.
\end{proof}

 \subsection{Retracting a Grushko factor}\label{sec_grushko}

In this subsection we introduce a blowup of a centered splitting that will   enable us to understand how a  freely decomposable group may be an \'etage of surface type. In particular, we will show that, if $G$ is an extended \'etage of surface type, so is one of its Grushko factors.
  As a corollary,  we will see that a free group   has no retractable splitting (this is a main step in the proof
of homogeneity of free groups \cite{PeSk_homogeneity}).

 \begin{figure}[htbp] 
  \centering
  \includegraphics[width=.8\textwidth]{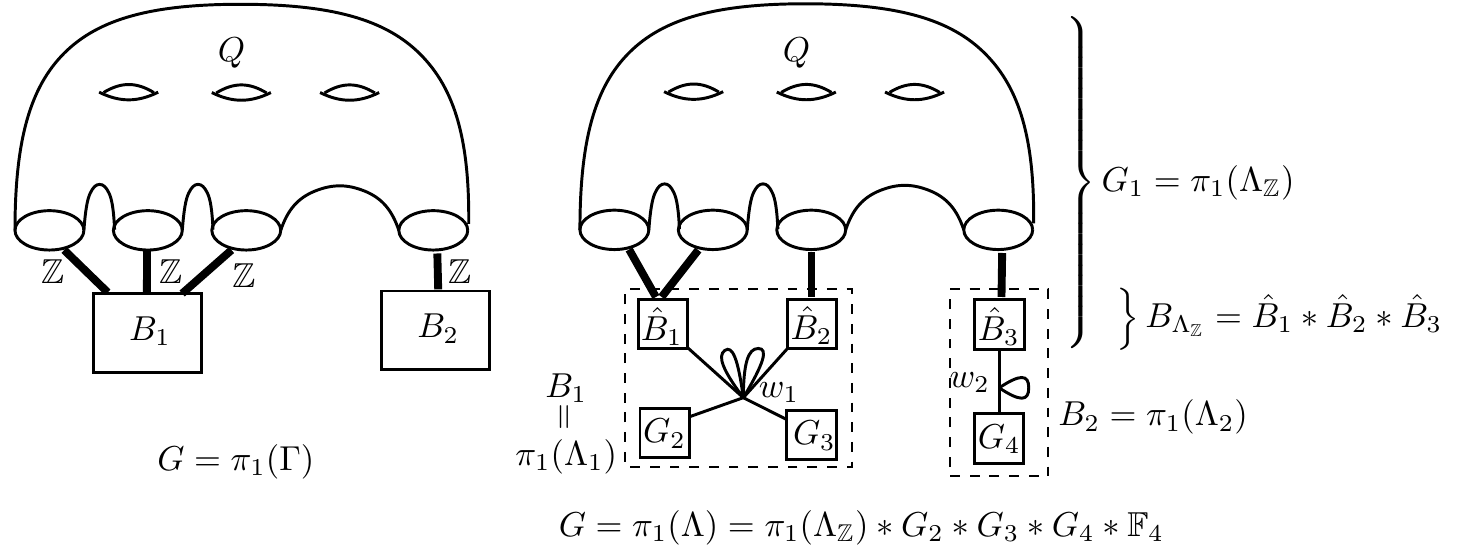}
 \caption{The Grushko blowup $\Lambda$ of $\Gamma$,
and the resulting centered splitting $\Lambda_\Z$ with fundamental group $G_1$ and   base   $B_{\Lambda_\Z}$.}
\label{fig_lesGrushko} 
\end{figure}

\begin{dfn}[Grushko blowup of a centered splitting, see Figure \ref{fig_lesGrushko}] \label{leslambda}
Let $\Gamma$ be a centered splitting of $G$ with bottom vertex groups $B_1,\dots, B_n$ and central group $Q$.
The \emph{Grushko blowup} of $\Gamma$ is the splitting $\Lambda$ obtained by
blowing up every bottom vertex of $\Gamma$ using a Grushko decomposition $\Lambda_i$ of $B_i$ relative to the incident edge groups (see Subsection \ref{gd}). 

We denote 
by $\Lambda_\Z\subset \Lambda$ the union of all edges   carrying a non-trivial group 
(the star of the central vertex).
The fundamental group of $\Lambda_\Z$ is a free factor $G_1$ of $G$ containing $Q$. 
\end{dfn}

The definition is illustrated in Figure \ref{fig_lesGrushko},
 which depicts a case where there are two bottom  vertex groups $B_1,B_2$. 
If, for  
some $i$, the group $B_i$ is freely decomposable relative to the incident edge groups, we choose a relative Grushko decomposition 
$\Lambda_i$  with a central vertex $w_i$ carrying a trivial group joined by a single edge to each  vertex with non-trivial group. 
Some of these vertices (carrying $\hat B_1$, $\hat B_2$, $\hat B_3$ on the figure) have  at least one  edge of $\Lambda_\Z$ attached to them, 
the base $B_{\Lambda_\Z}$  of $\Lambda_\Z$ is the free product of their groups. The other vertex groups of the $\Lambda_i$'s ($G_2$,\dots,$G_4$ on the figure) are Grushko factors of $G$. There may  also be loops attached at $w_i$, corresponding to the free group in the relative Grushko
decomposition of $B_i$. The edges with trivial group induce a decomposition of the form $G=G_1*G_2* \dots*G_k*F$
with $G_1=\pi_1(\Lambda_\Z)$  and $F$ free (of rank 4 on Figure \ref{fig_lesGrushko}).

Note that the splitting $\Lambda_\Z$  fails to be 
minimal if,
 for some $i$, 
  an edge $e$ of $\Gamma$ containing $v_i$ carries a group $A_e \simeq\Z$ such that  $B_i$ is a free product $B_i=A_e*B'_i$, and the groups carried by   all other   edges of $\Gamma$ containing $v_i$ are contained (up to conjugacy) in the complementary free factor $B'_i$.
If this does not happen, $\Lambda_\Z$ is minimal and is a centered splitting of $G_1$. 

The main result  of this subsection is   the following.

\begin{prop} \label{Grus}
Let $\Gamma$ be a   centered splitting of $G$, with central vertex $Q$ and base $B$. 
If $\Gamma$ is retractable, then: 
\begin{enumerate}
\item the group $G_1=\pi_1(\Lambda_\Z)$ is a Grushko factor of $G$ containing $Q$ 
(i.e.\  $G_1$ is one-ended);
\item $\Lambda_\Z$ is minimal, and is a retractable centered splitting of $G_1$;

\item there is a Grushko decomposition $G=G_1*G_2*\dots* G_k*F$ 
such that $$B\simeq  B_{\Lambda_\Z}*G_2*\dots* G_k*F',$$
with $ B_{\Lambda_\Z} $ the  base of $\Lambda_\Z$ 
 (possibly not   one-ended) and
  $F'$ a free factor of the free group  $F$. 

\end{enumerate}
\end{prop}

\begin{rem}\label{rem_Grushko}
By Assertion 1,  collapsing all edges of $\Lambda$ with non-trivial stabilizer yields a Grushko decomposition of $G$.
\end{rem}

We derive several consequences of the proposition before proving it, starting with two obvious corollaries. 

\begin{cor} \label{infbout}
If $G$ has a retractable splitting, so does one of its Grushko factors. \qed
\end{cor}

\begin{cor}[\cite{Perin_elementary, PeSk_homogeneity}] \label{paset}
  A free group is not an extended   étage of surface type. 
\qed
\end{cor}

If $G$ has a Grushko decomposition $G=G_1*G_2*\dots* G_k*F$, define $r(G)$ as $G_1*G_2*\dots* G_k$.

\begin{cor} \label{redu2}
If $G$ is an extended tower over $B$, then
$r(G)$ is an extended tower over a subgroup  isomorphic to $r(B)$.
 \end{cor}

\begin{proof}
 It suffices to prove the  result for an \'etage. It is clear if $G=B*\Z$.  If $G$ is an extended \'etage of surface type over $B$,   Proposition \ref{Grus} says that $G_1$ is an extended \'etage over $B_{\Lambda_\Z}$,
so by Remark  \ref{rk_prodlib}  $r(G)=G_1*G_2*\dots* G_k$ is 
an extended \'etage of surface type over $B_{\Lambda_\Z}*G_2*\dots* G_k $, which is of the form $r(B)*F''$ with $F''$ free.

\end{proof}

  The following result is a partial converse to Proposition \ref{tstab}.
  
\begin{cor} \label{redu}
If   $G*\bbF_r$ is an   
extended tower  over a group $B$ for some $r\geq 1$, 
then $G$ is an extended 
tower   over a group $B'$  such that $B\simeq B'*\bbF_s$ for some $s\geq 0$. \qed
\end{cor}

\begin{proof}
 By the previous corollary, $r(G)=r(G*\F_r)$  is an extended tower over $B':=r(B)$, and $G$ is an extended tower over $r(G)$.
\end{proof}

The rest of this subsection is devoted to the proof of Proposition \ref{Grus}. 
The key step is the following fact.

\begin{prop}\label{unbout}
Let $\Gamma$ be a   centered splitting of $G$. If it is retractable, 
 the group $G_1$ of Definition \ref{leslambda} is one-ended.
\end{prop}

\begin{proof}[Proof of Proposition \ref{Grus} from Proposition \ref{unbout}]
 Since 
 $G_1$ is one-ended, 
 it is a Grushko factor of $G$. 
 
 Minimality of $\Lambda_\Z$ will be proved in Lemma \ref{sgrush}.  Since $\Gamma$ is retractable, there is a \ndbpm\ from $Q$ to $G$. By Lemma \ref{cinqdouze} there is one with values in $G_1$, so $\Lambda_\Z$ is a retractable splitting of $G_1$.

The structure of $B$ follows from the above description of the Grushko blowup;
the rank  of $F'$ is the sum of the first Betti numbers of the graphs $\Lambda_i$ (the total number of   loops at the vertices $w_i$);  in the rank of $F$ there is  
an extra term equal  to  the number of bottom groups of $\Lambda_\Z$  minus the number of bottom groups of $\Gamma$.
\end{proof}

We now prove Proposition \ref{unbout}. This
 requires several lemmas.

\begin{lem} \label{sgrush}

If a   centered splitting $\Gamma$ of $G$ is retractable,
  the splitting  $\Lambda_\Z$ of Definition \ref{leslambda} is minimal.
\end{lem}

\begin{proof}
We assume that $\Lambda_\Z$ is not minimal and we argue towards a contradiction. 
As explained above, there is an edge $e=vv_i$ of $\Gamma$ 
  carrying a group $A_e$ such that  $B_i=A_e*B'_i$, and the groups carried by   all other   edges of $\Gamma$ containing $v_i$ are contained (up to conjugacy) in $B'_i$.

Let $C$ be the  boundary component   of $\Sigma$ associated to the edge $e$. 
Fixing  a   \ndbpm\ $p$, we let $\calc$ be  a maximal family of pinched curves on $\Sigma$, and we consider    the component 
$\Sigma'$ of $\Sigma\setminus\calc$ containing $C$. By Lemma \ref{lemcle4}   (and Remark \ref{vi}) 
$p$ maps $\pi_1(\Sigma')$ into  $B_i$ (up to conjugacy).
Moreover, $\pi_1(C)$ is mapped surjectively onto a conjugate of  $A_e$, 
and if $C'$ is another   component of $\Sigma'\cap\partial\Sigma$   then $\pi_1(C')$ is mapped into a conjugate of  $B'_i$. Projecting $B_i$ to
 $A_e 
\simeq \Z$ by killing $B'_i$, we 
  deduce that a generator of $\Z$ is a product of commutators or squares (depending on whether $\Sigma'$ is orientable or not), a contradiction.
\end{proof}

 The following remark will be used in Subsections  \ref{mbs} and \ref {agpe}.

\begin{rem}  \label{bolib}  
Assume that a cyclic splitting $\Gamma$ has a vertex $v$ which is as in Definition \ref{stype} (surface-type vertices), except that the map from the set of incident edges to the set of boundary components of $\Sigma$ is injective but not surjective; in other words, $\Sigma$ has at least one ``free'' boundary component
 (this  may occur 
at a flexible vertex of a  JSJ decomposition relative to a cyclic subgroup $A$, 
with a free boundary component 
whose  
fundamental group 
contains a conjugate of $A$).
The   argument used to prove the lemma then says that there is no \ndbpm\ $p:\pi_1(\Sigma)\ra G$.
In particular, if $\Gamma$ is a  centered splitting which is not minimal (thus contradicting our standing assumption), it is not retractable.
\end{rem}

\begin{lem} \label{passurf}
Let $\Gamma$ be a centered splitting of $G$ with central vertex group $Q$. 

If $J$ is a finitely generated one-ended subgroup of $G$, then some conjugate of $J$ either is contained in one of the bottom vertex groups, 
or contains a finite index subgroup of $Q$.
\end{lem}

\begin{proof}

If $J$ is elliptic in the Bass-Serre tree $T$ of $\Gamma$, it is contained (up to conjugacy) in a bottom group $B_i$ (not in $Q$, which is a free group). If not, the action of $J$ on its minimal subtree defines a non-trivial splitting $\Gamma_J$. Using conjugations, we may assume that $\Gamma_J$ has a vertex group $J\cap Q$, with an incident edge group of the form $J\cap C$ where $C$ is a maximal boundary subgroup of $Q=\pi_1(\Sigma)$. Since $J$ is one-ended, $J\cap C$ is non-trivial. It is maximal cyclic in $J\cap Q$ 
because $C$ is maximal cyclic in $Q$.
By 1-acylindricity of $T$ and minimality of $\Gamma_J$, we have $J\cap C\ne J\cap Q$. Lemma \ref{Scott} then implies that $J\cap Q$ has finite index in $Q$. 
\end{proof}

\begin{rem} \label{passurf2}
If $J$ is not elliptic in $T$, the proof shows that $J$ has a cyclic splitting with a surface-type vertex whose associated  surface   is  a finite cover of $\Sigma$.
\end{rem}

\begin{lem} \label{facgru}
Let $\Gamma$ be a   centered splitting of $G$. 
  If each bottom group $B_i$ is one-ended 
  relative to the incident edge groups,
then $G$ is one-ended. 
\end{lem}

\begin{rem}
  In this lemma, the implicit assumption that $\Gamma$ is a minimal splitting is essential.
\end{rem}

\begin{proof}
By \cite[Lemma 6.11]{Horbez_boundary}, 
if all Grushko   factors of $G$ are elliptic in $\Gamma$, then some vertex group of $\Gamma$ is freely decomposable 
relative to its incident edge groups,
contradicting our assumption.
Thus, there is a Grushko factor $J$ of $G$ which is not elliptic in $\Gamma$. Lemma \ref{passurf} implies that (up to conjugacy) $J$ contains a finite index subgroup of $Q$, hence contains $Q$. 
 This easily  implies that $G$ is one-ended.
 
 Indeed, $Q$ is elliptic in any   $G$-tree $S$ with trivial edge stabilizers, and so are edge groups of $\Gamma$,   as well as  the groups $B_i$ because of the assumption of the lemma. Since edge stabilizers of $S$ are trivial, adjacent vertex stabilizers of the Bass-Serre tree of $\Gamma$ must fix the same vertex of $S$, so $S$ must be trivial.
\end{proof}

Proposition \ref{unbout} follows by applying Lemma \ref{facgru} to the splitting $\Lambda_\Z$ of $G_1$, which is minimal by Lemma \ref{sgrush}.

\subsection{Reading an étage in the JSJ}
\label{readjsj}

  In this subsection it is important to assume that $G$ is a CSA group. We also assume that $G$ is one-ended. It then has a  
canonical cyclic JSJ splitting $\Gcan$ (associated to 
 the tree $(T_a)^*_c$ of Theorem 9.5 of  \cite{GL_JSJ},  see Subsection \ref{JSJ}  and \cite{GL_JSJ}).   

Now let  $\Gamma$ be a centered splitting of $G$, with central vertex $v$ and central vertex group $Q=\pi_1(\Sigma)$. The main result of this subsection is Lemma \ref{preconf}, saying that  $\Gcan$ is retractable if $\Gamma$ is retractable (provided that $G$ is not a closed surface group).
Recall (Definition \ref{dfn_retractable}) that a splitting  is retractable if one of its non-exceptional surface-type vertex  groups has a non-degenerate \bpm.

We first explain how to obtain $\Gamma$ from $\Gcan$  (see Figure \ref {fig_sous-surface}).
Because $G$ is one-ended, the central vertex group $Q=\pi_1(\Sigma)$ is elliptic in $\Gcan$ (see  Theorem 5.27 of \cite{GL_JSJ}). It is contained (up to conjugacy) in a vertex group $G_w$ of  $\Gcan$ which is neither rigid 
nor abelian,  so is surface-type. Let $S$ be the associated surface, so that $G_w=\pi_1(S)$.

\begin{figure}[ht!]
  \centering
  \includegraphics{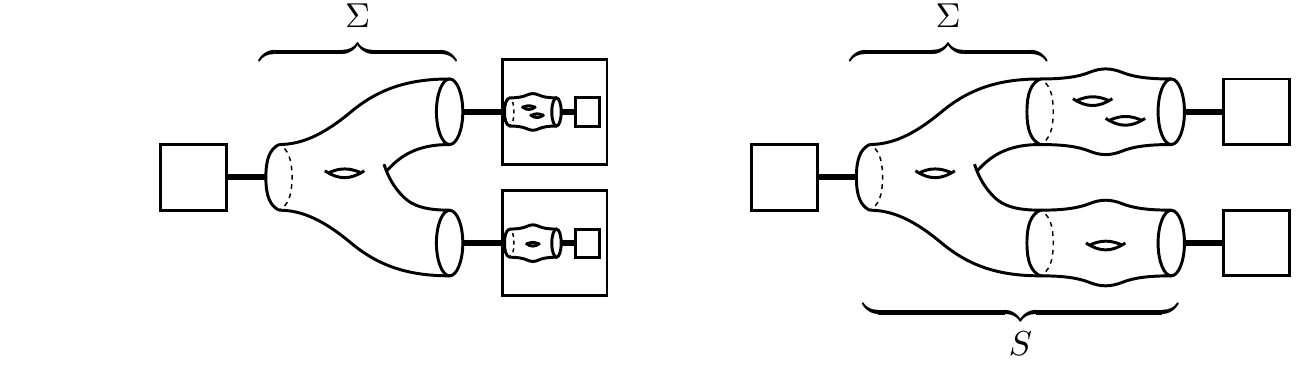}
  \caption{The surface   $\Sigma$ of the centered splitting $\Gamma$  (pictured on the left) is contained in the surface   $S$ of the JSJ splitting $\Gcan$  (on the right).}
  \label{fig_sous-surface}
\end{figure}

   The splitting $\Gcan$ is universally compatible (see \cite{GL_JSJ}): there is a splitting   $\hat \Gamma$ which collapses onto both $\Gamma$ and $\Gcan$. We may assume that $\hat \Gamma$
is  the lcm of $\Gamma$ and $\Gcan$: 
no edge is collapsed in both (see \cite{GL_JSJ}, Section A.5). The group $Q$ is elliptic in $\Gamma$ and $\Gcan$,  hence contained in a vertex group $G_{\hat v}$ of $
\hat \Gamma$. We claim that edges  of $\hat \Gamma$ incident on $\hat v$ are not collapsed in $\Gamma$.

To prove the claim, we work with the Bass-Serre trees and we denote by $\pi:T_{\hat\Gamma}\ra T_\Gamma$ 
the collapse map. We view $\hat v, v$ as vertices of $T_{\hat\Gamma}, T_\Gamma$, with $\pi(\hat v)=v$.
Consider the subtree $Y=\pi\m(v)$ of $T_{\hat \Gamma}$.
Let $\cale$ be the set of edges $e=v_ew_e$ 
 of $T_{\hat \Gamma}$  with $v_e\in Y$ and $w_e\notin Y$ which are not collapsed by $\pi$  (it may be identified with the star of $v$ in $T_\Gamma$).
By minimality of $T_{\hat \Gamma}$, $Y$ is the convex hull of $\{v_e\mid e\in\cale\}$.

 Note that $G_e$ fixes the segment $[\hat v,v_e]$ for $e\in \cale$.
Since distinct edges of $T_\Gamma$ incident on $v$ have non-commuting stabilizers,
this is also the case for any pair of distinct edges $e,e'\in \cale$, so 
$e,e'$ cannot be in the same connected component of $T_{\hat\Gamma}\setminus \{\hat v\}$.
If $e\in\cale$ and $v_e\neq \hat v$, minimality of $T_{\hat \Gamma}$ implies that 
there is no edge incident on the open segment $(\hat v,v_e)$,
and $e$ is the only edge incident on $v_e$ not contained in $[\hat v,v_e]$.
Since there is no redundant vertex in $T_{\hat \Gamma}$, it follows that   $v_e=\hat v$ for all $e\in \cale$, so $Y=\{\hat v\}$. This    proves the claim: edges   of $\hat \Gamma$ incident on $\hat v$ are not collapsed in $\Gamma$.

On the other hand, the     preimage of $w$ in $\hat \Gamma$ is a graph of groups decomposition of $\pi_1(S)$, and the  minimal subgraph of groups  is dual to a family of disjoint simple closed curves on $S$ by a standard result (\cite[Theorem III.2.6]{MS_valuationsI}). The group $Q$ is a vertex group of this splitting (it is carried by $\hat v$), so $\Sigma$ may be identified with a subsurface of $S$  (see Figure \ref{fig_sous-surface}). In particular, $S$ is non-exceptional if $\Sigma$ is non-exceptional  (unless $G=\pi_1(N_3)$). 

To sum up:

\begin{lem} \label{compajsj}
One may obtain any centered splitting
$\Gamma$ from $\Gcan$ as follows:  one selects a surface-type vertex $w$ of $\Gcan$ and  an incompressible subsurface $\Sigma$ of the associated surface $S$, one refines $\Gcan $ by blowing up $w$ using a splitting of $G_w=\pi_1(S)$ having a vertex $v$ carrying $\pi_1(\Sigma)$, and one collapses all edges of the refined splitting not containing $v$. \qed
\end{lem}

Since the set of incompressible surfaces of a given surface $S$ is finite up to the action of the modular group of $S$, we obtain the following finiteness result. 

\begin{cor} \label{fini-1}
 Let $G$ be a one-ended torsion-free CSA group. Up to isomorphism, the set of groups $B$ such that $G$ has a centered splitting with base isomorphic to $B$ is finite. \qed
\end{cor}

Keeping the same notations as in Lemma \ref{compajsj}, we now compare the base    $B_\Gamma$ of  the   centered splitting $\Gamma$ to   $B_{S}$, 
the base  of the centered splitting    obtained from   $\Gcan$  by collapsing all edges not containing   $w$. 

\begin{lem} \label{bases} 
$B_\Gamma$ is   isomorphic to the free product of $B_{S}$ with a free group. 
\end{lem}

\begin{proof}
   See Figure \ref{fig_sous-surface}.
View   $\hat\Gamma$ as a graph of spaces $X$, so that $B_\Gamma$ (resp.\ $B_{S}$) is the free product of the fundamental groups
of the connected components of $X\setminus \Sigma$ (resp.\ $X\setminus S$).
Each connected component of $X\setminus \Sigma$ is a union of connected components of $X\setminus S$,
attached by surfaces with boundary. Since the closure of each component of $S\setminus\Sigma$ contains a component of $\partial \Sigma$, 
the fundamental group of each connected component of $X\setminus \Sigma$ is the free product of the fundamental groups
of the connected components of $X\setminus S$ which it contains, and a free group.
 \end{proof}

\begin{lem}\label{preconf}
Let $G$ be a one-ended torsion-free CSA group,   not a closed surface group. If  $G$ has a retractable cyclic splitting $\Gamma$, then the canonical cyclic JSJ splitting 
$\Gcan$ of $G$ 
 is retractable. 
\end{lem}

\begin{proof}
  We may assume that $\Gamma$ is a centered splitting of $G$, with central vertex   group $Q=\pi_1(\Sigma)$. Let $w$ and $S$ be as in Lemma \ref{compajsj}.

 Let   $p:Q\to G$ be a \ndbpm\ associated to  $\Gamma$, and 
let $\hat p:G\to G$ be an extension of   $p $ ``by the identity'' as in Lemma \ref{extid} (viewing $Q$ as a vertex group of $\Gamma$). Its restriction to $\pi_1(S)$ is a \bpm, because 
any maximal boundary subgroup of $S$ (i.e.\ an incident edge group near the vertex $w$ of $\Gcan$) is an incident edge group of $Q$ or is contained in 
a conjugate of a bottom vertex group of $\Gamma$
(because edges of $\hat \Gamma$ containing $\hat v$ are not collapsed in $\Gamma$).

We show that the \bpm\ $\hat p_{|\pi_1(S)}$
is non-degenerate. Clearly its image is    non-abelian. 
If $\hat p$ sends $\pi_1(S) $ isomorphically onto a conjugate, then $p$  is non-pinching, so Lemma \ref{n=1} implies that 
$\Gamma$ has a single bottom vertex group $B_1$, 
and that   $\hat p(G)$ is contained in a conjugate of $B_1$.
In particular, the image of $\pi_1(S)$ contains no conjugate of $Q$, a contradiction.
\end{proof}

\section{Towers as an ordering: core and prototypes
}
\label{ordre}

In this section $G$ is a CSA group.   We will define prototypes, and associate a core $\core(G)$ to $G$. We will see in Section \ref{elemequiv} that, when $G$ is a non-abelian hyperbolic group, it is elementarily equivalent to $\core(G)$ or $\F_2$, and two  non-abelian hyperbolic groups are elementarily equivalent if and only if their cores are isomorphic.

\subsection{Defining an order} \label{do}

  Let $G $ be a torsion-free CSA group. 
  
    Recall from Subsection \ref{ext} that a centered splitting $\Gamma$ is retractable if there is a \ndbpm;  equivalently (see Proposition \ref{equivpassimple})  there is a retraction $\rho$ from $G$ to a product of conjugates $\tilde B_i$ of the $B_i$'s  (to  $\tilde B_1*\Z$ if $\Gamma$ is simple  and $B_1$ is abelian)  with $\rho(Q)$   non-abelian.
 We then say that $G$ is an extended \'etage of surface type over $H=\tilde B_1*\dots*\tilde B_n$.

We shall use extended \'etages to define an order on the set of isomorphism classes of CSA groups.

\begin{dfn} \label{ord} Let $   G,H$ be 
torsion-free CSA groups. 

We define $G\gtun  H$ if  
$G $ is an extended   \'etage over $H$. 
In other words, $G> _1H$ if   $G\simeq H*\Z$, or $G$ has a retractable splitting with base   isomorphic to $H$.

We then define 
$G\tge H$ if $G$ is an extended tower over $H$, i.e.\  if 
there is a sequence $G=G_0, G_1,\dots,G_p  \simeq H$, with $p\ge0$,    such that $G_i\gtun G_{i+1} $  for each $i$. We write $G\tgt H$ if $p>0$. 

\end{dfn}

  In particular, the free group $\bbF_r$ satisfies $\bbF_r\tge \{1\}$.

\begin{rem} \label {glim2}
If $G\tgt  H$, then $H$ 
may be viewed as a proper quotient of $G$, and also as  a subgroup of $G$. 
  If $G$ is hyperbolic, $H$  is   quasiconvex, so is a hyperbolic group.

Also note that $G*G'\tge H*H'$ if $G \tge H $ and $ G'\tge  H'$, and $G*\bbF_r\tge G$.
\end{rem}

  \begin{lem} \label{lem_etages}
 $G\tge H$ if and only if there is a sequence $G=G_0, G_1,\dots,G_p$ such that $G_i$ is an 
 \'etage of surface type over $G_{i+1}$, and   $G_p\simeq H*F$ with $F$ free (possibly trivial or cyclic).

 If $G\tge H$ with $H$   non-abelian,   there exists a  free group $F$ such that $G*F$ is a simple tower over   a subgroup isomorphic to $H$.
 The converse is true if $H$ 
 has no  cyclic free factor.
\end{lem}

\begin{proof}
The first assertion follows from  Remark \ref{zenbas}. The second one follows from  
Corollary \ref{ett}, and its converse from  Corollary   \ref{redu}. 
\end{proof}

\begin{cor} 
If we restrict to non-abelian   CSA groups  with no 
cyclic free factor, 
then $G\tge H$ if and only if some $G*F$ with $F$ free  is a simple tower over $H$.
\qed
\end{cor}

\begin{prop} \label{dcc2}
The relation $\tge$ defines a partial order on the set of isomorphism classes of  
torsion-free CSA groups.  There is no infinite sequence $G=G_0\tgt G_1\tgt \cdots \tgt G_i\tgt \cdots$.
\end{prop}

\begin{proof}   This follows from Lemma \ref{dccg}:
if $G\tgt H$, then $H$  has smaller first Betti number mod 2.
When $G$ is hyperbolic, one may also argue
that the  $G_i$'s are $G$-limit groups,
and there is no infinite sequence of $G$-limit groups $G_i$ such that $G_{i+1}$ is a proper quotient of $G_i$ \cite{Sela_diophantine7}. 
\end{proof}

\begin{prop} \label{fini1} 
Given a torsion-free   CSA group $G$, the set of  groups $H$ 
such that $G\tge H$ is finite up to isomorphism.
\end{prop}

\begin{proof}
By  Lemma 
  \ref{lem_etages} and Proposition \ref{dcc2}, 
it suffices to prove finiteness for the set of  groups $H$ such that $G$ is an extended \'etage of surface type over $H$.  

The result follows from Corollary \ref{fini-1} if $G$ is one-ended.
If $G$ is infinitely-ended, Proposition \ref{Grus} expresses $H$ as $H=B_{\Lambda_\Z}*G_2*\cdots *G_k*F'$, 
where $G=G_1*G_2*\cdots * G_k*F $ is a Grushko decomposition of $G$, the group $F'$ is   isomorphic to a free factor of $F$, 
and  $G_1 $  is an extended \'etage of surface type over $ B_{\Lambda_\Z}$. Finiteness follows from the one-ended case.
\end{proof}

Proposition \ref{fini1} will imply that, given a  torsion-free hyperbolic group $G$, there are only finitely many
isomorphism classes of groups that elementarily embed in $G$ (Corollary \ref{fini2}).

\begin{dfn}[Prototype] \label{dfn_proto} 
A torsion-free   CSA group $G$ is a \emph{prototype} if it is minimal for $\tge$. 
\end{dfn}

The trivial group, and groups with no cyclic splitting, in particular     groups with Kazhdan's property (T), are prototypes.
  On the other hand, $\bbF_2$ is not a prototype.

\begin{lem} \label{eqprot}
Let $G$ be a 
 torsion-free   CSA group. 
The following are equivalent:
\begin{enumerate}
\item $G$ is a  prototype;
\item  
$G$ has no cyclic free factor and is not an extended \'etage of surface type (i.e.\ it 
 has no   retractable splitting); 
\item 
$G$ is a free product of one-ended groups which are not extended \'etages of surface type.
\end{enumerate}
\end{lem}

 \begin{proof}
Suppose that $G$ is a prototype. By minimality, it cannot be written $G=H*\Z$  and cannot be an extended \'etage of surface type.
This proves $1\implies2$.

 For  $2\implies3$, we simply note that, if $G$ is not an extended \'etage, neither are its Grushko factors by Remark \ref{rk_prodlib}. 

If $G$ is a free product of one-ended groups which are not extended \'etages of surface type, it has no cyclic free factor, and is not an extended \'etage of surface type by Corollary \ref{infbout}. It is   minimal for the order, hence is a prototype.
 \end{proof}

 \begin{rem}\label{genprot}
In particular, if $G$ is not a prototype, it has a cyclic free factor or one of its Grushko factors is an extended  \'etage (necessarily of surface type).  

A one-ended group is a prototype if and only if it is not an extended \'etage  (of surface type).  
A 
group 
is a prototype if and only if it is a free product of one-ended prototypes. 

If $G$ is a prototype, no $G*F$ with $F$ free is an extended \'etage of surface type (this follows from Corollary \ref{redu2}).

The fundamental group of a closed hyperbolic  surface $S$ is a prototype if and only $S$ is exceptional (the non-orientable surface of genus 3), see Example \ref{csurf}.

\end{rem}

  The following theorem   was proved by Sela \cite{Sela_diophantine7} for hyperbolic groups, using their classification 
  up to elementary equivalence.
We will give a proof independent of this theory.

\begin{thm}
\label{thm_confluence}
Let $G$ be a 
 torsion-free CSA group. Up to isomorphism, there exists a unique prototype $\core(G)$  such that $G\tge \core(G)$. 
\end{thm}

\begin{dfn}\label{dcore}
$\core(G)$ is called the \emph{core} of $G$.
\end{dfn} 

\begin{rem}\label{cpl}
For instance   $\core(G)\simeq G$ if $G$ is a prototype, and $\core(G)=\{1\}$ if $G$ is a free group or a non-exceptional closed surface group.    The  core of $G$   embeds into $G$ as a retract, 
 but not in a canonical way in general (see Subsection \ref{core}). 
 Also note that  $\core(G_1*G_2)\simeq 
\core(G_1)*\core(G_2)$ by Remark \ref{rk_prodlib}. 
\end{rem}

\begin{rem}\label{rem_ecore}
   The equality
  $\core(G_1*G_2)\simeq\core(G_1)*\core(G_2)$ requires  
  $\core(\F_n)$ to be trivial.    
    This has a drawback 
  when considering elementary equivalence: 
  we will see in Section
  \ref{elemequiv} that   a hyperbolic group $G$ is equivalent to $c(G)$ if $c(G)\ne\{1\}$,
  to $\bbF_2$ if $G$ is non-abelian with $c(G)=\{1\}$.  We will
  therefore introduce (Definition \ref{elcore}) 
  the    \emph{elementary core}
  $\ecore(G)$, which for $G$
a  non-abelian hyperbolic group is always elementarily equivalent to $G$ (in particular,  $\ecore(\F_2)=\F_2$), but 
  at the cost of losing the nice behaviour under free products.
 \end{rem}

The existence of $\core(G)$  in Theorem \ref{thm_confluence} follows   directly from the descending chain condition (Proposition \ref{dcc2}).
Our proof  of uniqueness relies on 
the next lemma, which will be proved in Subsection \ref{conflu}.   When $G$ is hyperbolic, uniqueness may also be deduced from the classification up to elementary equivalence (see \cite{Sela_diophantine7}).

We denote by 
$\calg$ the class of all 
torsion-free CSA groups.

\begin{lem}[Minimal base] \label {lem_confluence2}
Given $G\in\calg$, there exists   $M\in\calg$ such that $G\tge M$ and, if $G\gtun  G'$ with $G'\in\calg$,
 then $G'\tge M$.
\end{lem}

The group $M$ will be constructed   in Subsection \ref{confl} from the minimal base of a splitting, introduced in Subsection \ref{mbs}.

\begin{preuve}[Proof of  Theorem \ref{thm_confluence} from Lemma \ref{lem_confluence2}]
Noetherianity (Proposition \ref{dcc2}) guarantees the existence of  a prototype $P$ such that $G\tge P$. 
We show that confluence holds for $G$: if $G\tgt G_1$ and $G\tgt  G_2$, there exists $G'$ such that $G_1\tge G'$ and $G_2\tge G'$. This clearly implies uniqueness of $P$,  which is what we need to prove.

By noetherianity, we may assume that confluence holds for all $\tilde G$ such that $G\tgt  \tilde G$. 
Since $G\tgt  G_i$,  there exists $A_i$ such that $G\gtun  A_i\tge G_i$. The group $M$ provided by Lemma \ref{lem_confluence2} satisfies $A_i\tge M$ for $i=1,2$. We assume that confluence holds for $A_i$, so there exists $C_i$ such that $G_i\tge C_i$ and $M\tge C_i$. As confluence holds for $M$, there exists $G'$ such that 
$C_i\tge G'$, hence $G_i\tge G'$.
\end{preuve}

\subsection{Proof of the Minimal Base  Lemma \ref{lem_confluence2}} 
\label{conflu}
\subsubsection{The minimal base of a splitting} \label{mbs}

 We fix a non-trivial cyclic splitting $\Gamma$ of $G$. 
   We assume that  its Bass-Serre tree  is 2-acylindrical, and no edge joins two surface-type vertices   or two 
    abelian vertices. 
 We will apply this subsection to the   canonical cyclic JSJ splitting   $\Gcan$ of a one-ended CSA group, but the CSA property  is not  really necessary here.

Recall (Definition \ref{dfn_retractablesurf}) that a non-exceptional surface-type vertex group $\pi_1(\Sigma)$ of $\Gamma$
is retractable (in $G$)  if there exists a \ndbpm\ $p:\pi_1(\Sigma)\to G$. By definition, $p$ acts as a conjugation on each boundary subgroup, has non-abelian image, and is not an isomorphism onto a conjugate of $\pi_1(\Sigma)$. 

\begin{figure}[ht!] 
  \centering
  \includegraphics{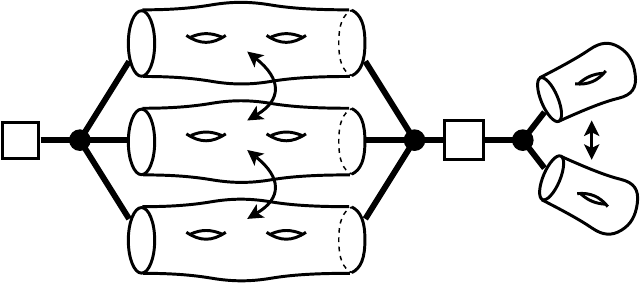}
  \caption{Twinnings between surfaces of a splitting}\label{fig_twinning}
  \label{tw}
\end{figure}

As in Example \ref{bete}, a  \ndbpm\  $p:\pi_1(\Sigma)\to G$  may be an isomorphism onto (a conjugate of) some 
 other 
surface-type vertex group $\pi_1(\Sigma')$ of $\Gamma$. We then say that $p$ is a \emph{twinning} of $\Sigma$ with $\Sigma'$, and that $\Sigma'$ is a  \emph{twin} of $\Sigma$   (see Figure \ref{tw}). On the other hand, if $p$ is not a twinning, 
we say that 
it  is a \emph{strong \bpm} 
and that 
$\Sigma$, or $\pi_1(\Sigma)$, is \emph{strongly retractable}.

Let $p$ be a twinning. Boundary subgroups of $\pi_1(\Sigma)$ are mapped to conjugates, hence to boundary subgroups of $\pi_1(\Sigma')$. This implies that $\Sigma$ and $\Sigma'$ have the same number of boundary components   and are homeomorphic  (see Lemma 3.13 of \cite{Perin_elementary}).
In particular, $p\m$ is a twinning. 
Twinning thus induces an equivalence relation on the set $\cals_\Gamma$ consisting of all  surfaces associated to surface-type vertices
of $\Gamma$ (by convention, $\Sigma$ is twinned with itself). This relation is compatible with   retractability and strong retractability.

\newcommand{\GiS}{\Gamma_{\! i}(S)}

Let $\cals$ be any subset of   $\cals_\Gamma$. If we remove  from $\Gamma$ the open stars of all vertices associated to  surfaces not  
in $\cals$, we get a graph $\Gamma(\cals)$ whose components we denote by $\GiS$. 

The graphs of groups $\GiS$   may be points. They are not necessarily minimal: they may have   terminal cyclic vertices $v$ with the  incident edge group equal to the vertex group.
 If this happens, we  remove $v$ and the open edge (this does not change the fundamental group of the graph of groups). If the edge joins $v$ to a surface-type vertex $w$, this vertex is no longer surface-type because the 
surface has gained a free boundary component; this will not be a problem because by Remark \ref{bolib}  $w$  cannot be retractable as a vertex of  $\GiS$ (see the proof of Proposition \ref{toursurm}   below).

 From now on, we shall denote by $\GiS$ the minimal graph of groups obtained by removing edges as just explained (it may consist of a single vertex, cyclic or not). 
Generalizing Definition \ref{centspl}, we denote by $B_i(\cals)$ the  fundamental group
of 
$\GiS$, and by $B(\cals)$ the free product of the groups $B_i(\cals)$, which we call   the \emph{base} of $\cals$.

Given $\Gamma$, 
  let $M(\Gamma)$ be the base 
 of a subset $\cals_0\inc\cals_{\Gamma}$  satisfying the following conditions:  
 \begin{itemize}
\item  it contains all non-retractable surfaces; 
\item it contains no strongly retractable surface; 
\item  if $\Sigma$ is retractable but not strongly retractable, $\cals_0$ contains exactly one surface in the  twinning equivalence class of $\Sigma$ (which cannot be a singleton). 
\end{itemize}

The next lemma says that $M(\Gamma)$ is well-defined   up to isomorphism, we call it the \emph{minimal base}  of $\Gamma$. 

\begin{lem} \label{bbdef}
Up to isomorphism, $M(\Gamma)$ is independent of the choice of $\cals_0$: it does not change if we replace a surface $\Sigma$ which is  retractable but not strongly retractable by a twin surface $\Sigma'$.  

\end{lem}

\begin{proof}  
Let $v$ and $v'$ be the vertices of $\Gamma$ carrying $ \pi_1(\Sigma)$ and $ \pi_1(\Sigma')$ respectively.
By assumption, there is an isomorphism   $p:\pi_1(\Sigma)\to  \pi_1(\Sigma')$ 
 which acts as a conjugation on each boundary subgroup of $\pi_1(\Sigma)$.
If $C$ is an incident edge group at $v$, its image is an incident edge group at $v'$, and by   2-acylindricity the edges of $\Gamma$ carrying $C$ and $p(C)$ join $v$ and $v'$ to the same   vertex.
This implies that there is an automorphism of $\Gamma$ (viewed as a graph with no extra structure) sending $v$ to $v'$ and fixing all other vertices. 

 The twinning isomorphism $p:\pi_1(\Sigma)\ra \pi_1(\Sigma')$  and  its inverse 
act by conjugations on the boundary subgroups
of $\Sigma$ and $\Sigma'$.
 They may  therefore be extended to an automorphism of $G$
``by the identity'' as in Lemma \ref{extid}.  

 This may also be viewed topologically. 
The group $G$ is  
the fundamental group of a graph of spaces containing
two homeomorphic surfaces $\Sigma,\Sigma'$ attached to the rest of the space with the same attaching maps.
The twinning isomorphism is 
represented by a   pair of inverse homeomorphisms between $\Sigma$ and $\Sigma'$  which fix the boundary. These homeomorphisms 
may be extended by the identity to the rest of the space. 
This yields an automorphism of $G$ exchanging $ \pi_1(\Sigma)$ and $ \pi_1(\Sigma')$
and acting as a conjugation on each 
$B_i(\cals_\Gamma\setminus\{\Sigma,\Sigma'\})$. 
\end{proof}

If $S$ is a  retractable surface  of $\Gamma$, we write 
$\cals_{-S}$ for $\cals_\Gamma\setminus {\{S\}}$.

\begin{lem}\label{fortfaible} Let $S$ and $\Sigma$ be distinct surfaces of a splitting $\Gamma$. 
Let $\Gamma_{\! i}(\cals_{-S} )$ be the component  of $\Gamma(\cals_{-S} )$ containing the vertex carrying $\Sigma$, and $B_i(\cals_{-S} )$ its fundamental group.

If $S$ is retractable and $\Sigma$ is strongly retractable (in $G$), then $\Sigma$ remains strongly retractable in $B_i(\cals_{-S} )$.
\end{lem}

 \begin{proof}  Let $p:\pi_1(\Sigma)\to G$ be a strong \bpm. If it satisfies one of the conditions of Lemma \ref{desc}, we get  a strong \bpm\  $p:\pi_1(\Sigma)\to B_i(\cals_{-S} )$ as required.

 We may therefore assume that $p$ is non-pinching, and by Lemma \ref{lemcle2} that $\Sigma$ has an incompressible subsurface $Z $ whose fundamental group is mapped to 
   a finite index subgroup $K$ of $\pi_1(S)$. 

Since $S$ is retractable,  $G$ is an extended \'etage over $B(\cals_{-S} )$ and there is a  retraction $\rho
 $ from $G$ to a   subgroup $\tilde B$ isomorphic
 to $B(\cals_{-S} )$ (see Proposition  \ref{equivpassimple}). As in the proof of Lemma \ref{desc}, we define $p'=\rho\circ p$ 
  and we apply  Lemma   \ref{cinqdouze} 
   to get retractability in $B_i(\cals_{-S} )$ rather than just in $\tilde B$. 
   We only have to make sure that $p'$ is not an isomorphism between  $\pi_1(\Sigma)$ and a conjugate of $\pi_1(\Sigma')$ 
   for some surface $\Sigma'\ne S$ (possibly equal to $\Sigma$). 

 If it is, $\rho$ is injective on $K$, hence on $\pi_1(S)$ because $G$ is torsion-free. 
Since $\rho(K)\subset \pi_1(\Sigma')$ and $\rho(\pi_1(S))$ contains $\rho(K)$ with finite index,
it follows that $\rho(\pi_1(S))\subset \pi_1(\Sigma')$, 
because  no subgroup of $G$ contains $\pi_1(\Sigma')$ as a proper subgroup of finite index.
Boundary subgroups of $\pi_1(S)$ are mapped to conjugates, hence into boundary subgroups of $\pi_1(\Sigma')$.  
Lemma 3.10 of \cite{Perin_elementary} implies that the image of $\pi_1(S)$ has finite index in $\pi_1(\Sigma')$. Using the same complexity $k $ as in  Lemma 3.12 of \cite{Perin_elementary}, we now write $k(\Sigma)\ge k(Z)\ge k(S)\ge k(\Sigma')=k(\Sigma)$. All these complexities are equal, so $Z=\Sigma$ and $p$ is an isomorphism between $\pi_1(\Sigma)$ and $\pi_1(S)$, hence not a strong \bpm.
 \end{proof}

\begin{prop} \label{toursurm}
$G$ is an extended tower over $M(\Gamma)$.
\end{prop}

\begin{proof} 
If $S\in\cals_\Gamma\setminus\cals_0$, it  is retractable and $G$ is an extended étage over   $B(\cals_{- S})$. Retracting   surfaces iteratively 
 will  then express $B(\cals_{- S})$  as an extended étage over some $B(\cals_\Gamma\setminus\{S,S'\})$, 
and eventually 
 $G$ as an extended tower over $M(\Gamma)=B(\cals_{0})$. We just have to be careful, as the retractability type of surfaces might change during the process.

First suppose that there is a surface $S\in\cals_0$ which is retractable (but not strongly retractable). Let $\cals_1$ be its set of twins (excluding $S$ itself). We may retract the surfaces of $\cals_1$ using the twinnings with $S$, so as to express $G$ as an extended  tower over  $B(\cals_\Gamma\setminus\cals_1)$. Note that $\Gamma(\cals_\Gamma\setminus\cals_1)$ is connected, because the vertices carrying $S$ and its twins are adjacent to the same   
vertices. 

By Lemma \ref{fortfaible}, surfaces not in $\cals_0\cup \cals_1$ remain   
retractable in $\Gamma(\cals_\Gamma\setminus\cals_1)$ 
 (because they have a twin or   they are strongly retractable). The surface $S$ itself either remains a surface but becomes non-retractable, or it stops being a surface (it gains a free boundary component when we make $\Gamma(\cals_\Gamma\setminus\cals_1)$ minimal,   see Remark \ref{bolib}).  

Performing this operation for every retractable 
surface   $S\in\cals_0$ lets us assume that every retractable surface is strongly retractable. Using Lemma \ref{fortfaible} again, we can now retract them one by one.
\end{proof}

\subsubsection{Proof of confluence}\label{confl}

We can now prove Lemma \ref{lem_confluence2}.
We   first assume that $G$  (a CSA group) is one-ended. Let $\Gcan$ be  its canonical cyclic JSJ decomposition (see Subsection \ref{JSJ}).        It satisfies the assumptions of Subsection \ref{mbs}.

If  $\Gcan$
 is non-trivial, 
 we   define 
   $M=M(\Gcan)$ 
 as the minimal  base of $\Gcan$. 
 We have $G\tge M$ by Proposition \ref{toursurm}. We now suppose that $G$ is an extended étage of surface type over $G'$  and we show $G'\tge M$. 
 
 If the étage is associated to a surface $S$ of $\Gcan$, this surface is retractable;   we can choose $\cals_0$ with $S\notin\cals_0$, and we have seen in the proof of Proposition  \ref{toursurm} that $G'=B(\cals_{-S} )$ is an extended tower over   a subgroup isomorphic to $M$.
  In general,    Lemmas \ref{compajsj} and \ref{bases} yield  a surface $S$ of $\Gcan$ such that $G'$ is the free product of a free group with $B(\cals_{-S} )$, so $G'\tge B(\cals_{-S} )\tge M$.
  
    Still assuming $G$ one-ended, suppose that $\Gcan$
 is trivial. There are two cases. If $G$ is either   rigid (it has no cyclic splitting) or $\Z^2$ or the exceptional closed surface group, then $G$ has no retractable splitting, 
and there is no $G'$ such that $G\gtun G'$; 
we define $M=G$ in this case.
Otherwise, $G$ is  the fundamental group of  a non-exceptional closed surface  (see \cite{GL_JSJ}).
 As explained in Example \ref{csurf}, $G$ is an extended tower over 
$\{1\}$, and we define $M =\{1\}$. 
Any $G'$ such that $G\gtun G'$ is free hence $\tge \{1\}$.

The lemma is clear if $G$ is cyclic, so  suppose that  $G$ has infinitely many ends.
If $G$ is free, we can take $M=\{1\}$. 
Otherwise, we consider a Grushko decomposition $G=G_1*\cdots *G_k*F$  with $F$ free. 
Since $G_i$ is one-ended, it has a minimal base $M_i$ and 
we define $M$ as the free product of the 
$M_i$'s.  

Using Remark \ref{glim2}, we can write 
$G\tge G_1*\cdots * G_k\tge M_1*\cdots * M_k =M$. Similarly, $G'\tge M$ if $G=G'*\Z$. If $G$ is an extended \'etage of surface type over $G'$, 
  we get  after reordering $G'=  B_{\Lambda_\Z}*G_2*\dots*G_k*F'$ as in Proposition \ref{Grus}, with $G_1 $    an extended \'etage of surface type over $ B_{\Lambda_\Z}$. 
We then have $B_{\Lambda_\Z}\tge M_1$ by the one-ended case, hence  $G'\tge M_1*\cdots * M_k$.

\section{Elementary equivalence}
\label{elemequiv}

Though motivated by model theory, the paper so far is purely group theory. We now give Sela's classification of hyperbolic groups up to elementary equivalence. 

  From now on, we restrict to hyperbolic groups.
Recall that we associated a core $\core(G)$  to any  torsion-free hyperbolic group $G$ (Definition \ref{dcore}). 
\begin{thm} [\cite{Sela_diophantine7}]\label{classif}
Two non-abelian torsion-free hyperbolic groups are elementarily equivalent if and only if their cores are isomorphic.
\end{thm}

We split this statement into two. 

\begin{thm}\label{sensfac}
 Let $G,H$ be non-abelian 
 torsion-free hyperbolic groups. If $G$ and $H$ are 
   elementarily equivalent,  their cores are isomorphic.
\end{thm}

We shall give a proof of this theorem (the ``only if'' direction in Theorem   \ref{classif}) in Section \ref{interp} (see Corollary 
\ref{cor_eqv_elem}).   
  As we will see, the converse direction is an easy consequence of Theorem \ref{Tarski} (``Tarski''),
which contains the fact that 
$H*F\equiv H$ 
if $F$ is free and $H$ is a non-abelian torsion-free hyperbolic group, and in particular a solution to Tarski's problem,   so will
 ---unfortunately--- not be proved here. 
We recall its statement. 

\begin{enonce*}{Theorem \ref{Tarski} (``Tarski'' \cite{Sela_diophantine7})}
   Let $G$ be a non-abelian torsion-free hyperbolic group. If $G$ is a  \emph{simple} 
 tower
 over a non-abelian subgroup $H<G$,   then the inclusion $H\hookrightarrow G$ is an elementary embedding.
In particular, $G$ and $H$ are elementarily equivalent. \qed
 \end{enonce*}

 It is important here   that the
surfaces appearing in the tower   be non-exceptional.
 We shall also see   (Theorem \ref{contrex})
 that the conclusion may be wrong if $G$ is a \emph{non-simple}   tower 
 over $H$.

 \begin{cor}
 \label{equivtour}
  Let $G,H$ be non-abelian torsion-free hyperbolic groups. 
     If $G\tge H$, then $G$ and $H$  are 
   elementarily equivalent.
   \end{cor}

\begin{proof}
  Recall (Corollary \ref{ett}) that   there exists $r\geq 0$ such that $G*\bbF_r$ is a simple tower over a subgroup $\tilde H$
  isomorphic to $H$. 
Since $G*\bbF_r$ is a simple tower over $G$ (with étages of free product type), Theorem \ref{Tarski} implies
 $G\equiv G*\bbF_r\equiv \tilde H\equiv H$.
\end{proof}

\begin{example}[see Example \ref{expt}] \label{expt2}
Let $G$ be a non-abelian torsion-free hyperbolic group. Suppose that $g\in G$ is non-trivial and is a commutator. 
Then $\langle G,x,y\mid g=[x,y] \rangle$ is elementarily equivalent to $G$. 
\end{example}

\begin{example}[see Example \ref{csurf}] \label{csurf2}
The fundamental group of a closed   surface $\Sigma$ with $\chi(\Sigma)\le-2$ 
is elementarily equivalent to $\bbF_2$. 
\end{example}

  In order to prove    the ``if'' direction in Theorem \ref{classif},
we  want to deduce from  Theorem  \ref{Tarski}
  that $G$ is elementarily equivalent to its core $\core(G)$. Since the theorem does not apply when $H=\core(G)$ is abelian (i.e.\ when $\core(G)=\{1\}$), we  
introduce a slightly modified definition. 
  
Recall that a group is a prototype if it is minimal for $\tge$, and the core $\core(G)$ is the unique prototype such that $G\tge \core(G)$. 
When $G$ is a non-abelian free group, $\core(G)$ is the trivial group, and in particular is not elementarily equivalent to $G$.
As we will see, the case where $\core(G)=\{1\}$ is the only exception, and in this case $G$ is elementarily equivalent to $\bbF_2$.
To overcome this exception, we   give the following definition.

\begin{dfn}[Elementary core] \label{elcore}
 Let $G$ be a non-abelian torsion-free hyperbolic group. We define its \emph{elementary core} 
 as the group $\ecore(G)$ well-defined up to isomorphism by
$$ \begin{cases}
   \ecore(G)=\core(G)& \text{if $\core(G)\ne 1$}\\
   \ecore(G)=\bbF_2 & \text{if $\core(G)= 1$.}\\
 \end{cases}$$
\end{dfn}

Note that $\ecore(G)$ and $\ecore(H)$ are isomorphic if and only if $\core(G)$ and $\core(H)$ are isomorphic, so this change does not affect the validity of Theorems \ref{classif} and \ref{sensfac}. On the other hand, it is not always true that $G\tge \ecore (G)$ (see Example \ref{example_not_extended_tower}).
  As noted in Remark \ref{rem_ecore}, although $\core(G_1*G_2)=\core(G_1)*\core(G_2)$, this is no more true for  $\ecore$.

  We 
    can now state
the following   consequence of ``Tarski'', which, together  with
Theorem \ref{sensfac} (proved in Subsection \ref{interp}), clearly implies Theorem \ref{classif}.  

 \begin{cor}
 \label{equivcore}
Any non-abelian torsion-free hyperbolic group $G$ is 
   elementarily equivalent
   to its elementary core $\ecore(G)$.
\end{cor}

\begin{proof}
 This  is clear if $\core(G)\ne\{1\}$   using  Corollary \ref{equivtour}, since $G\tge\core(G)=\ecore(G)$. If $\core(G)=\{1\}$, we have $G\tge\{1\}$, hence $G*\bbF_2\tge \bbF_2$ by Remark \ref{rk_prodlib}, and $G\equiv G*\bbF_2\equiv \bbF_2=\ecore(G)$.
\end{proof}

\begin{rem} Although $\ecore(G)$ always embeds into $G$ as a retract,
there are examples of hyperbolic groups for which $\core(G)$  is non-trivial and does not embed elementarily in $G$ (see Theorem \ref{contrex}).
\end{rem}

  The following lemma will be used in  Section \ref{ee}. 

\begin{lem}\label{stab_core}
  Let $G$ be a non-abelian torsion-free hyperbolic group.
  There exists $r\geq 0$ such that $G*\bbF_r$ is a simple tower over a subgroup isomorphic to $\ecore(G)$.
\end{lem}

\begin{proof}
  By definition (see Subsection \ref{do}),   $G$ is an extended tower over $\core(G)$.
  If $\core(G)\ne\{1\}$, then  Corollary \ref{ett} says that there exists $r\geq 0$ such that
  $G*\bbF_r$ is a simple tower over a subgroup isomorphic to $\core(G)=\ecore(G)$.
  If $\core(G)=\{1\}$, then  $G*\bbF_2\tge \bbF_2$ by Remark \ref{rk_prodlib}, so there exists $r\geq 0$ such that $G*\bbF_2*\bbF_r$ is a simple tower over a subgroup isomorphic to $\bbF_2=\ecore(G)$.
\end{proof}

It follows from Theorem \ref{Tarski}
that many non-free finitely generated groups (in particular surface groups, see Example \ref{csurf2}) are elementarily equivalent to $\F_2$. They may be characterized as follows.

\begin{cor}[Elementarily free groups] 
\label{cor_tarski}
Let $G$ be a finitely generated non-abelian torsion-free group.
The following are equivalent:
\begin{enumerate}
\item $G$ is elementarily equivalent to $\bbF_2$; 
\item $G$ is hyperbolic and $\ecore(G)=\bbF_2$;
\item $G$ is hyperbolic and $\core(G)=\{1\}$;
\item $G$ is hyperbolic and $G\tge \{1\}$;
\item $G$ is an extended tower over  $\Z$;

\item  
 there exists a free group $ \bbF_r=\grp{a_1,\dots ,a_r}$ with $r\ge2$ such that $G*\grp{a_1,\dots,a_r}$ is a simple tower over $\grp{a_1,a_2}$;
\item there exists a free group $ \bbF_r$ such that $G* \bbF_r$ is a simple tower over a  non-abelian free subgroup;
\item 
 there exists $r\geq 2$ such that the obvious embedding
  $\bbF_2\hookrightarrow  \bbF_r*G$ is elementary (see Section \ref{ee}).
 \end{enumerate}
\end{cor}

\begin{proof}
2, 3, 4 are clearly equivalent. $4\implies 5$ and 
$6\implies 7$ are  clear. We prove $1\implies2$, 
$5\implies6$,
$7\implies1$, and 
$6\iff8$.

If $G\equiv \bbF_2$, it is hyperbolic by a result of Sela \cite{Sela_diophantine7} (see also \cite{Andre_hyperbolicity} for the case with torsion) 
 saying that being hyperbolic is invariant under elementary equivalence among finitely generated torsion-free groups;
  more simply, one can use \cite[Corollary 4.4]{Sela_diophantine1} 
   saying that limit groups containing no $\Z^2$ are hyperbolic.
The implication 
$1\implies 2$ then follows from Theorem \ref{sensfac}.

  To prove $5\implies 6$, note that  $G*\Z$ is an extended tower over $\bbF_2$ by Remark \ref{rk_prodlib}. By Corollary \ref{ett}, there exists $s$ such that $G*\bbF_s$ is a simple tower over a subgroup $H<G*\bbF_s$ isomorphic to $\bbF_2$. By Remark \ref{rk_prodlib}, $G*\bbF_s*\grp{a_1,a_2}$ is a 
simple tower over   $H*\grp{a_1,a_2}$, hence over $\grp{a_1,a_2}$.

We now prove $7\implies1$.
If $G*\bbF_r$ is a simple tower over $\bbF_2$, then $G*\bbF_r$ (hence also $G$) is hyperbolic by standard combination theorems \cite{BF_combination}, and $G\equiv G*\bbF_r \equiv \bbF_2$ by Theorem \ref{Tarski}.

The equivalence between 6 and 8 follows from  Theorem   \ref{Tarski} (``Tarski''
) 
and Theorem \ref{thcl} below.
\end{proof}

We will see in Subsections \ref{nst} and \ref{core} that it is necessary to use a free group $\bbF_r$ in items 6, 7, 8.

\section{Elementary embeddings}\label{ee}

  The previous section was devoted to elementary equivalence of abstract hyperbolic groups. In this section we consider elementary embeddings. In particular, we discuss Perin's theorem about elementarily embedded subgroups of hyperbolic groups \cite{Perin_elementary}, and related facts about cores and prototypes.

\subsection{Characterizing elementarily embedded subgroups} \label{ees}

   Recall the following main result by  Sela \cite{Sela_diophantine7}, 
which implies a solution to Tarski's problem.

\begin{enonce*}{Theorem \ref{Tarski} (``Tarski''
)}
   Let $G$ be a non-abelian torsion-free hyperbolic group. If $G$ is a  \emph{simple} 
 tower
 over a non-abelian subgroup $H<G$,   then the inclusion $H\hookrightarrow G$ is an elementary embedding.
In particular, $G$ and $H$ are elementarily equivalent. \qed
 \end{enonce*}

Recall  that, if $G$ is a  
simple tower over $H$ or, more generally, an extended tower over $H$ relative to $H$, then 
$H$ is well-defined as a subgroup of $G$ (up to conjugacy); see Definition \ref{relto}  for relative towers.

 The main result of this subsection is the following   theorem, where $3\implies 1$ relies on ``Tarski'' and $1\implies2$ is the main result of Perin's thesis \cite{Perin_elementary}.

\begin{thm}[\cite{Perin_elementary}
]
\label{thcl}
Let $G$ be a non-abelian torsion-free hyperbolic group,  and let $H<G$ be a   non-abelian 
subgroup.
The   following are equivalent:
\begin{enumerate}
\item $H$ is elementarily embedded in $G$; 
\item $G$ is an extended tower over $H$ relative to $H$;
\item there exists a finitely generated free group $F$ such that $G*F$ is a simple tower over  $H$ (viewed as a   subgroup of $G*F$ in the obvious way).
\end{enumerate}
\end{thm}

\begin{rem}[Warnings]\label{rem_warning}

   If $G$ is an extended tower over $H$ relative to $H$, 
the intermediate groups $G_i$ are not elementarily embedded in $G$
in general: although $G$ is an extended tower over $G_i$, this tower is not necessarily relative to $G_i$.

   Similarly, Assertion $3$ gives a chain of subgroups $G*F>G_1>\dots>G_n=H$ with $H<G$. 
     The intermediate groups $G_i$ 
   are elementarily embedded in $G*F$.
They are elementarily embedded in $G$ if they are contained in $G$, 
  but this is not the case in general.

\end{rem}

\begin{proof}[Proof of Theorem \ref{thcl}]
  We prove $2\implies 3 \implies 1$.  The implication $1\implies2$  
  (Perin's 
  theorem \cite{Perin_elementary}) will be proved in the next section (see Corollary \ref{thcl2}).

The implication  $2\implies3$ 
  follows from Corollary \ref{ett} and Remark \ref{rk_tour_rel},
saying that $G*F$ is a simple tower over a subgroup $H'<G*F$  conjugate 
to $H$.

 The implication $3\implies 1$ is a simple application of ``Tarski'' (Theorem \ref{Tarski}):
if $G*F$ is a simple tower over $H$, 
the embeddings of $H$ and $G$ into $G*F$ are elementary by   ``Tarski'', so $H\preceq_e G$ by Remark \ref{tm}. 
\end{proof}

\begin{cor}[\cite{Perin_elementary}]\label{cor_FF} 
  If $H$ is an elementarily embedded subgroup of a free group $F$, then $H$ is a free factor of $F$.
\end{cor}

\begin{proof}
  Indeed,   $F$ is   an extended tower over $H$   (relative to $H$).  But a free group has no structure of   extended étage of surface type by Corollary \ref{paset},  so $H$ is a free factor of $F$.
\end{proof}

Combining $1\implies 2$ in Theorem \ref{thcl} with Theorem \ref{fini1} yields:

\begin{cor} \label{fini2}
Given a torsion-free hyperbolic group $G$, there are only finitely many isomorphism classes of groups that embed elementarily in $G$. \qed
\end{cor}

We will see in Example   \ref{imee} 
that, in general,
the number of elementarily embedded subgroups $H\preceq_e G$ may be infinite, even up to automorphisms of $G$.
This cannot happen, however, for one-ended subgroups $H$   (even without assuming that the embeddings are elementary), since   in this case
$G$   only contains finitely many conjugacy classes of subgroups isomorphic to $H$ \cite[Thm 5.3.C']{Gromov_hyperbolic}. 

\begin{cor}
\label{ajf}
Let $G_1$, $G_2$ be torsion-free hyperbolic groups. If $H_i$ is elementarily embedded in $G_i$, then $H_1*H_2$ is elementarily embedded in $G_1*G_2$. 
\end{cor}

\begin{proof}
Follows from $1\iff 3$ in Theorem \ref{thcl} and Remark \ref
{rk_prodlib}.
\end{proof}

This may be compared to the following result, which has been extended without assuming hyperbolicity by Sela in \cite{Sela_diophantineX}. It follows    from Theorem \ref{classif} and Remark \ref{cpl}.

\begin{thm}[Sela] 
  Let $G_1,G_2,G'_1,G'_2$ be   non-abelian torsion-free hyperbolic groups. If $G_i$ is elementarily equivalent to $G'_i$,
then $G_1*G_2$ is elementarily equivalent to $G'_1*G'_2$.\qed
\end{thm}

\subsection{Prototypes}

  The core $\core(G)$ of a non-abelian torsion-free hyperbolic group $G$  is a prototype, but $\core(G)$ is  not elementarily equivalent to $G$ when it is trivial. This is why we have   introduced the elementary core $\ecore(G)$   (Definition \ref{elcore}), equal to $\F_2$ when $\core(G)=1$;
 it is always elementarily equivalent to $G$ (Corollary \ref{equivcore}).
The elementary core embeds elementarily, provided that we take the free product of $G$ with a free group.

\begin{cor}
\label{coretour} 
If $G$ is a non-abelian torsion-free hyperbolic group, 
there exists $r\geq 0$  such that     $\ecore(G)$ has an elementary embedding   
 into $G*\bbF_r$. 
\end{cor}

\begin{proof}
This is a direct consequence of   Lemma \ref{stab_core}  
and   ``Tarski''. 
\end{proof}

We will see in Subsection \ref{core} that taking a free product with a free group is necessary in general (see Theorem \ref{contrex}).
When the core is one-ended, however, we deduce from   ``Tarski'' 
that $\ecore(G)$   embeds  elementarily   into $G$.

\begin{cor}
\label{cor_core_unbout}
  Let $G$ be a non-abelian torsion-free hyperbolic group.
  If its core $\core(G)$  is one-ended   (hence equal to 
  the  elementary core $\ecore(G)$), 
it admits an elementary embedding into $G$.
\end{cor}

\begin{proof}
  Using Corollary \ref{ett}, write $G*\bbF_r$ as a simple tower over a subgroup $P<G*\bbF_r$ isomorphic to $\core(G)$.
Since $P$ is one-ended, it is automatically conjugate to a subgroup of   $G$.
  By Theorem \ref{thcl} (or Theorem \ref{Tarski} and Remark \ref{tm}), $P$ is an elementarily embedded subgroup of $G$.
\end{proof}

Theorem \ref{thcl} also implies:

\begin{cor} \label{pasel}
A  hyperbolic prototype $P$ contains no proper elementarily embedded subgroup.
\end{cor}

\begin{proof}
Follows from $1\iff 2$   in Theorem \ref{thcl} and Lemma \ref{eqprot}.
\end{proof}

The converse of Corollary \ref{pasel} is not true: in 
  Subsections \ref{nst} and \ref{core} we will see   hyperbolic  groups which have no 
proper elementarily embedded subgroup but are not   prototypes.
To get an equivalence, we must consider groups of the form $P*F$   with $F$ free.

If $G=G_1*G_2$, we say that the   free factor $G_1$ is \emph{co-free} if $G_2$ is free. By   ``Tarski'', 
non-cyclic co-free free factors are elementarily embedded.

\begin{prop} \label{ssgpeprot}
A  hyperbolic  group $P$ is a prototype if and only if it has no cyclic free factor and, for any free group $F$,  any elementarily embedded subgroup $H\preceq_e  P*F$ is a co-free free factor.
\end{prop}

\begin{proof}
Suppose that $P$ has no cyclic free factor and is not a prototype. 
By Lemma \ref{eqprot}, $P$ has a Grushko factor $P_1$ which is an extended \'etage over 
 some $P'_1$, so by 
 Corollary \ref{ett} some $P_1*F$ is a simple \'etage over  a group   isomorphic to $P'_1*\bbF_2$. Writing $P=P_1*Q_1$, we see that 
  $P*F$ is a simple \'etage over a subgroup  $H$ isomorphic to $P'_1*Q_1*\bbF_2$. This group $H$ is elementarily embedded in   $P*F$  by   ``Tarski'',
   but is not a co-free free factor because   $P'_1\not\simeq P_1$  by Lemma \ref{dccg}.  
  
  Conversely,  assume that $P$ is a prototype (so has no cyclic free factor), and let $H$ be elementarily embedded in $P*F$. If $H$ is not a co-free free factor, then $P*F$ is an \'etage   of surface type by Theorem \ref{thcl}   and Remark \ref{zenbas},   contradicting Corollary \ref{infbout}.
\end{proof}

\subsection{Lattice-like properties}\label{llp}

The following result shows that the set of hyperbolic groups that are elementarily equivalent to a
given torsion-free hyperbolic group has lattice-like properties.

\begin{prop}
\label{prop_lattice} 
  Let $G_1,G_2$ be   non-abelian torsion-free hyperbolic groups,   with $G_1\equiv G_2$.
  \begin{itemize}
  \item There exists   a hyperbolic group  $G$ in which both $G_1$ and $G_2$ elementarily embed.
  \item There exists $G'$ (namely $\ecore(G_1)\simeq \ecore(G_2)$) and $r$ such that $G'$ elementarily embeds in $G_1*\bbF_r$ and $G_2*\bbF_r$.
  \end{itemize}
\end{prop}

\begin{rem}
  On the other hand, there exist $G_1\equiv G_2$ such that there is no $G'$ that elementarily embeds in both $G_1$ and $G_2$.
  Indeed, one may take $G_1=\bbF_2$ and $G_2=\pi_1(N_4)$, the fundamental group of the non-orientable surface of genus 4 (see Lemma \ref{lem_N4}).
\end{rem}

\begin{proof}
 By Theorem    \ref{sensfac}   we have $\core(G_1)\simeq \core(G_2)$, hence $\ecore(G_1)\simeq \ecore(G_2)$.

  The second assertion is a direct consequence of  
  Corollary \ref{coretour}.
 We now prove the first one. 
Replacing $G_i$ by $G_i*\bbF_{r_i}$ (in which $G_i$ embeds elementarily),
  Lemma \ref{stab_core} lets us assume without loss of generality 
  that $G_i$ is a simple tower over a subgroup $H_i$ isomorphic to $\ecore(G_i)$.
Choose an isomorphism between $H_1$ and $H_2$.
  By Remark \ref{rk_amalgam}, the amalgam $G=G_1*_{H_1=H_2}G_2$ is a simple tower over $G_2$, and symmetrically $G$
  is a simple tower over $G_1$. By   ``Tarski'', 
  $G_1$ and $G_2$ are elementarily embedded in $G$.
\end{proof}

\section{Partial elementary maps extend to the relative core}
\label{interp}

The main result of this section is Theorem \ref{lethm}, from which we shall deduce:
\begin{itemize}
\item
 Sela's characterization of elementary equivalence among hyperbolic groups;
\item Perin's theorem characterizing elementarily embedded subgroups of hyperbolic groups;
\item the homogeneity of free groups, proved by  Perin and the third named author, and  independently by Ould-Houcine  
\cite{PeSk_homogeneity,Houcine_homogeneity},   and more generally of free products of prototypes and free groups;
\item a characterization of tuples of elements of torsion-free hyperbolic groups with the same first-order properties in terms of towers, generalizing the characterization of tuples with the same first-order properties as  
bases of non-abelian free groups \cite{PeSk_homogeneity}. 
\end{itemize}
 The first two items are if and only if statements, 
  with one direction  relying on ``Tarski''.
     Dente-Byron and Perin  \cite{DePe_homogeneity} recently gave a complete characterization of homogeneous torsion-free hyperbolic groups,
   and  S.\ André proved that virtually free groups are almost homogeneous \cite{Andre_homogeneity}.
  
  The proof of Theorem \ref{lethm}  uses preretractions, introduced in \cite{Perin_elementary}. In the last subsection we  give a simpler proof
  of Propositions 5.11 and 5.12 of \cite{Perin_elementary}, stating   that preretractions yield retractions.

\begin{dfn}[Partial elementary map] \label{pem}
  Let $G,G'$ be two groups.

A \emph{partial elementary map} from $G$ to $G'$ consists of
  a subgroup   $A \subset G$ and a  map
  $f:A\rightarrow G'$ 
such that, for any
formula $\phi(x_1,\dots,x_n)$ and any $a_1,\dots,a_n$ in $A$, one has
$$G\models\phi(a_1,\dots,a_n)\text{ if and only if }G'\models\phi(f(a_1),\dots,f(a_n)).$$

 Note that $f$ is   then an isomorphism between $A$ and $f(A)$, and that $f\m:f(A)\ra G$ defines a partial elementary map from $G'$ to $G$.

When no confusion is possible, we denote   a partial elementary map  simply by $f:A\ra G'$, with $G$ implicit.
\end{dfn}

\begin{rem}
    The existence of a partial elementary map $f:A\ra G'$ implies $G\equiv G'$ (even
    if $A=\{1\}$).
  \end{rem}

  \begin{rem}\label{rem_inverse}
  One often defines  a partial elementary map with $A$    any subset of $G$ (possibly not  a subgroup). But such a map extends uniquely to a partial elementary map from $\grp{A}$ to $ G'$, so there is no loss of generality in requiring $A$ to be  a subgroup.
\end{rem}

\begin{rem} \label{extrem}
Partial elementary maps interpolate between elementary equivalence and elementary embeddings in the following sense:
\begin{itemize}
\item  When $A=\{1\}$, the 
  map $f:\{1\}\ra G'$   is a partial elementary map from $G$ to $G'$ if and only if $G\equiv G'$.

\item At the other extreme, if $A=G$, then $f:G\ra G'$
is a partial elementary map from $G$ to $G'$ 
if and only if it is 
an elementary embedding. 
\end{itemize}

\begin{rem}\label{typ}
 A partial elementary map may also be defined using the notion of a type.  
Recall that the \emph{type} of an  $n$-uple   $\bar{a}=(a_1,\dots,a_n)$ of elements of $G$, denoted $tp^{G}(\bar{a})$, is the set of all first-order formulas $\phi(x_1,\dots,x_n)$ in $n$ variables (without constants)
such that $G\models \phi(a_1,\dots,a_n)$.
Thus   $f:\grp{a_1,\dots,a_n}\rightarrow G'$ is a partial elementary map if and only if  $tp^{G}(\bar{a})=tp^{G'}(f(\bar{a}))$.

\end{rem}
\end{rem}

  \begin{dfn}[Relative prototype, relative core] \label{relprot}
    Let $G$ be a group, and let $A\subset G$ be a subgroup.
   We say that $G$ is a \emph{prototype relative to $A$} if it is not an extended étage relative to $A$ (see Definition \ref{relto}).
    
   A subgroup $C<G$ is a \emph{core of $G$ relative to $A$} if $C$ contains $A$,
   is a prototype relative to $A$, and $G$ is an extended tower over $C$ relative to $A$. 
   Relative cores exist by Lemma \ref{dccg}.  \end{dfn}

  We can now state the main result of this section.
 
\begin{thm}\label{lethm}
  Let $G,G'$ be non-abelian torsion-free hyperbolic groups,  and let $f:A\ra A'$ be an isomorphism between two subgroups of $G$ and $G'$ respectively.
  Let $C<G$ and $C'<G'$ be cores of $G$ and $G'$ relative to $A$ and $A'$ respectively. 
  
   If $f$ defines a partial elementary map from $G$ to $G'$, it  extends to an isomorphism $C\ra C'$  (the converse   follows from ``Tarski'' (Theorem \ref{Tarski}), as explained in Subsection \ref{unsens} below,  and implies uniqueness of relative cores).
\end{thm}

Note that the subgroups $A$ and $A'$ are  not assumed to be finitely generated, and are allowed to be trivial.
 See Proposition 6.2 of \cite{PeSk_homogeneity} for 
 the 
 one-ended case,
 and Theorem 3.10 of \cite{DePe_homogeneity} for the case where $G=G'$ and $A,A'$ are finitely generated.

Sela's characterization of elementary equivalence (Theorem \ref{sensfac})
  and Perin's description of elementary embedded subgroups (Theorem \ref{thcl}) 
both appear as extreme cases of this theorem (see Remark \ref{extrem});  in both theorems,    the ``if'' direction relies on ``Tarski''.

\begin{cor}[\cite{Sela_diophantine7}]\label{cor_eqv_elem}
 Let $G,H$ be non-abelian 
 torsion-free hyperbolic groups.   Then $G,H$ are 
   elementarily equivalent if and only if  their cores are isomorphic.
\end{cor}

\begin{proof} 
 Apply the theorem with $A=\{1\}$. The core of $G$ relative to $A$ is $\core(G)$.
\end{proof}

\begin{cor}[\cite{Perin_elementary} + ``Tarski'']
   \label{thcl2}
Let $G$ be a non-abelian torsion-free hyperbolic group,  and let $H<G$ be a    
non-abelian subgroup.
Then $H$ is elementarily embedded in $G$   if and only if $G$ is an extended tower over $H$ relative to $H$.
\end{cor}

This statement implies in particular that  any elementarily embedded subgroup of $G$ is finitely generated and hyperbolic. We give a proof assuming that we know this in advance. The general case requires an extra argument that will be given at the end of Subsection  \ref{pfsensfac}.

\begin{proof}[Proof of Corollary \ref{thcl2} (assuming  $H$ to be hyperbolic)]
  For the ``only if'' direction,   we take $A=H$ and we apply Theorem \ref{lethm} to the inclusion $A\hookrightarrow G$, which is a partial elementary map from $H$  (assumed to be hyperbolic) to $G$. 
We deduce that $H$ is the core of $G$ relative to $H$. In particular, by definition of the relative core, $G$ is an extended tower over $H$ relative to $H$, as required.  
The ``if'' direction was proved in Theorem \ref{thcl}  using ``Tarski'' (without assuming that $H$ is hyperbolic).
\end{proof}

Recall that a countable group $G$ is \emph{homogeneous} if, for every tuples $\ul u,\ul u'$
of elements of $G$   having the same  cardinality and the same type in $G$, there is an automorphism of $G$ that sends $\ul u$ to $\ul u'$.

  It immediately follows from Theorem \ref{lethm} that prototypes are homogeneous  (see Remark 6.9 of \cite{PeSk_homogeneity}).
More generally, we have:

\begin{cor}\label{cor_homogeneity}
    If a torsion-free hyperbolic group
  $G=P_1*\dots*P_p*\bbF_r$ is a free product of one-ended prototypes  $P_i$ with a free group,
  then $G$ is homogeneous.   In particular \cite{PeSk_homogeneity,Houcine_homogeneity},
  the free group $\bbF_r$ is homogeneous  for all  $r\geq 1$.
\end{cor}

\begin{proof}[Proof of Corollary \ref{cor_homogeneity}]
  We may assume that $G$ is not cyclic, as the result is clear in this case.
  Consider two $n$-uples $\bar{u}=(u_1,\dots,u_n)$ and $\bar{v}=(v_1,\dots,v_n)$ in $G$ having the same type.
  Let  $A=\grp {u_1,\dots,u_n}$ and $A'=\grp{v_1,\dots,v_n}$ be the subgroups  that they generate. 
  The fact that $\bar{u}$ and $\bar{v}$ have the same type means that the map $u_i\mapsto v_i$
  extends uniquely to an isomorphism $A\ra A'$ which is a partial elementary map from $G$ to itself (see Remark \ref{typ}).
  
  By Theorem \ref{lethm}, this map extends
  to an isomorphism $\phi:C\ra C'$, where $C$ and $C'$ are the cores of $G$ relative to $A$ and $A'$ respectively.  Since the Grushko factors $P_i$ of $G$ are prototypes, Corollary \ref{infbout} implies that $G$ is not an étage of surface type. 
It    is an extended tower over both $C$ and $C'$, so
  $C$ and $C'$ are co-free free factors of $G$,
  say $G=C*F$ and $G=C'*F'$ for some free subgroups $F,F'$ of $G$. Since $C$ and $C'$ are isomorphic, 
  $F$ and $F'$ have the same rank and
  $\phi$ extends to an automorphism of $G$.
  By construction $\phi(\bar{u})=\bar{v}$, so $G$ is homogeneous.
\end{proof}

 The following  immediate corollary   of Theorem \ref{lethm} generalizes Proposition 7.1 in \cite{PeSk_homogeneity}.    

\begin{cor}\label{type_char}
Let $\bar{a}$ be a tuple in a torsion-free hyperbolic group $G$, and let $C$ be the  
core of $G$   
  relative to $\langle\bar{a}\rangle$. Let $G'$ be a finitely generated group,   and $\bar{b}$ a tuple in $G'$. 
Then $tp^G(\bar{a})=tp^{G'}(\bar{b})$ if and only if     there is an isomorphism from $C$ to the core of $G'$ relative to $\bar{b}$ which sends $\bar{a}$ to $\bar{b}$.  
\qed
\end{cor}

\begin{proof}
  Apply Theorem \ref{lethm}, noting that $G'$ may be assumed to be hyperbolic:
if  $tp^G(\bar{a})=tp^{G'}(\bar{b})$, then $G'\equiv G$, so $G'$ is hyperbolic by \cite{Sela_diophantine7}; on the other hand,  $C$ is hyperbolic because $G$ is hyperbolic, so $G'$ is hyperbolic if its relative core is isomorphic to $C$.  
\end{proof}

\subsection{Isomorphic cores yield partial elementary maps}
\label{unsens}

 We start by proving the converse   of Theorem \ref{lethm}. It is an easy consequence of the following fact (based on ``Tarski'').

\begin{prop}\label{prop_pem}
  Let $G$ be a non-abelian torsion-free hyperbolic group, and  $A<H<G$ two subgroups.
  Assume that $G$ is an extended tower over  $H$ relative to $A$.
  
  If $H$ is non-abelian,   the identity map $\id_A:A\ra A$  defines  a partial elementary map from $H$ to $G$;
  in other words, every tuple of elements of $A$ has the same type in $H$ as in $G$.
  
  If $H$ is abelian (possibly trivial), then $\id_A$ defines a partial elementary map from $H*\bbF_2$ to $G$.
\end{prop}

\begin{rem}
The statement that  $\id_A$ defines a partial elementary map from $H*\bbF_2$ to $G$ is also valid if $H$ is non-abelian.
\end{rem}

\begin{proof}
  We first assume that $H$ is non-abelian.  By Corollary \ref{ett} and 
   Remark  \ref{rk_tour_rel}, 
  there exists a free group $F$ such that
  $G*F$ is a simple tower over a subgroup $H'<G*F$ isomorphic to $H$, and the isomorphism $\phi:H\ra H'$ is 
  the identity on $A$. In particular, tuples in $A$ have the same type in $H$ and in $H'$.

  By ``Tarski'' (Theorem \ref{Tarski}) $H'$ is elementarily embedded in $G*F$, and $G$ is elementarily embedded in $G*F$.
   This implies that tuples in $H'\cap G$ (in particular, tuples in $A$) have the same type in $H'$(hence in $H$)  as in $G$.

  In the  general case, 
  $G*\bbF_2$ is an extended tower over $H*\bbF_2$, so by the previous case
  $\id_A$ defines  a partial elementary map
  from $H*\bbF_2$ to $G*\bbF_2$, hence from   $H*\bbF_2$ to $G$ since the embedding $G\hookrightarrow G*\bbF_2$ is elementary.
\end{proof}

\begin{proof}[Proof of the converse of 
Theorem \ref{lethm}.]
We assume that $f$ extends to an isomorphism $C\to C'$, and we show that it is partial elementary.
  Consider any $n$-uple $\bar{a}=(a_1,\dots,a_n)$ of elements of $A$.
  The type of $\bar{a}$ in $C$ agrees with the type of $\bar{b}=f(\bar{a})$ in $C'$ because $f$ extends to an isomorphism  (if $C$ is abelian, we use types in $C*\bbF_2$ and $C'*\bbF_2$ rather than $C$ and $C'$ in this argument).
  Now the type of $\bar{a}$ in $C$ 
  agrees with the type of $\bar{a}$ in $G$ by Proposition \ref{prop_pem},
  and similarly for the type of $\bar{b}$.
This  implies 
that $f$ is a partial elementary map from $G$ to $G'$.
\end{proof}

The next subsections are devoted to the proof of Theorem \ref{lethm}.

\subsection{$A$-groups and plain groups}\label{agpe}
Theorem \ref{lethm} is about  hyperbolic groups $G,G'$ with an isomorphism $f:A\ra A'$ between subgroups $A<G$ and  $A'<G'$.
Using $f$, we shall view both $G$ and $G'$ 
as groups with a specified embedding of $A$, which we call \emph{$A$-groups} (see for instance \cite{BMR_algebraicI}). 

The  core $C$ of $G$ 
relative to $A$   is an $A$-group. It has a Grushko decomposition $C=C_1*\dots * C_p$   relative to $A$, and we may assume 
  $A\inc C_1$ (there is no   extra free group in the decomposition because $C$ is a core, but $C_1$ may   be free
  with $A$ not contained in  a proper free factor,
  or even cyclic if $A\simeq\Z$). Thus $C_1$ is an $A$-group, but we 
  will also work with the other factors $C_i$. These factors are not $A$-groups,  they have no extra structure, we call them \emph{plain groups}.  We will always assume $A\ne \{1\}$, and the reader interested only in Corollary \ref{cor_eqv_elem} may forget about $A$-groups (and skip this subsection),  thinking that all groups (including $C_1$) are plain groups. 

    All groups considered in this proof will be either $A$-groups or plain groups. 
      To unify notation, we will view a plain group (i.e.\ a group with no extra structure) as a $\{1\}$-group,  and everything will be understood to be \emph{relative} (to $A$, or to $\{1\}$, i.e.\ non-relative). This applies in particular to splittings, one-endedness, Grushko decompositions, JSJ decompositions.

        Flexible vertices of the canonical  cyclic JSJ decomposition relative to $A$ are relative QH vertices (see \cite[Def.~5.13, Th.~9.5]{GL_JSJ}).
      Because of the acylindricity properties of the tree of cylinders defining the canonical JSJ decomposition, they are of \emph{relative surface type}
      in the following sense.
      
First, a surface-type vertex group $G_v=\pi_1(\Sigma)$ in the sense  of  Definition \ref{stype} remains  surface-type in the relative sense provided that  its intersection with any conjugate of $A$ is contained in a boundary subgroup of $\pi_1(\Sigma)$. 

   If $A\simeq\Z$, another possibility is that  $G_v=\pi_1(\Sigma)$,  with $\Sigma$ a compact surface having a free boundary component $C\inc \bo \Sigma$ which ``carries'' $A$; more precisely, 
$\pi_1(C)$ contains a conjugate of $A$ (as a subgroup of finite index),
incident edge groups are fundamental groups of   components of $\bo \Sigma$ other than $C$, 
and this induces a bijection between the set of incident edges and the set of boundary components of $\Sigma$ other than $C$. 

\begin{rem}\label{rem_retractable_relatif}
    If $G$ has a splitting with a relative surface-type vertex $v$ of the second kind, there exists no \ndbpm\ $\pi_1(\Sigma)\ra G$ by Remark \ref{bolib}.
This implies that, if $v$ is a retractable relative surface-type vertex, then collapsing all edges not adjacent to $v$ yields a retractable centered splitting, 
with no need to change the definition of a centered splitting.
\end{rem}

      If $G$  is an $A$-group, 
      we say that it is an extended \'etage or an extended tower over a subgroup $H$  only if $H$ contains $A$ and the \'etages are relative to $A$.
      A plain group cannot be an \'etage over an $A$-group.
      As in Definition \ref{relprot}, an $A$-group is a prototype if it is not an extended \'etage, 
      and the core of an $A$-group is as in Definition \ref{relprot}.

      If $G$ is an $A$-group, a \emph{morphism} $g:G\to G'$ is a homomorphism equal to the identity on $A$. In particular, $G'$ has to be an $A$-group: there is no morphism from  an $A$-group  to a plain group. Embeddings, isomorphisms, automorphisms defined on $A$-groups always have to be morphisms.
 We note that, since the étages are relative to $A$, the retractions associated to \'etages are the identity on $A$, hence are morphisms.

\begin{rem}\label{rk_inverse} 
If  $G$ is an $A$-group and $G'$ is a plain group,
there is no restriction on morphisms $g:G'\ra G$, but if $g$ is bijective $g\m$ is not a  morphism.
 \end{rem}

For $A$-groups, formulas with constants in $A$ make sense, and
we   say that two $A$-groups $G$ and $G'$ are \emph{elementarily equivalent as $A$-groups},  denoted  $G\equiv_A G'$,
if they satisfy the same formulas with constants in $A$.
Similarly, if $G,G'$ and $f:A\to A'$ are as in Theorem \ref{lethm}, we view $G'$ as an $A$-group by identifying $A'$ to $A$ thanks to $f$, and 
$G\equiv_A G'$ if and only if $f$ is a partial elementary map from $G$ to $G'$.

One can then reformulate 
Theorem \ref{lethm} as follows (with $G,G'$ non-abelian hyperbolic groups).

\begin{thm}\label{lethm2}
  Let $G,G'$ be two plain groups, or two $A$-groups with $A\neq\{1\}$.
  Let $C,C'$ be their cores. 
  If $G$ is elementary equivalent to $G'$, 
  then $C$ and $C'$ are isomorphic. 
\end{thm}

In the case of $A$-groups, core, elementary equivalence, isomorphism must be understood in the context of $A$-groups, as explained above.

\subsection{Preliminary choices}\label{sec_choix}
In this subsection, we make some particular choices  in the Grushko factors of the  cores of $G$ and $G'$.
More generally, we consider a plain group or an $A$-group $J$ which satisfies the following condition:

\begin{quote}
$(*)$ $J$ is a one-ended torsion-free hyperbolic group; if it is a plain group, it     is not isomorphic to the fundamental group of a non-exceptional closed surface
\end{quote}
(recall that, if $J$ is an $A$-group,   one-endedness is understood as relative to $A$).

We denote by $\Gamma_J$ the canonical cyclic JSJ decomposition of $J$ (relative to $A$ if $J$ is an $A$-group). 
We will also use   the modified JSJ decomposition $\tilde\Gamma_J$
described in Subsection \ref{sec_modified}. 
 The key difference with 
 $\Gamma_J$ is     that all exceptional surfaces
 appearing in surface-type vertices of $\tilde\Gamma_J$ have finite mapping class group.
 This will be important in 
  the proof of Lemma \ref{shortening}.

We now explain how to choose certain elements 
 $d_J\in J $ and finite subsets  $D_{J'\to J} \inc J$, 
for  $J'$ also satisfying $(*)$.
 In the next subsection we will use Sela's shortening argument to choose finite subsets 
$F_{J,H} \inc J$,  for $H$ hyperbolic. The key properties of these elements and subsets will be summed up in Lemma \ref{lem_choix}.

We  first claim that 
  there is  a non-trivial  
  element $d_J\in J$ such that the orbit of $d_J$ under $\Aut(J)$ is  a finite union of conjugacy classes, with $d_J\in A$ if $J$ is an $A$-group;
  moreover, unless $J$ is a plain group isomorphic to $\pi_1(N_3)$, the fundamental group of the   closed non-orientable surface of genus 3, 
  $d_J$ is elliptic in any cyclic splitting of $J$.

This is obvious if $J$ is an $A$-group, 
as one can take for $d_J$ any non-trivial element of $A$  (recall that automorphisms and splittings   are relative to $A$  if $J$ is an $A$-group).
Otherwise, $J$ is a one-ended plain group, so consider   its   (non-modified) canonical cyclic JSJ decomposition $\Gamma_J$. 

If  $ \Gamma_J$ is non-trivial, any non-trivial element contained in an edge group of $ \Gamma_J$ has the desired properties because $ \Gamma_J$ is 
universally elliptic and $\Out(J)$-invariant, so automorphisms preserve the set of conjugacy classes of edge groups. 

If $ \Gamma_J$ is  trivial, there are  two possibilities. If  $J$ is rigid (it has no cyclic splitting), then $\Out(J)$ is finite by  Paulin's theorem \cite{Pau_arboreal} and   Rips' theory of $\R$-trees \cite{BF_stable},
so any $d_J\in J\setminus\{1\}$ has the required properties. The other possibility is that $J$  is a closed surface group. 
By condition $(*)$,
$J=\pi_1(N_3)$  and we take $d_J=g_{N_3}^2$, the square of the generator of $\pi_1(c_{N_3})$, with $c_{N_3}$ the special curve of $N_3$   (see Subsection \ref{sec_surfaces}). 

This shows the existence of $d_J$.
We choose such an element $d_J$ in an arbitrary way for every group $J$ 
satisfying $(*)$. 

We note that, if $H$ is hyperbolic, then, up to conjugacy in $H$, there are only finitely many possibilities
for the  images of $d_J$ under injective morphisms $\mu:J\to H$.
This is obvious if $J$ is an $A$-group because  $\mu$ being a 
 morphism means that $H$ is an $A$-group and $\mu$ is the identity on 
 $A$.
Otherwise, $J$ is a one-ended plain group and  there are finitely many possibilities for $\mu(J)$   up to conjugacy
by a result due to Gromov   (Thm.\ 5.3.C' of \cite{Gromov_hyperbolic}, \cite{Delzant_image}); 
moreover, if  two embeddings of $J$ into $H$ have the same image, they differ by  an automorphism of $J$, and finiteness follows from the   choice of $d_J$.

If $J$ satisfies $(*)$ 
and $H$ is hyperbolic, 
we define a finite set  $D_{J\ra H}\subset H\setminus\{1\}$ by choosing a representative of the conjugacy class of each image of $d_J$ by an injective morphism $J\ra H$ 
(we define $D_{J\to H}=\es$ if $J$ does not embed in $H$, in particular if $J$ is an $A$-group and $H$ is a plain group).

We consider in particular $D_{J'\to J}\subset J$, for $J,J'$ satisfying $(*)$,  with $D_{J'\to J}=\es$ if $J'$ is an $A$-group and $J$ is a plain group.

\begin{lem} \label{aufond}
Let  $J$ be a group satisfying $(*)$. 
 Assume that $J$ is contained in a group with a non-exceptional centered splitting $\Gamma$ (relative to $A$ if $J$ is an $A$-group). 
Then  the following elements of $J$ are contained in a conjugate of a bottom group of $\Gamma$: 
\begin{enumerate}\renewcommand{\theenumi}{(\arabic{enumi})}\renewcommand{\labelenumi}{\theenumi}
\item  the element $d_J$; 
\item the edge groups of $\tilde\Gamma_J$, and its vertex groups except those which are non-exceptional surface-type;
\item    each element of $D_{J'\to J}$, for any $J'$ satisfying $(*)$ and embedding into $J$. 
\end{enumerate}
\end{lem}

 \begin{proof}
  Given an edge group or a rigid vertex group of $  \Gamma_J$, or the group $A$, it is elliptic in every (relative) cyclic splitting of $J$.
   Thus (up to conjugacy) it is contained in a bottom group of $\Gamma$, or it is cyclic and contained in a maximal boundary subgroup of  the central vertex group  (see \cite{GL_JSJ}, Prop.\  5.21), hence also contained in a bottom group.
Unless $J$ is a plain group isomorphic to $\pi_1(N_3)$, this argument also applies to $d_J$, which is universally elliptic, so (1) holds.
If $J'$ is not a plain group isomorphic
   to $\pi_1(N_3)$, this argument applied in $J'$  proves 
   (3).
 
   Edge or vertex groups of the   modified JSJ decomposition $\tilde\Gamma_J$  as in (2)
   are edge groups or rigid vertex groups of $  \Gamma_J$, or are contained in an exceptional surface-type vertex group of $  \Gamma_J$. 
Lemma \ref{lemcle3} then implies that (2) holds.

If $J$ is a plain group isomorphic to $\pi_1(N_3)$, the splitting $\tilde\Gamma_J$ is dual to the unique decomposition of $N_3$ into a punctured torus and a M\"obius band, and $d_J=g_{N_3}^2$ 
is conjugate into a bottom group by Lemma \ref{lemcle3}, so (1) holds. This argument also applies if $J'$ is a plain group isomorphic to $N_3$
so  (3)   holds in all cases. This concludes the proof.
\end{proof}
  
\subsection{Using the shortening argument}
Given an arbitrary hyperbolic group $H$, we are now going to choose another finite subset $F_{J,H}\subset J$
thanks to  Sela's shortening argument. To state it, we need the following definition.

\begin{dfn} [Related morphisms, \cite{Perin_elementary} Definition 5.15] \label{related}  
Let $H_1$, $H_2$ be plain groups or $A$-groups.
  Let $\Gamma$ be a splitting of  $H_1$.
  We say that two morphisms $\mu,\mu':H_1\to H_2$ are \emph{related} (with respect to $\Gamma$) 
if:
\begin{itemize}
\item for each vertex group $G_v$ of $\Gamma$ which is a non-exceptional surface-type vertex group,
  $\mu(G_v)$ is non-abelian if and only if $\mu'(G_v)$ is non-abelian;
\item if  $K$ is another vertex group    of $\Gamma$, or an edge group of $\Gamma$,  there is  $h\in H_2$   
such that  $\mu'$ agrees  with  $ad_h\circ \mu$  on $K$ (with $ad_h$ denoting conjugation by $h$).
\end{itemize}
\end{dfn}

\begin{rem}
If $H_1$ is an $A$-group, the splitting $\Gamma$ is relative to $A$, and  surface-type is to be understood in the    relative sense (as defined in  Subsection \ref{agpe}).
  In this case, the definition is empty if $H_2$ is a plain group, as there is no morphism from an $A$-group to a plain group.
\end{rem}

  In the following lemma, we do not assume that $H_2$ is a hyperbolic group, only that $H_2$ is a subgroup (maybe not finitely generated)
of a hyperbolic group. This generality will be useful in the proof of Corollary \ref{thcl2}.

\begin{lem}[{\cite[Th.\ 1.26]{Sela_diophantine7}, see also \cite[Prop.\ 4.3, Prop.\  4.13]{Perin_elementary}, \cite[Prop.\ 5.1]{GLS_finite_index}}] \label{shortening}\ \newline  
  Let $H_1$, $H_2$ be plain groups or $A$-groups, with $H_1$ a   one-ended torsion-free hyperbolic group
  and $H_2$   a subgroup of a torsion-free hyperbolic group.

There exists a non-empty finite set $F_{H_1,H_2}\subset H_1\setminus\{1\}$ such that,   given any non-injective 
morphism $\mu :H_1\ra H_2$,
there exists $\mu':H_1\ra H_2$ with the following properties: it is  related to $\mu$ with respect to the modified JSJ decomposition $\tilde \Gamma_{H_1}$
of $H_1$,  
and $\ker \mu'\cap F_{H_1,H_2}\neq\es$.
\end{lem}

\begin{proof}
We explain how to deduce this statement from  the shortening argument (\cite[Prop.\ 4.13]{Perin_elementary} or \cite[Th.\ 4.4]{PeSk_homogeneity}) which says
that, if  $H_1,H_2$ are  hyperbolic with  $H_1$ one-ended, there exists a finite set  $F_{H_1,H_2}\subset H_1\setminus\{1\}$
such that, for any non-injective morphism $\mu:H_1\ra H_2$, there exists an automorphism $\sigma$ of $H_1$ such that
$\mu'=\mu\circ\sigma$ kills an element of $F_{H_1,H_2}$ (with $\sigma_{|A}=\id$ if $H_1$ is an $A$-group);  note that this fact also holds   if $H_2$ is only assumed to be a subgroup of a hyperbolic group.

 We have to prove that $\mu$ and $\mu'$ are related (see Definition \ref{related}).
The first condition clearly holds, so we
check the second one.

The modified JSJ decomposition $\tilde\Gamma_{H_1}$ being invariant under automorphisms,
there exists a finite index subgroup of $\Aut^0(G)\subset \Aut(G)$ consisting of automorphisms that act as a conjugation
on all edge groups and all surface-type vertex groups whose mapping class group is finite  (see \cite[Th.~1.9]{Sela_structure}, \cite[Th.~1.1]{Lev_automorphisms}).
By Paulin's argument  \cite{Pau_arboreal}  and Rips theory \cite{BF_stable}, there is a further finite index subgroup $\Aut^1(G)
\subset \Aut^0(G)$
consisting of automorphisms acting as a conjugation on 
rigid vertex groups.
The point of using the modified JSJ decomposition $\tilde \Gamma_{H_1}$ (see Section \ref{sec_modified}, or \cite[Prop.\ 5.1]{GLS_finite_index}) is precisely to ensure that all exceptional surface-type vertex groups have finite mapping class group, so any $\mu\circ\sigma$ with  $\sigma\in \Aut^1(G)$ is related to $\mu$. After enlarging $F_{H_1,H_2}$, we may require  $\sigma\in \Aut^1(G)$ in the shortening argument and we obtain $\mu'$ related to $\mu$.
\end{proof}

\begin{rem}
  The lemma still holds if $H_1$ is any one-ended CSA group, using the fact that it is enough to use the modular group of $H_1$
  in the shortening argument.
\end{rem}

\begin{rem}
Being related to $\mu$ is weaker than being of the form $\mu\circ \sigma$, but
it may be expressed in first-order logic  (see \cite{Perin_elementary}, Lemma 5.18).
\end{rem}

The lemma allows us to choose $F_{J,H}\inc J$, for $J$ satisfying $(*)$ and $H$ a   subgroup of a hyperbolic group.

To sum up: 
\begin{lem}\label{lem_choix}
Let  $J,J'$ satisfy $(*)$, and let $H$ be a subgroup of a hyperbolic group. The element  $d_J\in J\setminus\{1\}$     and the   finite subsets $D_{J'\to J}$,  $F_{J,H}$ of $J\setminus\{1\}$ 
satisfy the following properties:
\begin{enumerate}
\item  $F_{J,H}\ne\es$.
\item
  For any injective morphism $\mu:J'\ra J$,
the element   $\mu(d_{J'})$ is conjugate to an element of $D_{J'\to J}$. 
\item 
Any non-injective morphism $J\ra H$
is related (with respect to $\tilde \Gamma_J)$ to one that kills an element of $F_{J,H}$. 
\qed
\end{enumerate}
\end{lem}

We end this subsection with the following lemma.

\begin{lem}\label{relat2}
  Assume that $G$ is an extended tower over $H$, and $J$ satisfies $(*)$. 
  Any injective morphism $\mu:J \ra G$  is related (with respect to $\tilde \Gamma_J$) to  a morphism $\lambda:J\ra G$ which is either non-injective,
  or 
    injective 
   with values in $H$.
  Moreover, $\mu(y)$ and $\lambda(y)$ are conjugate if $y=d_J$ or $y\in D_{J'\to J}$ with $J'$ satisfying $(*)$.
\end{lem}

\begin{proof}
  Let $G=G_0> G_1 > \dots > G_k=H*F$ be a sequence of subgroups
  associated to étages of surface type as in Proposition \ref{prop_retractions} and Remark  \ref{prop_retractions_rel}, with $F$ free and with retractions
 $\rho_i:G_i\ra G_{i+1}$.
   If $J$ is an $A$-group, 
 so are  $G$ and $H$ 
 (otherwise, there is no $\mu$),
 the tower is relative to $A$,  and   the  retractions $\rho_i$ are the identity on $A$, so   are  morphisms.  
   
 Let  $\mu_i=\tau_i\circ\rho_{i-1}\circ\dots \circ \rho_0 \circ \mu: J\ra G$, with $\tau_i:G_i\into G$ the inclusion and $\mu_0=\mu$.
 It is a morphism.
  We claim that, \emph{if $\mu_i$ is injective, then $\mu_{i+1}$ and $\mu_i$ are related,
and   $\mu_{i+1}(y)$, $\mu_{i}(y)$ are conjugate  for any $y\in \{d_J\}\cup D_{J'\to J}$.}

 This claim implies the lemma.   Indeed, if $\mu_k$ is injective,   applying the claim repeatedly shows that it  is related to $\mu$.
 Since $J$ is  (relatively) one-ended, its image is contained in $H$ up to conjugacy 
and the lemma follows in this case.
If $\mu_k$ is not injective,   there is a first index $i_0\geq 1$ such that $\mu_{i_0}$ is not injective,
and the claim implies that $\mu_{i_0}$ is related to $\mu$.
The moreover part of the statement immediately follows from the claim.

We now prove the claim. We let $\Lambda_i$  be the centered splitting of $G_i$ associated to the étage,
with central vertex group $Q_i$.
We also consider $y\in \{d_J\}\cup D_{J'\to J}$.
We check both conditions of relatedness (Definition \ref{related}), starting with the second one.

Since the bottom groups of $\Lambda_i$ are contained in $G_{i+1}$ up to conjugacy,
the retraction $\rho_i$ agrees with a conjugation on any conjugate of a bottom group of $\Lambda_i$, 
so the second condition holds thanks to Lemma \ref{aufond} saying that the image of rigid and non-exceptional surface groups of $\tilde \Gamma_J$ under $\mu_i$
are contained (up to conjugacy) in bottom groups of $\Lambda_i$.
This lemma also says that $\mu_i(y)$ is contained in a conjugate of a bottom group of $\Lambda_i$ so
$\mu_{i+1}(y)$ and $\mu_{i}(y)$ are conjugate in $G$.

If now $Q_J=\pi_1(S)$ is a non-exceptional surface-type vertex group of $\tilde\Gamma_J$,
Lemma \ref{lemcle2} says that (up to conjugacy) its image by $\mu_i$ is contained in a bottom group of $\Lambda_i$ or contains a finite index subgroup of $Q_i$.
In either case its image by $\rho_i$ is non-abelian  
and the first condition is satisfied. This proves the claim, hence also the lemma.
\end{proof}

\subsection{A first-order criterion for isomorphism}\label{foc}

Let $C,C'$ be the cores of $G,G'$,   as in Theorem \ref{lethm2}. They are plain groups or   $A$-groups with $A\neq\{1\}$ (they may be trivial or equal to $A$). Recall that, in $A$-groups, everything is relative to $A$. 

We denote by $C=C_1*\dots*C_p$ and $C'=C'_1*\dots*C'_{p'}$ their   Grushko decompositions.
In the relative case, we assume 
  $A<C_1$ and $A<C'_1$, so that  $C_1$ and $C'_1$ are $A$-groups; all other factors are plain groups.  All groups $C_i$ and $C'_j$ 
 satisfy  $(*)$.  They are prototypes, but this will not be used  in this subsection.

Our goal is to prove that $C$ and $C'$ are isomorphic. 
We shall use cohopfianity of one-ended torsion-free hyperbolic groups $H$ (\cite{Sela_structure},  \cite[Corollary 4.2]{PeSk_homogeneity} for the relative case):
any injective morphism $H\ra H$ is an isomorphism.
In particular,
if $C_i$ embeds into $C'_j$ and $C'_j$ embeds into $C_i$, then both embeddings are isomorphisms. 

Cohopfianity does not hold when the groups decompose as a free product, so
we cannot conclude that $C$ and $C'$ are isomorphic if each  embeds into the other.
We introduce a technical strengthening of the notion of embedding that will allow us to conclude that $C$ and $C'$ are isomorphic.

\begin{dfn}[Engulfing]\label{dfn_engulf}
Let $H$ be a plain group or an $A$-group.
  Say that $C=C_1*\dots*C_p$ \emph{engulfs} in   $H$
  if there exist injective morphisms $\mu_i:C_i\ra H$, for  $i\in\{1,\dots,p\}$, such that,
  for each $i_0\neq i_1\in\{1,\dots,p\}$, 
  $$\mu_{i_1}(d_{C_{i_1}})\text{ is not conjugate to an element in } \mu_{i_0}(D_{C_{i_1}\to C_{i_0}}).$$
\end{dfn}

A main feature of this definition is that it prevents the images of $\mu_{i_0}$ and $\mu_{i_1}$ 
from being conjugate in $H$. 

\begin{rem}\label{remeng}
  Recall that $D_{C_{i_1}\to C_{i_0}}$ is empty if       $C_{i_1}$ does not embed into $C_{i_0}$, so  
  the condition is automatically  satisfied in this case. Also note that the trivial group engulfs in any $H$, and $A$ (viewed as an $A$-group) engulfs in any $A$-group. If $C$ is one-ended ($p=1$), engulfing is equivalent to embedding.
\end{rem}

\begin{lem}\label{lem_bi_engulf} 
  If $C,C'$ engulf in each other, then $C$ and $C'$ are isomorphic.
\end{lem}

\begin{rem} 
  We do not claim that the morphisms $\mu_i$ from the definition of engulfing extend to an isomorphism $C\ra C'$,
 though they do after changing each of them by a conjugation.
\end{rem}

\begin{proof}
Consider $\mu_i:C_i\ra C'$ and $\mu'_j:C'_j\ra C$ as in Definition \ref{dfn_engulf}. 
The image $\mu_i(C_i)$ is isomorphic to $C_i$, hence one-ended
(as a plain group or an $A$-group accordingly)
  and contained in a conjugate of a unique $C'_j$.  After composing $\mu_i$ with a conjugation we may assume $\mu_i(C_i)\inc C'_j$.
  
  Define
 $\alpha:\{1,\dots,p\}\ra \{1,\dots,p'\}$  by saying that $\mu_i(C_i)\inc  
  C'_{\alpha(i)}$, and
$\alpha':\{1,\dots,p'\}\ra \{1,\dots,p\}$ similarly.
  Note that $\alpha(1)=\alpha'(1)=1$ in the context of $A$-groups since $\mu_1,\mu'_1$ are the identity on $A$.

If $\alpha$ and $\alpha'$ are bijective, every index $i\in\{1,\dots,p\}$ is $(\alpha'\circ\alpha)$-periodic and
the co-Hopf property implies that $C$ and $C'$ are isomorphic.

Thus assume that one of them, at least, is not injective, and argue towards a contradiction.
Up to exchanging the roles of $C$ and $C'$, 
there exist  $i_0\neq i_1\in \{1,\dots,p\}$ having the same image $j$ under $\alpha$, with 
$i_0$   periodic under $\alpha'\circ\alpha$.

The  co-Hopf property applied to the sequence of   
embeddings $$C_{i_0}\hookrightarrow C'_{\alpha(i_0)}\hookrightarrow C_{(\alpha'\circ\alpha)(i_0)}\hookrightarrow \cdots \hookrightarrow C_{i_0}$$
shows that $\mu_{i_0}$ is a bijection between $C_{i_0}$ and 
$ C'_{\alpha(i_0)}$.
 In view of Remark \ref{rem_inverse} note that, in the relative case, $\alpha(i_0)=1$ implies $i_0=1$, so that
$\mu_{i_0}\m$ is a morphism. 

Now  $\mu_{i_0}\m \circ \mu_{i_1}$ is an embedding of $C_{i_1}$ 
into
  $C_{i_0}$, hence maps $d_{C_{i_1}}$ into $D_{C_{i_1}\to C_{i_0}}$ up to conjugacy by the definition of $D_{C_{i_1}\to C_{i_0}}$
  (see Lemma \ref{lem_choix}).
Thus $\mu_{i_1}(d_{C_{i_1}})$ is conjugate to an element of $\mu_{i_0}(D_{C_{i_1}\to C_{i_0}})$, contradicting the definition of engulfing.
\end{proof}

As we cannot express injectivity in first-order logic,
the engulfing property is not a first-order property.
We therefore introduce another strengthening which, as we will see in Lemma \ref{lem_engulf_ee}, has a first-order interpretation.

\begin{dfn}[Strong engulfing]\label{dfn_strong_engulf}
  Say that  $C=C_1*\dots*C_p$ \emph{strongly engulfs} in $H$ with respect to  non-empty finite sets $F_i\subset C_i\setminus\{1\}$
  if there exist 
  morphisms $\mu_i:C_i\ra H$ such that:
  \begin{enumerate}
  \item there is no morphism $\lambda_i$ related to $\mu_i$ (with respect to the modified JSJ decomposition $\Tilde \Gamma_{C_i}$, see Definition \ref{related}) such that
    $\ker \lambda_i$ contains an element of $F_i$;
  \item for all $i_0\neq i_1\in\{1,\dots,p\}$, the element $\mu_{i_1}(d_{C_{i_1}})$ is not conjugate to an element of  $\mu_{i_0}(D_{C_{i_1}\to C_{i_0}})$.
  \end{enumerate}
\end{dfn}

 The following lemma explains the terminology.
 The fact that the core of $G$ strongly engulfs in $G$ 
is  a key point in the proof of Theorem \ref{lethm} and will be proved in  Lemma \ref{lem_proto_engulf}.

\begin{lem}\label{rem_strong}
  If $C $ strongly engulfs in $H$ with respect to the finite sets $F_{C_i,H}$, then $C$ engulfs in $H$. 
\end{lem}

\begin{proof}
The choice of the sets $F_{C_i,H}$ (see Lemma \ref{lem_choix}) guarantees   that any non-injective morphism $C_i\ra H$ is related to one that kills an element of $F_{C_i,H}$, so the first condition of strong engulfing implies that the $\mu_i$'s are  injective. 
\end{proof}

We now explain why   strong engulfing is a first-order property. 

\begin{lem}\label{lem_engulf_ee} 
Fix $C$ and finite subsets $F_i\subset C_i\setminus\{1\}$ for $1\le i\leq p$. Let $H$ and $H'$ be elementarily equivalent  plain groups,  or $A$-groups which are elementarily equivalent  as $A$-groups (i.e.\ 
$H\equiv_A H'$).
  If $C$ strongly engulfs in $H$ with respect to the $F_i$'s, 
  then $C$ strongly engulfs in $H'$ with respect to the $F_i$'s. 
\end{lem}

\begin{proof}
  This  is shown as in   \cite[Proposition 6]{Sela_diophantine6} or \cite[Lemma 5.18]{Perin_elementary}.
  We  sketch  the argument for completeness.

Given a finite presentation $$J=\grp{s_1,\dots,s_r \mid R_1(s_1,\dots,s_r)= \dots=R_l(s_1,\dots,s_r)=1}$$ 
of a group $J$, 
homomorphisms $\mu:J\ra H$ can be encoded by $r$-uples 
$(h_1,\dots,h_r)\in H^r$
satisfying the equations $R_j(h_1,\dots,h_r)=1$ via the one-to-one correspondence $$\mu\mapsto h_\mu=(\mu(s_1),\dots, \mu(s_r)).$$
If $J$ and $H$ are $A$-groups, let $a_n=w_n(s_1,\dots,s_r)$ be a (finite or infinite) sequence of generators of $A$,
written as words in the generators of $J$.
Then $A$-morphisms can be encoded by imposing additional equations with constants in  the subgroup $A\inc H$, namely $w_n(h_1,\dots,h_r)=a_n$.
Equational noetherianity of hyperbolic groups \cite[Th.\ 1.22]{Sela_diophantine7} says that this set of equations
is equivalent to a finite subset if $H$ is hyperbolic.

Now, given elements $z,z'\in J$ or a finitely generated subgroup $Z<J$, the facts that
$\mu$ kills $z$, that $\mu(Z)$ is abelian,   that $\mu(z)$ is conjugate to $\mu(z')$, all translate into first-order formulas on the tuple $h_\mu$.
Similarly, the fact that two morphisms $\mu$ and $\mu'$ differ by a conjugation on $Z$ is expressed by a first-order formula on
 $(h_\mu,h_{\mu'})$.

It follows that relatedness translates into a first-order property, and
that  $C=C_1*\dots*C_p$ strongly engulfs  in a hyperbolic group $H$ (with respect to the $F_i$'s)  if and only if some first-order property with coefficients in $A$
is satisfied by $H$.
Since $H$ and $H'$ are elementarily equivalent, and    $C$ strongly engulfs 
in $H$,
 it also does in $H'$.
  \end{proof}

\begin{rem}
  Instead of using equational notherianity, one may introduce a large finitely generated subgroup $A_0\inc A$ and use Lemma 4.20 of \cite{Perin_elementary} and Corollary 4.5 of \cite{PeSk_homogeneity}.
\end{rem}

\subsection{Proof of  Theorem \ref{lethm2} and Corollary \ref{thcl2}}
\label{pfsensfac}

Let $G,G'$ be two plain groups, or two $A$-groups with $A\neq\{1\}$.  Recall that, in $A$-groups, everything is relative to $A$. 

We first show Theorem \ref{lethm2}:
we assume that $G$ and $G'$ are elementarily equivalent (as plain or $A$-groups accordingly), and we show that their cores $C=C_1*\dots*C_p$ and $C'=C'_1*\dots*C'_{p'}$
  are isomorphic. 

  There are four successive steps in the proof, among which steps 2 and 4 have already been treated.
\begin{enumerate}
\item Prove that $C$ strongly engulfs in $G$ with respect to the sets $F_{C_i,G'}$, using  the fact that $C$ is a prototype.

\item It follows that $C$ strongly engulfs in $G'$ with respect to the sets $F_{C_i,G'}$ because engulfing of $C$ is a first-order property (Lemma \ref{lem_engulf_ee}).

\item Deduce that $C$ (weakly) engulfs in the core $C'$.

\item 
  One concludes that $C$ is isomorphic to $C'$: by symmetry of the argument, $C$ and $C'$ engulf in each other, so Lemma \ref{lem_bi_engulf} applies.
\end{enumerate}

There remains to prove steps 1 and 3.
Note that step 3 follows from Lemma \ref{rem_strong} in the special case when  $G'$ is a prototype ($C'=G'$).

Step 1 (the core of $G$ strongly engulfs in $G$)  relies on the fact that $C$ is a prototype and uses results of Perin about preretractions explained in   Subsection \ref{prtor}.

\begin{lem}\label{lem_proto_engulf}
  Let $G$ be an extended tower over  a prototype $C=C_1*\dots*C_p$. 
  For all $i\leq p$, consider any non-empty finite subset $F_i\subset C_i\setminus\{1\}$.
  Then $C$ strongly engulfs in $G$ with respect to the sets $F_i$. 
\end{lem}

\begin{proof}
  We show that the inclusions  $\mu_i:C_i\ra G$ define a strong engulfing.
  By the malnormality property stated in Lemma \ref{lem_malnormal_tour}, conjugates of $C_{i_0}$ and $C_{i_1}$ in $G$ intersect trivially  for $i_0\ne i_1$,
  so the second condition of Definition \ref{dfn_strong_engulf} is satisfied.
  Assume by way of contradiction that, for some $i\leq p$, there exists  a morphism $\lambda_i:C_i\ra G$
  related to the inclusion $\mu_i$ (with respect to $\tilde \Gamma_{C_i}$), which kills an element of $F_i$.
  In particular, $\lambda_i$ is non-injective.
 
 A map $\lambda_i$ which is related  to the inclusion is  called a \emph{preretraction} (\cite{Perin_elementary}, see Definition \ref{defpreretr}). 
  The existence of a non-injective preretraction $C_i\ra G$ (with respect to $\tilde \Gamma_{C_i}$) implies that  either $C_i$  is a non-exceptional closed surface group or $\tilde \Gamma_{C_i}$ is    non-trivial   and 
  retractable
  (in $C_i$):
this  follows from Propositions 5.11 and 5.12 of \cite{Perin_elementary}, we will give a   rather short proof 
in  Subsection \ref{prtor}  (see Theorem \ref{thm_preretraction}). 
This contradicts the fact that $C_i$ is a prototype.
\end{proof}

There remains to prove step 3.
Assuming that $C$ strongly engulfs in $G'$ with respect to the sets $ F_{C_i,G'} $,
 consider morphisms $\mu_i:C_i\ra G'$ as in Definition \ref{dfn_strong_engulf}. 
The choice of $F_{C_i,G'}$ (see Lemma \ref{lem_choix}) ensures that each $\mu_i$ is injective.
By Lemma \ref{relat2}, $\mu_i$ is  related to a morphism $\lambda_i$ which is  either non-injective  or injective with values in $C'$. The definition of strong engulfing rules out the first possibility.
The \emph{moreover} part of Lemma \ref{relat2} shows that the injective morphisms $\lambda_i:C_i\ra C'$
define an engulfing of $C$ into $C'$: the element $\lambda_{i_1}(d_{C_{i_1}})$ is conjugate to $\mu_{i_1}(d_{C_{i_1}})$,
and  elements of $\lambda_{i_0}(D_{C_{i_1}\to C_{i_0}})$ are conjugate to elements of $\mu_{i_0}(D_{C_{i_1}\to C_{i_0}})$, 
so
   $\lambda_{i_1}(d_{C_{i_1}})$ cannot be conjugate 
  to an element of  $\lambda_{i_0}(D_{C_{i_1}\to C_{i_0}})$ (in $G'$ hence   in $C'$). 

This completes the proof of Theorem \ref{lethm2}

We end this section by proving   Corollary
\ref{thcl2} without assuming that $H$ is hyperbolic. In particular, we do not assume that $H$ is finitely generated.
 
\begin{proof}[Proof of Corollary \ref{thcl2}] We view $H$ and $G$ as $A$-groups, with
  $A=H$, and everything is relative to $A$. 
  As above,   let $C$ be a core of $G$  and  $C=C_1*\dots*C_p$ its  Grushko decomposition  
  with $A\subset C_1$. We show $C_1=A$. This implies that $H$ is hyperbolic, and the corollary follows from Theorem \ref{lethm} as explained at the beginning of Section \ref{interp}.
 
  We assume $C_1\ne A$ and work towards a contradiction. Then the following statement is true:
 \emph {any  morphism $\mu$ from $C_1$ to $H=A$ 
  is related to a morphism $\mu'$ whose kernel meets   $F_{C_1,H}$.}
Indeed, $\mu$ cannot be injective (it is the identity on $A$, and $C_1\ne A$)
and $H$ is a subgroup of the torsion-free hyperbolic group $G$, so the
statement holds by construction of $F_{C_1,H}$ (see Lemma \ref{lem_choix}).

As in Lemma \ref {lem_engulf_ee}, 
  the statement may be expressed by the fact that $H$ satisfies some  first-order formula with constants in   $A=H$. 
  Since $H$ elementarily embeds in $G$, this formula holds in $G$ and says that 
  any  morphism $\mu$ from $C_1$ to $G$ is related to a morphism $\lambda$ whose kernel meets $F_{C_1,G}$.
Applied to  the inclusion  $\mu :C_1\into  G$, this yields a non-injective preretraction  $\lambda:C_1\ra G$. As in the proof of Lemma \ref{lem_proto_engulf},
 Theorem \ref{thm_preretraction}   contradicts the fact that $C_1$ is a prototype (relative to $A$).
\end{proof}

\subsection{Preretractions} \label{prtor}

Preretractions 
were introduced by Perin in her work \cite{Perin_elementary} on elementarily embedded subgroups.   
 She proved the important fact that they lead to étage structures
 (Propositions 5.11 and 5.12 of \cite{Perin_elementary},  Theorem \ref{thm_preretraction} below). 
After recalling the definition, we will give a simpler proof of this fact.

As in the previous subsections, all groups in this subsection are either plain groups or $A$-groups for a given group $A\neq \{1\}$. In the relative case, 
recall that all morphisms are the identity on $A$.
Splittings, Grushko decompositions, JSJ decompositions, surface-type vertex groups, \'etages, towers  
are 
all relative to $A$. 
 The groups $G$ and $J$ considered below are assumed to be hyperbolic. 

When proving Theorem \ref{lethm}, the following definition   was used with $J=C_i$, one of the Grushko factors of the core of $G$,
and $\Gamma=\tilde \Gamma_{C_i}$, the modified JSJ decomposition of $J$.

\begin{dfn}[Preretractions, \cite{Perin_elementary} Definition 5.9] \label{defpreretr}
  Let $J$ be a plain group or an $A$-group, and let $G$ be a group containing $J$.
  Let $\Gamma$ be a splitting of  $J$.
  
   A morphism $r:J\to G$ is a \emph{preretraction} (with respect to $\Gamma$, with values in $G$) if it is related to the inclusion   (in the sense of Definition \ref{related}).
 
   In other words, non-exceptional surface-type vertex groups of  $\Gamma$ have non-abelian image, 
   and $r$ agrees with a conjugation on the other vertex groups (including exceptional surface-type vertex groups) and on edge groups. 
  \end{dfn}
  
  \begin{rem}\label{rem_bpm} 
 If $Q$ is a non-exceptional surface-type vertex group of $\Gamma$ (in the relative sense if $J$ is an $A$-group),  then $r_{|Q}$ is a \bpm{} because every boundary subgroup has a conjugate which is an edge group or contains $A$ with finite index. It is non-degenerate if it is not an isomorphism onto a conjugate of $Q$.
  \end{rem}

  \begin{thm}[Perin]\label{thm_preretraction}
  Let $G$ be an extended tower over a subgroup $C$,
  and let   $C=C_1*\dots *C_p*F$ be the   Grushko decomposition of $C$.
  
   If,  for some $i$,  there is a non-injective preretraction $r:C_i\ra G$ with respect to the modified JSJ decomposition $\tilde \Gamma_{C_i}$ of $C_i$,
   then $C_i$ is an extended étage (hence not a prototype). 
 \end{thm}

 The theorem is a special case of the following proposition, applied with $J=C_i$ and $\Gamma=\tilde \Gamma_{C_i}$.
 Note that $\tilde \Gamma_{C_i}$
satisfies the assumptions of the proposition 
  because it is a tree of cylinders.

\begin{prop} \label{propsdechloe}
  Let $J$ be a plain group or an $A$-group.
Let $\Gamma$ be a   (possibly trivial) splitting of $J$ such
 that all edge groups  are isomorphic to $\Z$, all edges join a cyclic vertex to a non-cyclic one, and the Bass-Serre tree of $\Gamma$ is 1-acylindrical near vertices with non-cyclic stabilizer.

 Assume that
 $J$  is a one-ended
 free factor of a group $C$, that $G$ is an extended tower over $C$, 
 and that there is a non-injective preretraction
 $r:J\to G$
 with respect to $\Gamma$.
 
Then 
 $J$ is an extended \'etage of surface type.
\end{prop}

\begin{proof}
  To show that $J$ is an extended étage, it suffices to find a non-exceptional surface-type vertex group $Q=\pi_1(\Sigma)$ of $\Gamma$ and a 
  \bpm\ $p:Q\to J$  with $p(Q)$   not abelian and not conjugate to $Q$   (see Remark \ref{rem_retractable},  and Remark \ref{rem_retractable_relatif} in the relative case).

  We may assume
  that $J$ is not a plain group isomorphic to the fundamental group of a non-exceptional closed surface group, since
  $J$ is an extended \'etage in this case (see Example \ref{csurf}).
  
We first consider the case when  $G=C$ (i.e.\ $J$ is a free factor of $G$). 
We argue as in the proof of 
Proposition 4.7   of \cite{GLS_finite_index} (where furthermore   $C=J$).

We may assume that $r$ maps every  non-exceptional  surface-type vertex group $Q$ of $\Gamma$ isomorphically to a conjugate: otherwise,
by Remark \ref{rem_bpm}, the restriction of $r$ to $Q$ is a \ndbpm. It takes values in $G$, but by Lemma 
\ref{cinqdouze} there is one 
with values in $J$, and we are done.

By definition  of a preretraction, the other vertex groups and the edge groups of $\Gamma$ are also mapped isomorphically to conjugates. 
Define  $\tau:J\to J$ by composing $r$ with a retraction of $G$ onto its free factor $J$. 
Corollary    \ref{sixunm2} 
implies that $\tau$ is injective, a contradiction since $r$ is not injective. 

We now consider the general case, so let 
 $G=G_0>G_1>\dots >G_k=C*F$ be a chain of subgroups,   with $F$ free and
 with retractions $\rho_i: G_i\to G_{i+1}$ associated to the surface-type étages $G_i> G_{i+1}$ as in
 Proposition \ref{prop_retractions} and Remark  \ref{prop_retractions_rel}. Without loss of generality, we may redefine $C=G_k$.
We denote by   $\Lambda_i$  the associated centered  splitting of $G_i$, with surface group $Q_i$.

Let $r_i:J\to G_{i}$ be $\rho_{i-1}\circ\cdots\circ \rho_0\circ r$ (with $r_0=r$).
By definition of a preretraction, $r$  is a morphism which agrees with a conjugation on edge groups of $\Gamma$ and on  
every vertex group 
 which is not a non-exceptional surface group. 
Since the retractions are the identity on $J$, 
these properties also hold for $r_i$.

If $r_k:J\to C$ is a (non-injective)  preretraction, we   get retractability 
 by the special case above since $J$ is a free factor of $C$.
If $r_k$ is not a   preretraction,  $\Gamma$ has a non-exceptional surface-type vertex group $\pi_1(\Sigma)$ 
whose image under $r_k$ is   abelian. 
Since the image of $\pi_1(\Sigma)$ under $r$ is not   abelian, there is a largest  index $i$ 
such that the image of $\pi_1(\Sigma)$ under
$r_i:J\to G_i$ is   non-abelian.

  We claim that the restriction $p_i$ of $r_i$ to $\pi_1(\Sigma)$ is a pinching \ndbpm.
Indeed,  $p_i$ is boundary-preserving because   
$r_{|\pi_1(\Sigma)}$ is and each $\rho_j$ is the identity on $J$.
Its image is   non-abelian  by our choice of $i$, and cannot be conjugate to $\pi_1(\Sigma)$
because $\rho_i$ is injective on $J$ and $\rho_i\circ p_i$ has   abelian 
image. Thus $p_i$ is non-degenerate. To prove that $p_i$ is pinching, we apply  Lemma \ref{lemcle2}   to $\Lambda_i$. 
 Since $G$ is CSA, the image of $p_i$ cannot contain a finite index subgroup of $Q_i$, because $\rho_i(Q_i)$ is non-abelian, or be contained in a bottom vertex group of $\Lambda_i$   because $\rho_i$ is injective on each bottom vertex group. 
Lemma \ref{lemcle2} then concludes that $p_i$ is pinching.

Applying   Lemma \ref{desc} $k-i$ times, we   get a pinching \ndbpm\ from $\pi_1(\Sigma)$ to $G_{k}=C$
(note that, being one-ended, $J$ is contained in a bottom group of each $\Lambda_i$, as required in the lemma). 
This shows that $\Gamma$ is retractable in $C$, hence in $J$ by  Lemma 
\ref{cinqdouze}.
\end{proof}

 \begin{rem} \label{tourestrel}
   The   one-endedness assumption  on $J$ 
   is used only to guarantee that  the tower is relative to $J$, i.e.
  that   $J$ is contained in a single bottom group of the splittings $\Lambda_i$.  This assumption is not needed in \cite{Perin_elementary}.
 \end{rem}

\section{Examples} \label{examp}

In this paper we have used two kinds of towers, namely simple towers and extended towers  (possibly relative). Both are important:  for instance, 
prototypes are defined using extended \'etages (see Section \ref{ordre}), and  
simple towers induce elementary embeddings (Theorem \ref{Tarski}). 

When all surfaces appearing are complicated enough, there is no difference between simple and extended towers (Proposition \ref{3simple}). On the other hand, surfaces $\Sigma$ with $\chi(\Sigma)=-2$ are a source of  pathologies; for instance:
\begin{itemize}
\item not being a prototype (i.e.\ being an extended \'etage) is not equivalent to  being a simple \'etage, and  none of these properties is 
equivalent to having a proper elementarily embedded subgroup;
\item
being elementarily free is not equivalent to being an $\omega$-residually free   hyperbolic tower, and is not equivalent to being a regular NTQ group;
\item
the elementary core of $G$ does not always embed elementarily into $G$.
\end{itemize}

We shall give counterexamples to these and other statements. Most examples are provided by $N_4$, the non-orientable closed surface of genus 4, or by  an extended \'etage constructed using a   twice-punctured Klein bottle, a 4-punctured sphere, or a thrice-punctured projective plane
(see Subsection \ref {constr}).
At the end we will   compare elementary equivalence with quasi-isometry and study \emph{parachutes}, groups that   have a simple centered splitting with base  $\Z$.

\subsection{Simple towers vs extended towers}\label{sec_simple_vs_ext}

We have seen   (Corollary \ref{ett}) that an extended \'etage may be upgraded to a simple \'etage by taking a free product with a free group.
One does not   even need to change the group 
if the associated surface is complicated enough (Proposition \ref{3simple}). 
But we shall now see that extended \'etages with low-complexity surfaces behave differently than   simple \'etages.

To put things into perspective, we first note:

\begin{prop}
\label{implications}
 Let $G$ be a non-abelian
  torsion-free hyperbolic group.
Each of the following assertions implies the next one:
 \begin{enumerate}
\item  $G$ is a simple \'etage over some non-abelian group   (i.e.\   $G=G'*\bbZ$ with $G'$ non-abelian, 
or $G$ has a retractable simple centered splitting with non-cyclic base);
\item $G$ contains a proper elementarily embedded subgroup (equivalently, by Theorem \ref{thcl}, some $G*F$ with $F$ free is a simple tower over some proper non-cyclic subgroup $H\inc G$);
\item $G$ is not a prototype,   i.e.\ $G$ is an extended étage 
(equivalently,  $G$ has a cyclic free factor or $G$ has a   retractable centered splitting).   
\end{enumerate}
 If $G\not\simeq \bbF_2$ has a cyclic free factor or a retractable centered splitting  whose underlying surface $\Sigma$ satisfies $\chi(\Sigma)\le-3$, 
then $G$ satisfies 1.
\end{prop}

 The last assertion is a partial converse, showing that $3\Rightarrow 1$ holds in many cases.

\begin{proof}
  $1\implies2\implies3$ follows from Theorems \ref{Tarski} and \ref{thcl}, the partial converse   from Proposition \ref{3simple}  (Proposition \ref{3simple2} below if $G$ is a parachute).  
 \end{proof}

We will show in   Subsections \ref{nst} and \ref{core} that $2\implies 1$ and $3\implies 2$ are not always true.
In fact:
\begin{itemize}
\item  non-prototypes, in particular elementarily free groups, are not always simple \'etages, and do not always contain a proper elementarily embedded subgroup;
\item a group containing a proper elementarily embedded subgroup is not always a simple \'etage;
\item 
$G$ being an extended tower over a non-abelian group $H$ does not imply that $H$   embeds elementarily in $G$; 
 the core of   $G$ does not always embed elementarily into $G$.
 \end{itemize}

\subsection{A construction} \label{constr}

\begin{figure}[ht!]
  \centering
  \includegraphics{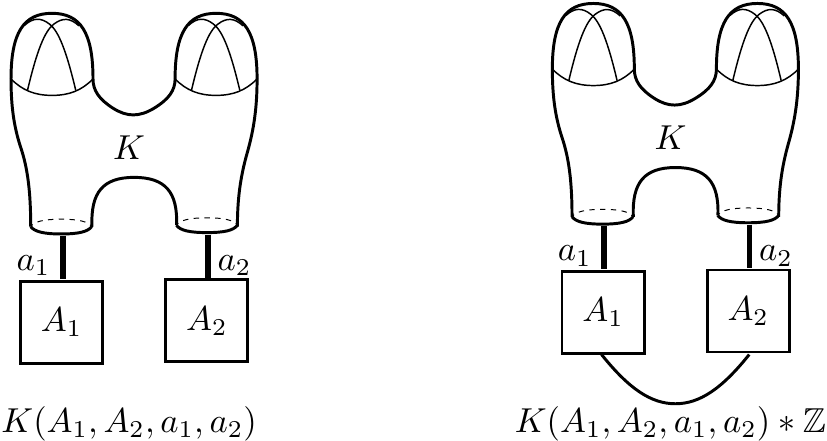} 
  \caption{The group $K(A_1,A_2,a_1,a_2)$.}
  \label{fig_Klein}
\end{figure}

Most of our examples are based on the following construction. 

\begin{example}[see Figure \ref{fig_Klein}] \label{ctrex}
Fix two torsion-free hyperbolic groups $A_1,A_2$, and two non-trivial elements $a_i\in A_i$  which are squares.
Let $K=\langle g_1,g_2,x,y\mid g_1g_2=x^2y^2\rangle$ be a free group of rank 3, viewed as the fundamental group of a twice-punctured Klein bottle   $N_{2,2}$.

Define $K(A_1,A_2,a_1,a_2)$ as the double cyclic amalgam  
$$K(A_1,A_2,a_1,a_2)=A_1*_{a_1={g_1}} K *_{ {g_2}=a_2} A_2.$$ 
 It is hyperbolic by \cite{BF_combination}. 

The centered splitting $\Gamma$ defined by the double  amalgam is retractable because $a_i$ is a square in $A_i$: the group $G=K(A_1,A_2,a_1,a_2)$ is a (non-simple) extended \'etage over $A_1*A_2$; in particular, $G\equiv A_1*A_2$ by Corollary  \ref{equivtour} (based on ``Tarski''). 
The group $G*\Z$ is a simple \'etage over a subgroup isomorphic to $A_1*A_2$ (see Proposition \ref{tstab}).
\end{example}

\begin{rem}
Viewing the twice-punctured Klein bottle as the union of a once-punctured Klein bottle and a pair of pants, we get a  simple centered splitting of $G$ whose surface is a once-punctured Klein bottle attached to a vertex carrying $A_1*A_2$. But this does not define a simple \'etage because,  unlike the twice-punctured Klein bottle, the once-punctured Klein bottle is an exceptional surface.
\end{rem}

\begin{rem}
 One may construct similar examples using a 4-punctured sphere $S_{0,4}$ instead of $N_{2,2}$,  with the central vertex of $\Gamma$  joined to each of the two bottom vertices by two edges, for instance   groups of the form $$\grp{A_1, A_2, g_1,g'_1,g_2,g'_2, t_1,t_2\mid g_1g'_1g_2g'_2=1,g_1=a_1, t_1g'_1t_1\m=a_1\m,g_2=a_2,t_2g'_2t_2\m=a_2\m}.$$
One may also construct examples based on a thrice-punctured projective plane.
\end{rem}

\begin{lem} \label{pastour}
Let $G=K(A_1,A_2,a_1,a_2)$ be as in Example \ref
{ctrex}  (with $a_i$ a square in $A_i$).
Suppose that, for $i=1,2$, the group $A_i$ is not   free   and has no cyclic splitting relative to $a_i$. Then $G$ is one-ended and   $\Gamma$   is its only retractable centered splitting; 
in particular, $G$ is not a simple étage over any subgroup. 

Moreover,   if $A_1$ and $A_2$ are one-ended,  then $G$ is not an extended étage relative to any subgroup $H$ isomorphic to $A_1*A_2$.
\end{lem}

  As in Definition \ref{relto}, saying that the \'etage is relative to $H$ means that $H$ is contained in a conjugate of a bottom group.

\begin{proof}
 The group $G$ is one-ended by Lemma \ref{facgru} (because $A_i$ is freely indecomposable relative to $a_i$). Let $\Gcan$ be its canonical cyclic JSJ decomposition. 

The   twice-punctured Klein bottle group $K$ is elliptic in $\Gcan$ (see  \cite [Proposition 5]
{GL_JSJ}),
and therefore so are $a_1$, $a_2$,  and also  $A_1$, $A_2$ by assumption. 
Thus $\Gamma$ dominates $\Gcan$.
The group $A_i$ cannot be conjugate  into a  QH vertex group of $\Gcan$ because we assumed $A_i$ not free,
so $A_i$ is conjugate into a rigid group and  $a_i$ is universally elliptic (\ie elliptic in all cyclic splittings of $G$).
This means that $\Gamma$ is universally elliptic, hence dominated by $\Gcan$, so $\Gamma$ and $\Gcan$ are in the same deformation space. 
 In particular they have the same non-cyclic vertex groups (up to conjugacy), so 
$K=\pi_1(N_{2,2})$ is the only surface-type vertex group of $\Gcan$.

Now suppose that $\Sigma$ is a retractable surface of some   centered splitting $\Lambda\ne\Gamma$. As explained in   
Lemma 
\ref{compajsj}, it may be identified with a subsurface  of  a surface  $\Sigma'$ appearing in $\Gcan$.  We have shown that $\Sigma'$ must be a twice-punctured Klein bottle.

  If $\Sigma=\Sigma'$, then $\Lambda=\Gamma$. Otherwise, $\Sigma$ is a proper non-exceptional subsurface 
  of a twice-punctured Klein bottle, 
  and the only possibilities are  a sphere with 4 punctures or a projective plane with 3 punctures. It follows that $\Gamma$ has at least 3 bottom vertices (carrying $A_1$, $A_2$, and $\Z$), and this prevents $\Lambda$ from being retractable by 
\cite{GLS_finite_index}   (see Remark \ref{ineg}).

  It immediately follows that $G$ has no simple étage structure. Since $\Gamma$ is its only structure of extended étage, the  ``moreover''  follows from the fact 
that no subgroup isomorphic to $A_1*A_2$ is conjugate into $A_1$ or $A_2$ because 
$A_1$ and $A_2$, being assumed to be one-ended, are co-Hopfian \cite{Sela_structure}.
\end{proof}

\begin{rem}[This remark will be used in Subsection \ref{minima}]\label{indice}
  We have seen that $\Gamma$ and $\Gcan$ are in the same deformation space. It follows that the Bass-Serre tree $\Tcan$ of $\Gcan$ is the tree of cylinders of the tree $T$ of $\Gamma$. It is not just a subdivision of $T$, because $a_i$ has roots in $A_i$.
  
 Indeed, for $i=1,2$ define $c_i\in A_i$ such that $a_i=c_i^{n_i}$ with $n_i>0$ and $n_i$ maximal ($c_i$ generates the centralizer of  $a_i$). The splitting $\Gcan$ has vertices $v_i$ of valence 2 carrying $\grp{c_i}$, joined to the vertex $v$ carrying $K$ by an edge carrying $\grp{a_i}$, and to the vertex carrying $A_i$ by an edge carrying $\grp{c_i}$. At the vertex $v_i$, the incident edge groups have indices $n_i$ and 1 in the vertex group,   so $n_1,n_2$ may be recovered from $\Gcan$. 
\end{rem}

The fact that $G$ is not a simple \'etage is not changed by ``adding constants''.

\begin{lem} \label{pastour2}
Let $G=K(A_1,A_2,a_1,a_2)$  
be as in Lemma \ref{pastour}.
Let $H$ be a torsion-free hyperbolic group, 
and $\hat G=G*H$.

Then $\hat G$ 
is not a simple étage over any subgroup  containing $H$.
\end{lem}

\begin{proof}
  Clearly, the étage cannot be of free product type because $G$ is one-ended, so the decomposition $\hat G=G*H$ is the only non-trivial decomposition
of $\hat G$ as a free product with $H$ contained in a factor.

Let therefore $\Gamma_{\hat G}$ be a   retractable centered splitting of $\hat G$ with a single bottom group $B$ containing $H$, and a  central group $Q$.
As in Subsection \ref{sec_grushko}  we consider the Grushko decomposition   $\Lambda_B$ of $B$ relative to the  incident edge groups,    but now also to $H$. 
Let
$\Lambda$ be  the splitting of  $\hat G$ obtained by blowing up   $\Gamma_{\hat G}$ using this decomposition of $B$, and let  $\Lambda_\bbZ$ be the 
union
of the edges of $\Lambda$ with non-trivial edge group.

By a relative version
of Proposition \ref{Grus}, $B_{\Lambda_\bbZ}=\pi_1(\Lambda_\bbZ)$ is a Grushko factor of $\hat G$   relative to $H$, hence is  conjugate to $G$ (we may assume $B_{\Lambda_\bbZ}=G$).
Since   the splitting $\hat G=G*H$ is the unique non-trivial minimal decomposition of $\hat G$ as a graph of groups with trivial edge groups,
$\Lambda_B$ has a single edge $e$ with trivial edge group, this edge is separating, and decomposes $B$ into a free product $B'*H$.

The splitting $B_{\Lambda_\bbZ}$ is therefore a centered splitting of $G$
with central vertex group $Q$ and a single bottom group $B'$.
  Since $\Gamma_{\hat G}$
is retractable, there is a non-degenerate boundary-preserving map from $Q$ to $\hat G$, hence   one from $Q$ to $G$ by 
 Proposition \ref{cinqdouze2}. This contradicts  Lemma \ref{pastour} 
 because $B_{\Lambda_\bbZ}$ is a retractable splitting of  
$G=K(A_1,A_2,a_1,a_2)$  with only one bottom group. 
\end{proof}

\begin{prop}\label{lebut}
Given arbitrary one-ended torsion-free hyperbolic groups $A_1,A_2$, there exists  a   one-ended hyperbolic group $G=K(A_1,A_2,a_1,a_2)$ satisfying the assumptions of  Lemma \ref{pastour}. 
In particular:
\begin{itemize}
\item $G$ is not a simple étage;
more generally,   by Lemma \ref{pastour2}, $G*H$ is not a simple étage over
  any subgroup containing $H$   (for $H$ an arbitrary torsion-free hyperbolic group);
\item   $G$ is an extended \'etage over $A_1*A_2$, so $G\equiv A_1*A_2$, but $G$  is not an extended étage relative to a subgroup isomorphic to $A_1*A_2$;
\item   if $A_1$ and $A_2$ are prototypes, then $G$ is minimal (it has no proper elementarily embedded subgroup).
\end{itemize}
\end{prop}

\begin{proof}
  Lemma \ref{nospl} below provides elements $b_i\in A_i$ such that $A_i$ has no cyclic splitting relative to $b_i$.
The element $a_i=b_i^2$ satisfies the requirements of  
Lemma \ref{pastour}, and the first two assertions follow.

 Now note that the only  non-trivial extended tower structure of $G$ is a single étage given by $\Gamma$. Indeed, $G$ is not an étage of free product type (it has no cyclic free factor), and by Lemma \ref{pastour} the only retractable centered splitting of $G$  is $\Gamma$, whose base $P \simeq A_1*A_2$ is a prototype, hence not an extended \'etage.
 If $H<G$ is a  proper elementarily embedded subgroup, then $G$ is an extended tower over $H$ relative to $H$ by Theorem \ref{thcl}.
This is impossible since the last étage  has to be simple   or of free product type by Remark \ref{rem_last_simple}.
\end{proof}

\begin{lem} \label{nospl}
If $A$ is a one-ended torsion-free hyperbolic group, there exists   a non-trivial element $a\in A$ such that $A$ has no cyclic splitting relative to $a$.
\end{lem}

\begin{proof}  If $G$ is a closed surface group, one can take any $a$ represented by a closed curve filling the surface. In general, one can prove the lemma by considering the canonical cyclic JSJ decomposition of $G$. We leave this as an exercise and we refer the reader to   Theorem 1.5 of \cite{GL_alea}, which proves  a stronger result  using random walks. 
\end{proof}

\subsection{$\omega$-residually free towers and NTQ groups}
\label{nst}
 In this subsection we will show that $\omega$-residually free towers fail to characterize elementarily free groups in two different ways. 
 
Let us first apply  Proposition \ref{lebut} with  $A_1,A_2$ one-ended torsion-free hyperbolic groups which are elementarily free (for instance, non-exceptional closed surface groups).
 The resulting  group $G=K(A_1,A_2,a_1,a_2)$ is elementarily equivalent to $A_1*A_2$, hence elementarily free. Thus:
 \begin{example} 
  There exists a one-ended hyperbolic group $G$ which is  elementarily free, 
   but is not a surface group and is not a simple tower over any proper subgroup. 
Moreover, $G*\bbF_r$ is not a simple tower over any proper subgroup containing $\bbF_r$.
\end{example}

 A  hyperbolic $\omega$-residually free tower  is defined in \cite{Sela_diophantine1} as a simple tower over a free product of free groups and non-exceptional surface groups. The group $G$ constructed above 
  is elementarily free, but is clearly not a   hyperbolic $\omega$-residually free hyperbolic tower.
   This shows that \cite[Theorem 7]{Sela_diophantine6} (in fact Proposition 6 therein) has to be corrected. 
It is not a regular NTQ group  (with constants) in the sense of \cite{KhMy_implicit,KhMy_elementary}, and 
the fact that $G*\bbF_r$ is not a simple tower over $\bbF_r$ means that $G*\bbF_r$ is not a
 constant-free regular NTQ group (see Section \ref{sec_defs}) so \cite[Theorem 41]{KhMy_elementary} has to be corrected.

The above example can be fixed if one considers {\em hyperbolic towers} in the sense of Perin \cite{Perin_elementary}, i.e.\ if one allows    centered splittings 
with several bottom vertices. 
Indeed, the group $G$ is a hyperbolic tower over $A_1*A_2$ in   Perin's sense, and thus elementarily free. 

  But there is yet another way that $\omega$-residually free towers fail to characterize elementarily free groups. The following 
example is elementarily free, but not even a hyperbolic tower in the sense of \cite{Perin_elementary}. It shows that, to characterize correctly elementarily free groups, one needs to   use extended towers as   defined in \cite{Perin_elementary_erratum}.

 \begin{figure}[ht!]
   \centering
   \includegraphics{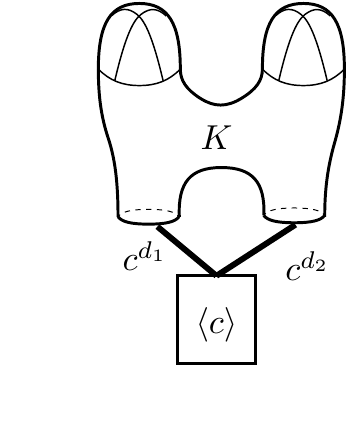}
   \caption{A group that is elementarily equivalent to $\bbF_2$ but is not an extended tower over $\bbF_2$  ($d_1,d_2$ even, with $ | d_1 | \ne  | d_2 |$).  By construction, it is an extended étage over $\bbZ$.}
   \label{fig_not_ext_tower}
 \end{figure}

\begin{example}[See Figure \ref{fig_not_ext_tower}]\label{example_not_extended_tower}
 For $d_1,d_2\ge2$ consider the group $$G=\grp{a_1,a_2,b_1,b_2,t,c\mid a_1^2a_2^2b_1b_2=1, b_1=c^{d_1},tb_2t\m=c^{d_2}}$$
  obtained from a cyclic group $\grp {c}$ by gluing the two boundary components of a twice-punctured Klein bottle   onto powers of $c$ (this group is a parachute, see Section \ref{parac}).

The defining splitting is retractable if $d_1$ and $d_2$ are even. 
In this case, $G$ is an extended étage over $\bbZ$, it is elementarily free, 
 but 
 it is not an extended tower over  
 $\F_2$ when $ | d_1 | \ne  | d_2 |$  (see Remark \ref{exte} below).\\
\end{example}

We have seen (Example \ref{csurf})  that
the fundamental group  $\pi_1(S_g)$ of the orientable surface of genus $g$ for $g\geq 2$, and the fundamental group  $\pi_1(N_g)$ of the non-orientable surface of genus $g$ for $g\geq 5$, are  simple \'etages over $\bbF_2$.   This implies  that any hyperbolic $\omega$-residually free hyperbolic tower is in fact a 
simple tower over $\bbF_2$ or over $\pi_1(N_4)$ (to see this, simply note that $\pi_1(N_4)*H$ is a simple tower over $\F_2$ if $H$ is $\Z$ or a non-exceptional surface group).

Unlike  the other non-exceptional surface groups, 
$\pi_1(N_4)$ is not a simple étage (see Example \ref{csurf}), although it   is an extended étage over $\bbF_2$.
The following lemma shows that $\pi_1(N_4)$ has no proper elementarily embedded subgroup,
so  $3\not\implies 2$ in Proposition \ref{implications}.

\begin{lem}[  {\cite[Cor.~3.13]{LPS_hyperbolic}}]\label{lem_N4}
  The non-orientable surface group $\pi_1(N_4)$  has no proper elementarily embedded subgroup.
\end{lem}

\begin{proof}
 Up to automorphism, the group $\pi_1(N_4)$ has exactly two retractable splittings, shown on Figure \ref{fig_N4}. But all the bottom groups of these decompositions are cyclic, so $\pi_1(N_4)$ has no extended tower structure  relative to $H$ with $H$ non-abelian, and  $\pi_1(N_4)$ has
no proper elementarily embedded subgroup by Theorem \ref{thcl}.
\end{proof}

Theorem \ref{contrex} below
gives other examples showing that $3\not\implies 2$   in Proposition \ref{implications}.

We may also apply Proposition \ref{lebut} with $A_1$ elementarily free and $A_2$ a prototype. In this case $G$ is not a simple \'etage, but   $A_2$ (which is the core of $G$) is an elementarily embedded subgroup by Theorem \ref{thcl} because $G$ is an extended tower over $A_2$ relative to $A_2$. This shows that $2\not\implies 1$ in Proposition \ref{implications}.

\subsection{The core} \label{core}

Let $G$ be a non-abelian torsion-free hyperbolic group.

\subsubsection{Groups with no elementarily embedded core}
Recall that the core $\core(G)$ of $G$ is the unique prototype such that $G\tge \core(G)$,
  and that we defined $\ecore(G)$ as $\core(G)$, unless   $\core(G)=1$ in which case we set $\ecore(G)=\bbF_2$. 
We have $\ecore(G)\equiv G$   (Corollary \ref{equivcore}), and by Corollary \ref{coretour} there exists $r\geq 0$ such that $\ecore(G)$ embeds elementarily in $G*\bbF_r$.

If  $\core(G)$ is one-ended,   it has an elementary embedding into $G$ by Corollary \ref{cor_core_unbout},
though as pointed out in the previous subsection     $G$ is not always a simple \'etage. 

On the other hand:

\begin{thm}\label{contrex}
  If $P$ is a prototype with more than one end, there is a torsion-free hyperbolic group $G$ such that $G$ is an extended \'etage over $P$
  (so that $P\simeq\core(G)$ is the core of $G$ and  $G\equiv P$), but $P$ does not elementarily embed into $G$.

  Moreover, $G$   has no proper elementarily embedded subgroup (but is not a prototype).
\end{thm}

This shows that Theorem 7.6 of \cite{Sela_diophantine7} has to be corrected.

\begin{proof}
 We write  $P$ as a free product of  one-ended prototypes $P=A_1*\dots*A_p$ with $p\ge2$. 
 If $p=2$, we have seen that the group $G=K(A_1,A_2,a_1,a_2)$ constructed in Subsection \ref{constr} has the required properties.
 
If $p>2$, we just take $G=K(A_1,A_2,a_1,a_2)*A_3*\dots*A_p$.   It is an extended \'etage over $P$ by Remark \ref{rk_prodlib}.
  Denote by $\Gamma_K$ the unique retractable  centered splitting of $K(A_1,A_2,a_1,a_2)$.
 If $\Gamma_G$ is any retractable splitting of $G$, we claim that it is obtained from $\Gamma_K$ by replacing the two bottom vertex groups
$A_1,A_2$ by their free product with some conjugates of $A_3,\dots,A_p$,  in such a way  that the base of $\Gamma_G$ is isomorphic to $P$.

Indeed, consider the  Grushko blowup $\Lambda$ of $\Gamma_G$ (Definition \ref{leslambda}), so that collapsing edges with trivial stabilizer in $\Lambda$ yields
back $\Gamma_G$. 
Applying Proposition \ref{Grus} shows that $G_1=\pi_1(\Lambda_\Z)$ is conjugate to $K(A_1,A_2,a_1,a_2)$,
and by uniqueness that $\Lambda_\Z$ is isomorphic to $\Gamma_K$.

This proves the claim and shows that all extended tower structures on $G$ have a single étage with two bottom groups and   base
isomorphic to the prototype $P$.
One concludes as 
 in the proof of Proposition \ref{lebut} 
that $G$ has no proper elementarily embedded subgroup.
 \end{proof}

 If $G$ is elementarily free, we have defined $c(G)=\{1\}$ and $\ecore(G)=\bbF_2$. Lemma \ref{lem_N4} shows that, in this case also,  $\ecore(G) $ does not always elementarily embed into $G$. In   Example \ref{example_not_extended_tower} above,
 $G$ is not an extended tower over $\ecore(G)=\bbF_2$, so $G\not\tge \ecore(G)$
 (though  by definition     $G\tge\core(G)$ always holds).

\subsubsection{Non-uniqueness of the core as a subgroup}

We have seen (Corollary \ref{fini2}) that, up to isomorphism, only finitely many groups $H$ embed elementarily into a given
  torsion-free hyperbolic group $G$.
But, given $H$ and $G$,  the number  of elementarily embedded subgroups  of   $G$ isomorphic to  $H$  
  may be infinite   up to conjugacy, and even up to automorphisms of $G$. Here are exemples with $H=\ecore(G)$.

\begin{example}[Infinitely many  elementary embeddings of the core] \label{imee}
 Let $$F=\grp{e_1,e_2,e_3,e_4}\simeq \bbF_4$$ be the fundamental group of the subsurface of genus 2 of the orientable surface $S_4$ 
considered  in Example \ref{pasel0}.
It is an elementarily embedded subgroup of   $\pi_1(S_4)$, and so is any non-abelian free factor of   $F$. 
In particular,  $\pi_1(S_4)$ has infinitely many non-conjugate elementarily embedded subgroups isomorphic to its elementary core $\F_2$.
  One can check that this family of subgroups is also infinite up to automorphisms of $\pi_1(S_4)$.
For instance, consider the image of $\grp{e_1,e_2}$ by automorphisms of $F$ sending the commutator $[e_1,e_2]$
to elements represented by curves with many self-intersections.

If $A$ and $B$ are prototypes, the group $G=A*B*\bbF_n$ has infinitely many conjugacy classes of free factors of the form $A^g*B^h$ with $g,h\in G$,
and these are elementarily embedded subgroups isomorphic to $\ecore(G)$.
  This collection is finite up to automorphisms of $G$ 
but, by gluing a retractable surface to $G$ in a complicated way, one can construct a larger group $G'$ 
such that this collection still consists of elementary embedded cores,
and is infinite up to automorphisms of $G'$.
 \end{example}

If $\core(G)$ is one-ended, it elementarily embeds into $G$ (Corollary \ref{cor_core_unbout}), 
and $G$ only contains finitely many conjugacy classes of subgroups isomorphic to $\core(G)$ by
a theorem of Gromov \cite[Th. 5.3.C']{Gromov_hyperbolic}, but there is no uniqueness in general.

\begin{example}[Several  elementary embeddings of a one-ended core]\label{example_nonunique}
Consider Example \ref{bete}  with a group $B_1$ having no cyclic splitting. Then $A=\grp{B_1,F(x,y)}$ and  $A_B=\grp{B_1,F(x_B,y_B)}$ are   non-conjugate  one-ended elementarily embedded subgroups isomorphic to $\core(G)$.  As above, one may modify this example by gluing a surface to $G$ so that no automorphism sends $A$ to $A_B$. 
\end{example}

This   example relies on the existence of  a surface in the JSJ decomposition of $\ecore(G)$. 
If there is no surface, the embedding of a one-ended core into $G$ is indeed unique  up to conjugacy. 

\begin{prop}
Let $G$ be a non-abelian torsion-free hyperbolic group. Suppose that its core  $\core(G)$ is one-ended,   not isomorphic to $\pi_1(N_3)$, and the cyclic JSJ decomposition of $\core(G)$ has no non-exceptional surface-type vertex. Then all subgroups  $H$ of $G$ isomorphic to $\core(G)$ are conjugate.
\end{prop}

\begin{rem}
  The subgroups $H$ in the proposition need not   be elementarily embedded, but it is crucial that they be isomorphic to $\core(G)$.
\end{rem}

\begin{proof}
    By definition, $G$ is an extended tower over  $
    \core(G)$, so by Corollary \ref{ett} there exists a free group $F$ such that
  $G*F$ is a simple tower over  a subgroup isomorphic to $\core(G)$. 
  Consider subgroups $G*F=A_0\supset\cdots \supset A_p\simeq\core(G)$ such that, for each $i$, either $A_i=A_{i+1}*\Z$ or $A_i$ is a simple \'etage of surface type over $A_{i+1}$. Being one-ended, $A_p$ is conjugate to a subgroup of  $G$, so we may assume without loss of generality that $A_p<G$.

  Let $H<G<A_0$ be a subgroup isomorphic to   $A_p$. 
  We will show by induction   that, for all $i$,  the subgroup $H$ is conjugate   (in $G*F$) to a subgroup of $A_i$. 
  Co-hopfianity   of $A_p$ then implies that $H$ is conjugate to $A_p$, in $G*F$ hence in $G$ by malnormality of $G$, concluding the proof.

 We therefore assume that $H$ is conjugate into $A_i$. If $A_i=A_{i+1}*\Z$, then $H$ is conjugate into $A_{i+1}$ by   one-endedness. 
If $A_i$ is an \'etage of surface type and $H$ is not conjugate into the bottom group $A_{i+1}$, Lemma \ref{passurf} and Remark  \ref{passurf2} imply that  $H$ has a cyclic splitting with a surface-type vertex. The associated surface   $\Sigma$ covers a non-exceptional surface by Remark  \ref{passurf2}, so is non-exceptional.     Lemma \ref{compajsj} implies 
 that the JSJ decomposition of $H$   has a   surface   $S$ containing $\Sigma$ as an incompressible subsurface. This surface $S$ must be $N_3$ or non-exceptional,  contradicting our assumptions on $\core(G)$.
\end{proof}

\begin{rem} The proof shows that every one-ended finitely generated subgroup of $G$ whose cyclic JSJ decomposition has no non-exceptional surface-type vertex is   isomorphic to $\pi_1(N_3)$ or conjugate into   $\core(G)$.
\end{rem}

 \subsection{Elementary equivalence and quasi-isometry}
 It is well-known that elementarily equivalent groups may be very diverse. We illustrate this by comparing elementary equivalence with quasi-isometry  (see   \cite{GhHa} 
 for basics about quasi-isometry). 
 
\begin{prop}
 If $G$ is any non-abelian  torsion-free  hyperbolic group, there are infinitely many groups $G_n$ which are elementarily equivalent to $G$ but are all distinct up to quasi-isometry.
\end{prop}
 
\begin{proof}
 For $n\ge3$, let $H_n$ be the fundamental group of the space obtained by gluing $n$ copies $T_i$ of a once-punctured torus along their boundary $C_i$. It is a one-ended torsion-free hyperbolic group, whose canonical cyclic JSJ splitting has $n$ vertices carrying the groups $\pi_1(T_i)$ and one vertex carrying $\pi_1(C_i)$, a cyclic group independent of $i$ (see Section 3.9 of \cite{GL_JSJ} for the case $n=3$; the general case is similar). 
 
 The Bass-Serre tree of this JSJ splitting has vertices of infinite valence and vertices of valence $n$, so depends on $n$ (as a simplicial tree, with no group action on it). It follows from \cite{Bo_cut} that   the groups $H_n$ have non-homeomorphic Gromov boundaries, hence  are quasi-isometrically distinct.
 On the other hand, $H_n$ is a simple tower over 
 the fundamental group of the closed orientable surface of genus 2, so is elementarily equivalent to  $\F_2$.
 
 In general, the groups $G_n=G*H_n$ are elementarily equivalent to $G$, but are quasi-isometrically distinct by \cite{PaWh_QI}. 
\end{proof}

In the other direction, we have proved \cite{GLS_finite_index} that, for many torsion-free hyperbolic groups $G$, there exist finite index subgroups $G_n$ of $G$ which are all distinct up to elementary equivalence (recall that finite index subgroups of $G$ are quasi-isometric to $G$).  This applies in particular if $G$ is cubulable or is a hyperbolic limit group, but is not a free product of cyclic groups and surface groups. In full generality, we  do not know whether, given a non-abelian hyperbolic group $G$, there always exist infinitely many torsion-free groups which are quasi-isometric to $G$ (hence hyperbolic)  but not elementarily equivalent.

\subsection{Parachutes}\label{parac}

\begin{figure}[ht!]
  \centering
  \includegraphics{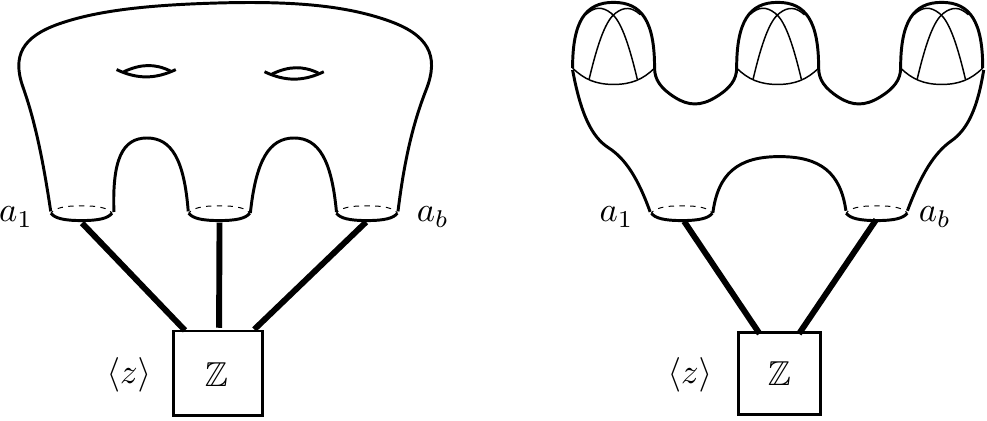}
  \caption{Two parachutes}
  \label{fig_def_parachute}
\end{figure}

  We now consider parachutes, i.e.\ groups $G$ having a simple  centered   splitting $\Gamma$ with base $\Z$ (Figure \ref{fig_def_parachute}).
  We allow exceptional surfaces and non-retractable splittings.
These groups were studied in \cite{GLS_finite_index} (where they were called \emph{socket groups with identified sockets}). 
They are torsion-free hyperbolic groups, and are one-ended by  Lemma \ref{facgru}. 
 We shall show (Corollary \ref{synth}) that they are   either  prototypes or elementarily free groups.

If the surface $\Sigma$ associated to $\Gamma$ has genus $g$ and   $b$ boundary components, 
  $G$ is  presented as 
  $$\left\langle
    \begin{matrix}a_1,\dots,a_b\\
      z,t_1,\dots,t_b\\
      u_1,v_1,\dots,u_g,v_g
    \end{matrix}\ \middle|\ 
    \begin{matrix}
      a_1\cdots a_b=[u_1,v_1]\cdots[u_g,v_g]\\
      t_iz^{d_i}t_i\m=a_i\\
      t_1=1
  \end{matrix}
\right \rangle$$
or
$$\left\langle
\begin{matrix}
  a_1,\dots,a_b\\
  z,t_1,\dots,t_b\\
  u_1, \dots,u_g
\end{matrix}
\ \middle| \ 
\begin{matrix}
  a_1\cdots a_b=u_1^2 \cdots u_g^2\\
  t_iz^{d_i}t_i\m=a_i\\
  t_1=1
\end{matrix}
\right\rangle,$$
depending on whether $\Sigma$ is orientable or not, with indices $d_i \ne0$   (possibly negative).  

Minimality of   $\Gamma$ implies $\sum | d_i | \ge2$. The case when $\sum | d_i |=2$ is special, since $G$ is  then  a closed  surface group
(the surface is the union of $\Sigma$ and an annulus or a M\"obius band).
We have seen (Example \ref{csurf}) that such groups are  elementarily free,  except  that  $\pi_1(N_3)$ is a prototype.

\begin{lem}\label{bonjsj}
If $\sum | d_i | \ge3$, then $\Gamma$ is the   cyclic  JSJ decomposition of $G$.  In particular, if $\sum |d_i|\geq 3$ and $\Sigma$ is exceptional, then $G$ is a prototype.
\end{lem}

\begin{proof}
  The proof is based on Bowditch's  construction of the cyclic JSJ tree $\Tcan$ from the topology of 
the Gromov boundary
  $\partial G$  of $G$ \cite{Bo_cut}.
  
Let $  T$ be the Bass-Serre tree of $\Gamma$. Let $E$ be the open star of the vertex with stabilizer $\grp{z}$.
It consists of exactly $N=\sum |d_i|$ edges, all fixed by a power of $z$.

Let $\Lambda_z\subset\partial G$ be the pair of fixed points of $\grp{z}$.
By \cite{Bo_cut},  $\partial G\setminus \Lambda_z$ 
can be partitioned into $N$ disjoint open sets corresponding to the  connected components of $T\setminus E$.
It follows that $\partial G\setminus \Lambda_z$ has at least $N$ components.

Since $N\geq 3$, \cite[Proposition 5.13]{Bo_cut} and \cite[Lemma 4.3]{Gui_reading} show that $\grp{z}$ is universally elliptic (elliptic in every cyclic splitting of $G$). In particular, the edge stabilizers of $T$ are universally elliptic, so $T$ is  dominated by the cyclic JSJ tree $\Tcan$. On the other hand, 
  QH vertex stabilizers are elliptic in   $\Tcan$ by \cite[Prop.~5]{GL_JSJ}, so  $T$ dominates 
$\Tcan$. 
It follows that $T$ is a JSJ tree.
Since it is its own tree of cylinders, it coincides with $\Tcan$ (see \cite{GL_JSJ}).

If $\sum |d_i|\geq 3$ and $\Sigma$ is exceptional, then 
 $G$ has no 
 retractable centered splitting by Lemma \ref{preconf}, so is a prototype
(recall that, by Definition \ref{dfn_retractable}, a retractable centered splitting is non-exceptional).
\end{proof}

\begin{rem}\label{rem_multiparachute}
  The argument applies in the following more general setting:  $G$ has a centered splitting $\Gamma$ with all bottom groups cyclic,
  and   for each bottom group the sum of the indices of the incident edge groups is at least 3. 
\end{rem}

We proved in \cite{GLS_finite_index} that, depending on the values of $g,b$, and the $d_i$'s,  either $\Gamma$ is retractable and $G$ is elementarily free, or $\Gamma$ is not retractable.   
When retractable, the splitting $\Gamma$ expresses $G$ as an extended \'etage over $\Z$, but not as a simple \'etage (see Remark \ref{pasz}). As in Proposition \ref{3simple}, we get a simple \'etage if $\Sigma$ is complicated enough.

\begin{prop} \label{3simple2}
 Assume that $\Gamma$ is  retractable, and 
 the   surface $\Sigma$  either has Euler characteristic satisfying   $ | \chi(\Sigma) | \ge  3$ or   is  
a  twice-punctured torus. Then $G$ is a simple \'etage over $\bbF_2$. 

\end{prop}

The conclusion of the proposition implies that $G$ has a proper elementarily embedded subgroup (isomorphic to $\bbF_2$).
It  applies whenever the   
surface $\Sigma$ is 
 not a 4-punctured sphere  $S_{0,4}$, a   thrice-punctured projective plane $N_{1,3}$, a   twice-punctured Klein bottle $N_{2,2}$, or $N_{3,1}$   (note that $\Sigma$ cannot be a   once-punctured torus  or an exceptional surface since $\Gamma$ is retractable). 
In these special cases we reach the opposite conclusion:

\begin{prop}\label{prop_parachute}
Assume that 
$\Sigma$ has Euler characteristic satisfying
$ | \chi(\Sigma) | \leq  2$, and   $\Sigma$ is not a twice-punctured torus.
Then    
$G$ is not a simple \'etage  (over a non-abelian group)  and
$G$ has no proper elementarily embedded subgroup.
\end{prop}

 Combining Lemma \ref{preconf},   Corollary \ref{cor_tarski}, and the propositions, we get:

\begin{cor}\label{synth}
 Let $G$ be a parachute.
 \begin{itemize}
 \item If $\Gamma$ is not retractable
   (in particular if $\Sigma$ is exceptional), then $G$ is a prototype. 
\item If $\Gamma$ is retractable, $G$ is elementarily free; it 
 has a proper elementarily embedded subgroup if and only if   $\Sigma$ satisfies 
  $ | \chi(\Sigma) | \ge  3$ 
  or is a twice-punctured torus. \qed
 \end{itemize}
\end{cor}

 We shall now prove 
the propositions, starting  with the easier one.
They are true if  $\sum |d_i|=2$ (i.e.\ $G$ is a surface group), as explained in Subsection \ref{nst}. 
 We therefore   assume $\sum  | d_i | \ge3$.  Lemma \ref{bonjsj} then  implies that $\Gamma$ is the   cyclic JSJ decomposition of $G$.

\begin{proof}[Proof of Proposition \ref{prop_parachute}]
  If $\Sigma$ is exceptional, then  $G$ is a prototype by Lemma \ref{bonjsj} and the proposition holds.
 We are left with the surfaces $N_{3,1}$, $N_{2,2}$, $N_{1,3}$, $S_{0,4}$.
By Theorem \ref{thcl}, 
  it suffices to show that $G$ is not an extended \'etage 
over a non-abelian group relative to a non-abelian group. Suppose it is. 
The associated surface $\Sigma'$ may be viewed as a proper   incompressible subsurface of $\Sigma$ (see   Lemma \ref{compajsj}); moreover, the complement of $\Sigma'$ in $\Sigma$ cannot consist only of annuli and M\"obius bands because the \'etage is relative to a non-abelian group, so $\chi(\Sigma')=-1$.
Since $\Sigma'$ is non-exceptional  and $\Sigma$ is not   a twice-punctured torus,  
the only possibility is that $\Sigma'$ is a once-punctured torus, $\Sigma$ is $N_{3,1}$,
and $\Sigma\setminus \Sigma'$ is a twice-punctured projective plane.

We then write
$$   G=\grp{u,v,u',z\mid  [u,v]u'^2=z^d}=\grp{u,v,u',z\mid [u,v]=z^{d}u'^{-2} }$$
where $\pi_1(\Sigma')=\grp{u,v}$ and the bottom group of the splitting defining the extended \'etage 
is  $\grp{z,u'}\simeq \bbF_2$.
But this splitting cannot be retractable since   $z^{d}u'^{-2}$ is not a commutator in $\grp{z,u'}$.
This concludes the proof.
\end{proof}

\begin{rem} \label{exte}
 Some of the groups in Proposition \ref{prop_parachute} are (non-relative) extended \'etages over $\F_2$, but most of them  are not. For instance, as in Example \ref{example_not_extended_tower}, consider parachutes with $\Sigma=N_{2,2}$. If the indices $d_1$ and $d_2$ are even, $G$ is elementarily free. We claim that it is not an extended tower over $\F_2$ if $ | d_1 | \ne | d_2 | $. 

The argument is the same as in the previous proof, but we now have to consider the case when $\chi(\Sigma')=-2$. It may be a three-punctured projective plane or a 4-punctured sphere, obtained from $\Sigma$ by removing  one M\"obius band,   two M\"obius bands, or a non-separating  annulus.  Using Lemma \ref{lemcle4}, one checks that the corresponding splittings cannot be retractable.
  \end{rem}

\begin{proof}[Proof of Proposition \ref{3simple2}]
Given a \ndbpm\  $p:\pi_1(\Sigma)\to  G$, let $\calc$ be a maximal family of  pinched curves on $\Sigma$.
It  is non-empty by Lemma \ref{lemcle4}.

  \emph{Case 0: there exist two boundary components $C_1$, $C_2$ of $\Sigma$
  that are separated by $\calc$.} In this case, we may argue as in the proof of Proposition \ref{3simple},
using   an arc $\gamma$ from $C_1$ to $C_2$ meeting a curve $\delta\in\calc$ exactly once.
The surface $\hat \Sigma$ obtained by removing  a pair of pants obtained as
a regular neighborhood of $C_1\cup\gamma\cup C_2$  is non-exceptional because
$\chi(\hat\Sigma)=\chi(\Sigma)+1\le-2$ (when $\Sigma$ is a   twice-punctured torus,   $\hat \Sigma$
is a   once-punctured torus, hence non-exceptional).
  As in the proof of Proposition \ref{3simple},  this defines a retractable simple centered splitting  with base  $\bbF_2$ and surface $\hat \Sigma$,
so the proposition holds in this case.

\begin{figure}[ht!]
  \centering
\includegraphics{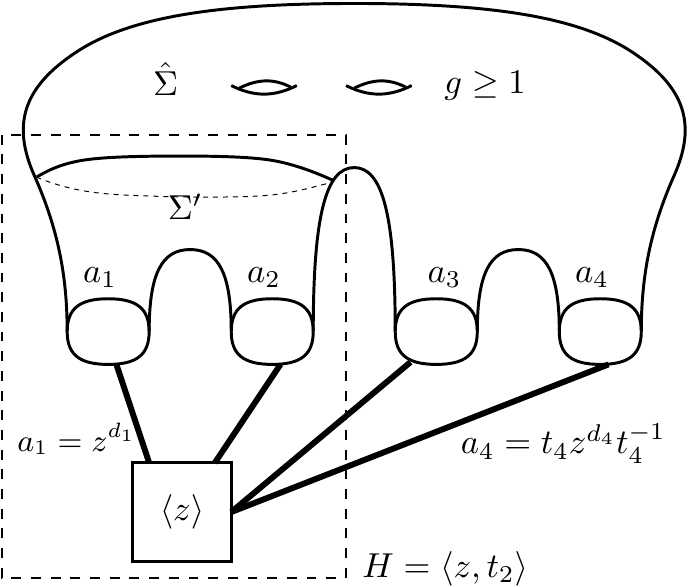}
  \caption{Case 1:  the new étage when $\Sigma$ is orientable}
    \label{fig_orient}
\end{figure}

\emph{The remaining cases: $\calc$ does not separate boundary components.}
We shall distinguish  several cases.
In each case we define a new étage structure
by
writing the surface $\Sigma$ as the union of two subsurfaces $\Sigma=\Sigma'\cup \hat \Sigma$ intersecting in a simple closed curve.
The surface $\Sigma'$ satisfies $\chi(\Sigma')=-1$ and contains at least one component of $\partial \Sigma$, and $\hat \Sigma$ is associated to a simple centered splitting with base $H\simeq \F_2$.

The retractability of $\Gamma$ implies $d_1+\dots +d_b=0$ (and in particular $b\geq 2$) when $\Sigma$ is orientable, and $d_1+\dots +d_b$  even when $\Sigma$ is non-orientable (\cite{GLS_finite_index}, Propositions 6.7 and 6.8).
Using this, we 
construct  a retraction $\rho:G\to H$ such that the central group $\pi_1(\hat\Sigma)$ has non-abelian image.   

  We use the presentation of $G$ given   at the beginning of this subsection, with 
  $g$  the genus of $\Sigma$  and $b$  the number of boundary components. 
  It suffices to define $\rho$ on the generators $z,t_i,u_j,v_j$ since $a_i= t_iz^{d_i}t_i\m$.

\emph{Case 1: $\Sigma$ is orientable} (see Figure \ref{fig_orient}).
If   $g=0$, then any curve in $\calc$ separates two boundary components of $\Sigma$ and we are done.
Assume therefore that   $g\ge 1$.
Recall that $b\ge2$. Decomposing $\Sigma$ into a pair of pants  $\Sigma'$ with fundamental group $\grp{a_1,a_2}$ and its complement $\hat\Sigma$,
one gets a simple centered splitting with bottom group $H=\grp{z,t_2}$ 
 and central vertex group $Q=\grp{a_3,\dots,a_b,u_1,v_1,\dots u_g,v_g}$.  As in case 0, the surface $\hat \Sigma$ is non-exceptional.

One then defines a retraction $\rho:G\to H$ by 
\begin{align*}
  z,t_2&\mapsto z,t_2 \\
 t_3,\dots,t_b,   u_2,\dots,u_g,v_2,\dots,v_g&\mapsto 1 
\end{align*}
and there remains to define the images $\tilde u_1$, $\tilde v_1\in H$ of $u_1,v_1$ satisfying the equation
$$z^{d_1}t_2z^{d_2}t_2\m z^{d_3+\dots +d_b}=[\tilde u_1,\tilde v_1].$$
Since $d_1+\dots+d_b=0$, we may take  $\tilde u_1=z^{d_1}t_2 $ and $\tilde v_1=z^{d_2}$. 
Finally, $\rho(Q)$ contains $\grp{\tilde u_1, \tilde v_1}
$,
which is non-abelian since $d_2\neq 0$.

\begin{figure}[ht!]
  \centering
\includegraphics{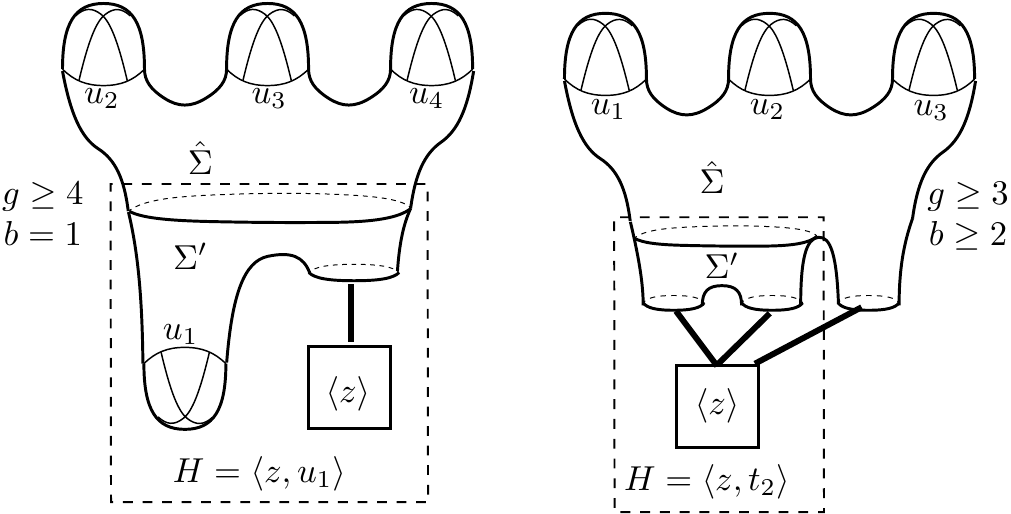} 
  \caption{Cases 2 and 3: $\Sigma$ is non-orientable of genus $g\geq 3$}
    \label{fig_nonorient}
\end{figure}

\emph{Case 2: $\Sigma$ is non-orientable with a single boundary component} (see Figure \ref{fig_nonorient}).
Note that   $|\chi(\Sigma)|\geq 3$ implies $g\ge4$,
and recall that  $d_1$ is even. 
In this case,  we take for $\Sigma'$ a twice-punctured projective plane   containing $\partial\Sigma$
(see Figure \ref{fig_nonorient}).
Its complement $\hat \Sigma$ in $\Sigma$  is homeomorphic to $N_{g-1,1}$, hence non-exceptional because $g\geq 4$.
The group  $G$ has a simple centered splitting with bottom group $H=\grp{z,u_1}$ and  
central vertex group $Q=\pi_1(\hat\Sigma)=\grp{u_2,\dots,u_g}$.

One   defines   $\rho:G\to H$ by 
\begin{align*}
  z,u_1&\mapsto z,u_1 \\
  u_4,\dots,u_g&\mapsto 1, 
\end{align*}
and there remains to define the images $\tilde u_2$, $\tilde u_3$ of $u_2,u_3 $ satisfying the equation
$z^{d_1}=u_1^2\tilde u_2^2\tilde u_3^2$.
  Since $d_1$ is even and non-zero, one may take $\tilde u_2=u_1\m$ and $\tilde u_3=z^{
 d_1/2}$,
and $\rho(Q)=\grp{\tilde u_2,\tilde u_3}$ is non-abelian.

\emph{Case 3: $\Sigma$ is non-orientable of genus $g\geq 3$ with $b\geq 2$ boundary components}
(see Figure \ref{fig_nonorient}).
As in case 1,   we take for $\Sigma'$ a pair of pants,
and we get a simple centered splitting of $G$ with bottom group $H=\grp{z,t_2}$
and central vertex group $Q=\grp{a_3,\dots,a_b,u_1,\dots,u_g}$.

One  defines $\rho$ 
by 
\begin{align*}
  z,t_2&\mapsto z,t_2 \\
 t_3,\dots,t_b,   u_4,\dots,u_g&\mapsto 1 
\end{align*}
and there remains to define the images  $\tilde u_1,\tilde u_2,\tilde u_3  $ of $u_1,u_2,u_3$ in $H$ satisfying the equation
$$  \tilde u_1^2\tilde u_2^2\tilde u_3^2=z^{d_1}t_2z^{d_2}t_2\m z^{d_3+\dots +d_b}=[z^{d_1}t_2,z^{d_2}]z^{2k}$$
with $d_1+\dots +d_b=2k$. 
This is possible since $[a,b]c^2$ is a product of three squares in any group 
(e.g.\ $[a,b]c^2=(ac\m b)^2(b\m c)^2(c\m b a\m b\m c^2)^2$,
which can be viewed as a consequence of the homeomorphism $T^2\# P^2\simeq P^2\#P^2\#P^2$).
One may thus take $\tilde u_1=z^{d_1}t_2 z^{d_2-k}$, $\tilde u_2=z^{k-d_2}$, and $\tilde u_3=z^{-k+d_2}t_2\m z^{2k-d_1-d_2}$. 
Note that 
$\grp{\tilde u_1,\tilde u_2,\tilde u_3}$ is non-abelian unless $k=d_2$ and $2k -d_2=0$, i.e. $k=d_2=0$, which is impossible.

\begin{figure}[ht!]
  \centering
\includegraphics{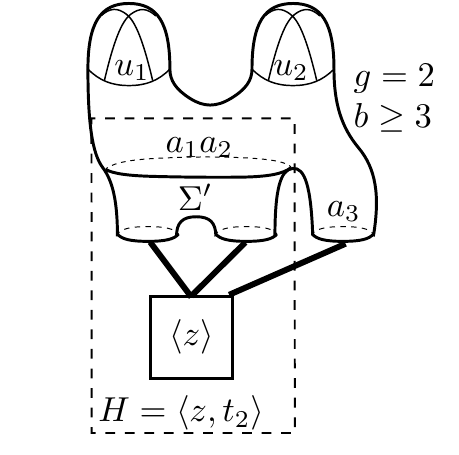} 
  \caption{Case 4: $\Sigma$ is non-orientable with genus $g\leq 2$, and $b\geq 3$}
    \label{fig_nonorient2}
\end{figure}

\emph{Case 4: $\Sigma$ is non-orientable of genus $g= 2$}  (see Figure \ref{fig_nonorient2}).
Since   $|\chi(\Sigma)|\geq 3$, we have $b\ge3$. 
Since  $d_1+\dots+d_b$ is even,  
one may assume (after renumbering) that   $d_2$ or $d_2+d_3$ is even.
As above,   one may take for $\Sigma'$ a pair of pants, so that $G$ has a simple centered splitting with bottom group $H=\grp{z,t_2}$. 
One   defines $\rho$ 
by 
\begin{align*}
z,t_2 & \mapsto z,t_2 &
t_3&\mapsto 1\text{ or }t_2 &
t_4,\dots, t_b&\mapsto 1,
\end{align*}
where we map $t_3$ to $1$ if $d_2$ is even and to $t_2$ otherwise. We treat the case where $d_2$ is even,
the other case being   similar.
To complete the definition of $\rho$, we need to find   images $\tilde u_1,\tilde u_2 $ of $u_1,u_2$ in $H$
  which satisfy the equation $$z^{d_1}t_2z^{d_2}t_2\m z^{d_3+\dots +d_b}= \tilde u_1^2\tilde u_2^2.$$
This is possible because $d_2$ and $d_3+\dots +d_b+d_1$ are even.
  Since   $\rho( a_1  a_2)=z^{ d_1}t_2z^{d_2}t_2\m$ does not commute with $\rho(a_3)$, which is conjugate to $z^{d_3}$, the image $\rho(Q)$ is non-abelian.

\emph{Case 5: $\Sigma$ is non-orientable of  genus 1.}  We show that this cannot happen under the assumption that $\calc$ does not separate boundary components. 
  Being two-sided, all curves in $\calc$ are separating. The components of $\Sigma\setminus \calc$ which do not meet $\partial\Sigma$
are punctured projective planes  or spheres bounded by pinched curves, and their fundamental group 
has trivial  image under $p:\pi_1(\Sigma)\to G$ because $G$ is torsion-free. The connected component containing $\partial \Sigma$ has cyclic image by Lemma \ref{lemcle4} and 
this forces $p$ to have an abelian image, a contradiction.
  \end{proof}

\paragraph{Multi-parachutes.}
  More generally, suppose that $G$ has a  (possibly 
non-simple) centered splitting $\Gamma$ with bottom groups all isomorphic to $\Z$. We call such a group a multi-parachute. For instance, one may connect all boundary components of $\Sigma$ to  distinct cyclic vertex groups;  groups obtained in this way were studied in  \cite {GLS_finite_index} under the name of  socket groups, and elementarily free socket groups are described in \cite[Proposition 6.2]{GLS_finite_index}.

We now study   general multi-parachutes. In particular, we show that they are  either prototypes or elementarily free groups,   and we determine  whether
they have proper elementarily embedded subgroups.
 
  Like parachutes, multi-parachutes are one-ended hyperbolic groups.
 After gluing annuli and M\"obius bands to the surface if needed, we may assume that $\sum d_i(v)\ge3$ for each bottom vertex $v$ (we denote by $d_i(v)$ the indices of incident edge groups in the vertex group), so $\Gamma$ is the cyclic JSJ decomposition of $G$ by Lemma \ref{bonjsj} and Remark \ref{rem_multiparachute}. 
The case of surface groups and parachutes has already been studied, so assume that $\Gamma$ has at least two (cyclic) bottom groups.

If  $\Gamma$ is  not retractable (in particular if $\Sigma$ is exceptional), $G$ is a prototype   by Lemma \ref{preconf},
 so  we now assume that $\Gamma$ is retractable.   In this case $G$ is an extended étage over a free group, hence elementarily free by Corollary \ref{cor_tarski}.  We distinguish three cases, depending on $\Sigma$. 

$\bullet$  If 
$|\chi(\Sigma)|\ge 3$, Proposition \ref{3simple} shows that $G$ is a simple tower over a non-abelian free group,
thus providing an elementarily embedded proper free subgroup  by ``Tarski''.

$\bullet$
If 
$\Sigma$ is a twice-punctured torus, then $\Gamma$ has to be simple 
by Proposition 6.2 of  \cite {GLS_finite_index} (see Remark \ref{ineg}).
 Thus $G$ is a parachute, and it has  an elementarily embedded proper free subgroup by Proposition \ref{3simple2}.

$\bullet$  If 
 $ | \chi(\Sigma) | \le2$ with $\Sigma$  non-exceptional,   not a twice-punctured torus,
 then $G$ is not an  extended \'etage 
over a non-abelian group relative to a non-abelian group (so has no proper elementarily embedded subgroup by Theorem \ref{thcl}).
This is proved in the same way as   Proposition \ref{prop_parachute}.  

 We have proved the following   generalisation of Corollary \ref{synth}.
   \begin{cor}\label{synth2}
     Let $\Gamma$ be a multi-parachute with $\sum d_i(v)\geq 3$ at each bottom vertex.
 \begin{itemize}
 \item If $\Gamma$ is not retractable
   (in particular if $\Sigma$ is exceptional), then $G$ is a prototype. 
 \item If $\Gamma$ is retractable, then $G$ is elementarily free;   in this case, it
   has a 
   proper elementarily embedded subgroup if and only if $|\chi(\Sigma)|\geq 3$ or $\Sigma$ is   a twice-punctured torus. \qed
 \end{itemize}
\end{cor}

One may decide under what conditions $\Gamma$ is retractable. This is done in Propositions 6.2, 6.7, and 6.8 of \cite {GLS_finite_index} for parachutes and socket groups. The other cases may also be worked out, we leave details to the reader.
  The complete list of elementarily free  multi-parachutes having no proper elementarily embedded subgroup (called minimal in the next section) is shown on Figure \ref{fig_minimal}.
We will see that, together with $\F_2$ and $\pi_1(N_4)$, this 
describes all minimal finitely generated elementarily free groups (Theorem \ref{minimal}).

\begin{figure}[ht!]
  \centering
  \includegraphics{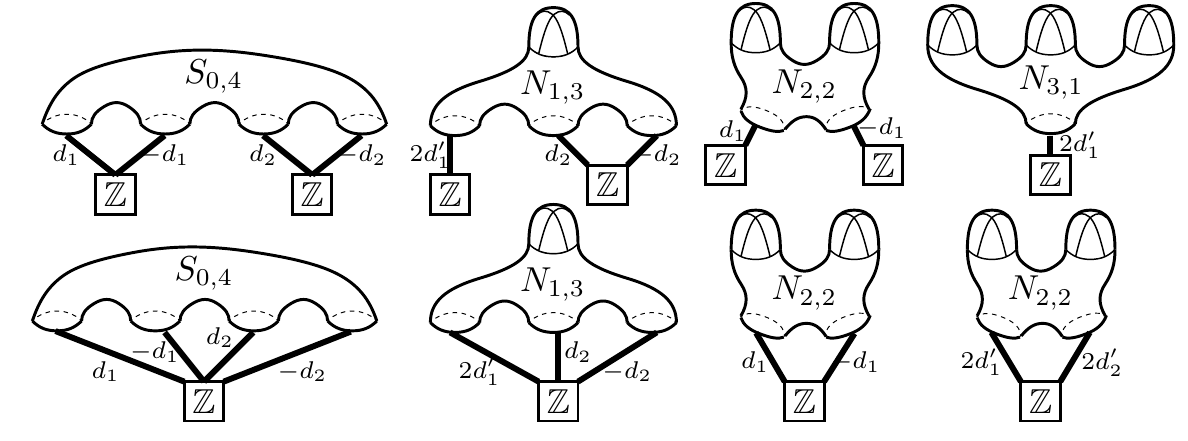}
  \caption{The complete list of finitely generated, minimal, elementarily free groups other than $\F_2$ and $\pi_1(N_4)$;
    the sum of the indices at   each cyclic group should be $\geq 3$.
  }\label{fig_minimal} 
\end{figure}

\section{Minimal and maximal models}\label{sec_minmax}

Although the definitions and facts that follow make sense in the broader context of any first-order structure in a countable language, we restrict ourselves to groups (in the first-order language of groups).  We refer to \cite{Marker_model,Hodges_model} for basic facts.

A group is called {\em minimal} if it has no proper elementarily embedded subgroup. For example, the group of integers is minimal.

A group  $H$ 
is called {\em prime} if it elementarily embeds into any other  group elementarily equivalent to it; equivalently,  
$H$ is countable and, for any tuple $\ul  h\in H^k$, there is a formula $\sigma$ (without constants) such that, given another tuple $\ul  h'\in H^k$, 
one has $H\models \sigma(\ul  h')$ if and only if
$\ul  h$ and $\ul  h'$ are in the same $\Aut(H)$-orbit;
in other words, $H$ is countable and, for all $k\geq 1$ and all tuples $\ul  h\in H^k$, the $\Aut(H)$-orbit of $\ul  h$ is $\es$-definable (i.e.\ definable without constants). 

We shall characterize 
 prime torsion-free hyperbolic groups and minimal elementarily free groups.

A countable group $H$ is \emph{homogeneous} if, for any $k\geq 1$,
  any two tuples $\ul  h,\ul  h'\in H^k$ satisfying the same first-order formulas (without constants) are in the same $\Aut(H)$-orbit.
Equivalently, the $\Aut(H)$-orbit of every finite tuple 
is {\em $\emptyset$-type definable}, i.e.\  definable by infinitely many first-order formulas (without constants).
Prime groups are homogeneous;  conversely, finitely generated free groups are homogeneous \cite{PeSk_homogeneity} but not prime.

A sequence of elementary embeddings $H_1\preceq_e H_2\preceq_e \dots\preceq_e H_n\preceq_e \cdots$ is called 
an \emph{elementary chain} of groups. 

Using a Fra\"{i}ss\'e  construction, 
 we   prove that,
given a   non-abelian 
torsion-free hyperbolic group $G$, there exists a  homogeneous countable (infinitely generated) group $\calm_G\equiv G$
which is the union of an elementary chain of finitely generated subgroups and in which all finitely generated groups elementarily equivalent to $G$ embed elementarily; moreover, $\calm_G$    is unique up to isomorphism. 
 This extends Theorem 5.9 of \cite{KMS_fraisse}.

\subsection{Primality}\label{primi}
Primality is characterized as follows.

\begin{thm}\label{prime}
A non-abelian torsion-free hyperbolic group  $G$ is   prime if and only if $G$ is a one-ended prototype. 
\end{thm}

\begin{proof}
 We first prove that prime groups are prototypes.
    If $G$ is prime and elementarily free, it embeds elementarily into $\F_2$, hence is isomorphic to $\F_2$ by Corollary \ref{cor_FF}.
  But $\F_2$ is not prime because it does not embed elementarily into  $\pi_1(N_4)$ by Lemma \ref{lem_N4} (though $\pi_1(N_4)\equiv\F_2$). This proves that no elementarily free group is prime.

Now suppose that $G$ is prime and not elementarily free  (so that $\ecore(G)=\core(G)\ne\{1\}$). Since   $G$  is elementarily equivalent to  $\ecore(G)$ by Corollary \ref{equivcore},
it embeds elementarily into $\ecore(G)$ by primeness.
But prototypes have no proper elementarily embedded subgroup by Corollary \ref{pasel}, so  this embedding is an isomorphism and $G$ must be a prototype.
This proves that only prototypes may be prime.

 Theorem \ref{contrex}   implies that infinitely-ended prototypes are not prime, and we 
  complete the proof by showing that any one-ended prototype $P$ is prime. 
By Corollary \ref{cor_core_unbout},
  $P$ elementarily embeds into any \emph{finitely generated}   (hence hyperbolic by \cite{Sela_diophantine7}) group $G$ elementarily equivalent to $P$. 
  This is not enough  because it does not apply if $G$ is an infinitely generated   
  group, so we quote  \cite[Th.~5.3]{PeSk_forkingI} to deduce that $P$ is prime.
 \end{proof}

  Though infinitely-ended prototypes are not prime,
  they satisfy a weak ``prime model'' property.

\begin{prop}[Weak prime models]
\label{prop_prime}
A non-abelian torsion-free hyperbolic group $G$ is the free product of a prototype and a free group if and only if, for any finitely generated group $K$ elementarily equivalent to $G$,
the group  $G$ elementarily embeds  into $K*\bbF_r$ for some $r\geq 0$.
\end{prop}

\begin{example}
This applies to the free group   $G=\bbF_n$ for $n\geq 2$ (see Corollary \ref{cor_tarski}).
\end{example}

\begin{proof}
Suppose $G=P*F^1$ with $P$ a prototype and $F^ 1$ free. If $K\equiv G$, then $P$ is the core of $K$, so $K\tge P$. By  Corollary \ref{ett},
there exist free groups $F^2,F^3$ such that $K*F^2$ is a simple tower over   a subgroup isomorphic to $P*F^3$, 
so $K*F^2*F^1$ is a simple tower over   a subgroup isomorphic to $P*F^3*F^1$ and the result follows.

To prove the converse, we let $K$ be  the core $c(G)$ and we assume that $G$ elementarily embeds into $c(G)*\bbF_r$.  By Proposition \ref{ssgpeprot}, $G$ is a co-free free factor of $c(G)*\bbF_r$, hence is isomorphic to the free product of $c(G)$ with a free group.
\end{proof}

\subsection{Minimality} \label{minima}

We saw that prototypes are minimal 
  (Corollary \ref{pasel}).  Conversely: 
  
\begin{thm}\label{minimal}
 Let $G$ be a non-abelian torsion-free hyperbolic group. 
\begin{enumerate}
 \item If $\core(G)$ is one-ended, then $G$ is minimal if and only if $G=\core(G)$ (i.e.\ $G$ is a prototype).
 \item If $\core(G)$ is infinitely-ended, there are infinitely many minimal hyperbolic groups elementarily equivalent to $G$.
 \item If $G$ is elementarily free, then $G$ is minimal if and only if $G$ is isomorphic 
   to $\F_2$,   $\pi_1(N_4)$, or one of the multi-parachutes shown on  Figure \ref{fig_minimal}.
\end{enumerate}
In particular, there are infinitely many minimal groups elementarily equivalent to $G$ (up to isomorphism) if and only if $\core(G)$ is 
  trivial or  infinitely ended.
\end{thm}

\begin{proof}
    Assertion 1 
  follows from Corollaries \ref{cor_core_unbout} and  \ref{pasel}.  
 If $\core(G)$ is infinitely-ended, we write its Grushko decomposition  $\core(G)\simeq A_1*\cdots*A_p$  with $p\ge2$ and we 
consider $G(a_1,a_2)=K(A_1,A_2,a_1,a_2)*A_3*\cdots*A_p$ as in 
    Theorem \ref{contrex}. We have seen that $G(a_1,a_2)$ is minimal and elementarily equivalent to $G$ if $a_i$ is a square and $A_i$ has no cyclic splitting relative to $a_i$.
    
     In order to get infinitely many different groups, we fix $a_1,a_2$ and for $n\ge1$ we consider $G_n=     G(a_1^n,a_2^n)=K(A_1,A_2,a_1^n,a_2^n)*A_3*\cdots*A_p$. These groups are non-isomorphic because by Remark  \ref{indice} their canonical JSJ decompositions are different: they are distinguished by indices of edge groups in vertex groups.

  To prove    Assertion 3, suppose that $G$ is elementarily free, minimal, not $\F_2$. It cannot have a cyclic free factor, so it is an extended \'etage of surface type.
     Let $\Gamma$ be the associated splitting. First suppose that  the associated surface satisfies   $|\chi(\Sigma)|\ge 3$. By Propositions \ref{3simple}
     (if $\Gamma$ is not simple) and \ref{3simple2} (if $G$ is a parachute), $G$ is a simple \'etage over a non-abelian group so is not minimal. Thus $|\chi(\Sigma)|\le 2$. 

  Since $G$ is elementarily free, all bottom groups have trivial core.     
     If some bottom group $H$ is non-abelian, then
$G$ is a tower over $H$ relative to $H$ so is not minimal  by Theorem \ref{thcl}.
It follows that $G$ is an \'etage with all bottom groups isomorphic to $\Z$, i.e.\    a multi-parachute.

 Let $\Sigma'$ be the surface obtained by gluing to $\Sigma$ all annuli or M\"obius bands corresponding to bottom vertices $v$ with $\sum d_i(v)=2$. It satisfies $|\chi(\Sigma')|\le 2$.

If it is  closed, 
it  must be $N_4$ because  
$\pi_1(N_3)$ is not elementarily free and $\pi_1(S_2)$ contains an  elementarily embedded subgroup isomorphic to $\F_2$.  Thus $G\simeq \pi_1(N_4)$.
Otherwise, we get a new multi-parachute 
with $\sum d_i(v)\geq 3$ at each vertex.
 As explained at the end of Section \ref{examp}, $G$ is on the list of Figure \ref{fig_minimal}.

This proves one direction of Assertion 3. The converse follows from Corollary \ref{cor_FF} (if $G$ is free), Lemma \ref{lem_N4} (if $G\simeq \pi_1(N_4)$), and Corollary \ref{synth2}  (if $G$ is a multi-parachute).  
\end{proof}

\subsection{Universality} \label{univ} 
 This subsection is devoted to the following result.

\begin{thm} \label{univers}
Given any  non-abelian  torsion-free hyperbolic group  $G$,  
there exists a unique countable  (infinitely generated) group $\calm_G$ elementarily equivalent to $G$ with the following properties:
\begin{enumerate}
\item $\calm_G$ is homogeneous;
\item  there exists an elementary chain of finitely generated subgroups $G_1\preceq_e G_2\preceq_e \dots\preceq_e G_n\preceq_e \cdots$
  such that $\calm_G=\bigcup_{i\geq 1} G_i$;
\item  all torsion-free hyperbolic groups elementarily equivalent to $G$ elementarily embed into $\calm_G$.  
\end{enumerate}
\end{thm}
     
The proof relies on a Fra\"{i}ss\'e construction. 
The following definition and  Theorem \ref{KMS_fraisse} extend Fra\"{i}ss\'{e}'s original result.  The theorem can be proved following the same line of thought and is considered folklore; nevertheless, a more detailed explanation is given in  Section 2.3 of \cite{KMS_fraisse}.

\begin{dfn}[strong elementary Fra\"{i}ss\'{e}]\label{e-Fraisse}
Let $\mathcal{K}$ be a 
non-empty class of   finitely generated groups 
which is countable (up to isomorphism) and
has the following properties:
\begin{itemize}
 \item  the class $\mathcal{K}$ is closed under isomorphism \emph{(isomorphism property, IP)};
 \item  the class $\mathcal{K}$ is closed under passing to finitely generated elementarily embedded subgroups, 
i.e. if  $K\in \mathcal{K}$ and $K'\preceq_e K$ 
 is a finitely generated elementarily embedded subgroup 
of $K$, then $K'\in\mathcal{K}$ \emph{(elementary hereditary property, e-HP)};
 \item  if $K_1$, $K_2$ are in $\mathcal{K}$, then there is $K$ in $\mathcal{K}$ 
 and elementary embeddings $f_i:K_i\rightarrow K$ for $i\leq 2$ \emph{(elementary joint embedding property, e-JEP)};
 \item  if $K_0$, $K_1$, $K_2$ are in $\mathcal{K}$ and  $f_i:A\rightarrow K_i$,  for $i=1,2$, are partial elementary maps of 
 some    finitely generated subgroup  $A\subset K_0$,  
 then there is a group $K$ in $\mathcal{K}$ and elementary embeddings $g_i:K_i\rightarrow K$ for $i=1, 2$ with $g_1\circ f_1=g_2\circ f_2$ \emph{(strong  
elementary amalgamation property, strong e-AP)}.  
\end{itemize}
Then $\mathcal{K}$ is a \emph{strong elementary Fra\"{i}ss\'{e} class} (strong $e$-Fra\"{i}ss\'{e} class for short). 
\end{dfn}

See Definition \ref{pem} for partial elementary maps.

\begin{rem}
    Note that, by e-JEP, all groups in $\calk$ are elementarily equivalent.
    When 
    $A=\{1\}$,   strong e-AP reduces to e-JEP.
\end{rem}

\begin{rem}\label{rem_eAP}
 Strong e-AP can be reformulated as follows with no reference to $K_0$: 
given $K_1,K_2\in\calk$, a   finitely generated subgroup $A_1\subset K_1$ and a partial elementary map $f:A_1\ra K_2$,
there exists a group $K \in \calk$ and elementary embeddings $g_i:K_i\ra K$ for $i=1,2$ with $g_1=g_2\circ f$.
\end{rem}

\begin{thm}\label{KMS_fraisse}
Let $\mathcal{K}$ be a strong $e$-Fra\"{i}ss\'{e} class. There exists a countable group 
$\mathcal{M}$ such that:
\begin{itemize}
 \item the class of finitely generated groups that elementarily embed in $\calm$ is exactly $\mathcal{K}$; 
 \item the group $\mathcal{M}$ is homogeneous;
 \item the group $\mathcal{M}$ is the union of an elementary chain of groups 
in $\mathcal{K}$.
\end{itemize}
Moreover, any other countable group with the above properties is isomorphic to $\mathcal{M}$. \qed
\end{thm}

 Using this theorem, we prove Theorem \ref{univers} by
  showing   that, given   a torsion-free hyperbolic group $G$, the class $\mathcal{K}_G$ of all finitely generated groups
  elementarily equivalent to $G$ 
  is an elementary Fra\"{i}ss\'{e} class. We just have to check the four properties of Definition \ref{e-Fraisse}.
Recall that $\calk_G$ consists of torsion-free hyperbolic groups \cite{Sela_diophantine7}. 
   
  The isomorphism property (IP) and the elementary hereditary property (e-HP) are obvious. The elementary joint embedding property (e-JEP) follows from   the first assertion of Proposition \ref{prop_lattice},
 which was proved as follows: if $G_1\equiv G_2$, then their cores $H_1,H_2$ are isomorphic by Theorem \ref{sensfac}, and 
elementarily embedded in $G'_1=G_1*\F_{r_1}$, $G'_2=G_2*\F_{r_2}$ for some free groups $\F_{r_1},\F_{r_2}$.
We then defined $G'=G'_1*_{H_1=H_2}G'_2$. 

  The strong  
elementary amalgamation property (strong e-AP)
 may be viewed as a relative version of (e-JEP). It is proved in a similar way,   using Remark \ref{rem_eAP} and replacing Theorem \ref{sensfac} by Theorem \ref{lethm}. 

\appendix
\section{Appendix: Towers 
are limit groups} 

The goal of this section is to prove Corollary \ref{ceta} saying that towers over a torsion-free hyperbolic
group $H$ are $H$-limit groups. This basic fact is stated in \cite{Sela_diophantine7}, but some details  of the proof are missing.
The same arguments  show that, if a finitely generated group has a strict resolution over a hyperbolic group $H$, then it is an $H$-limit group.

\begin{thm} \label{eta}
If $G$ is a simple étage of surface type over a torsion-free hyperbolic group $H$,   there exists a discriminating sequence of retractions $\rho_n:G\ra H$.
\end{thm}

Discriminating means that, given any finite subset $F\inc G$, the map $\rho_n$ is injective on $F$   for all $n$  large enough.  Recall that $H$ must be non-abelian (Remark \ref{pasz}).

\begin{cor} \label{ceta}
If $G$ is a 
tower (simple or extended) over a   non-abelian torsion-free hyperbolic group $H$, then it is an $H$-limit group.
\end{cor}

We first deduce the corollary from the theorem. The remainder of this appendix   will be  devoted to the proof of the theorem.

\begin{proof}
First suppose that $G >G_1>\dots>G_n=H$ is a simple tower.
All groups $G_i$ are hyperbolic, and there exists a discriminating sequence of retractions from $G_{i+1}$ to $G_i$ by Theorem \ref{eta} 
(it is well known that this also holds if $G_{i+1}=G_i*\Z$ is an \'etage of free product type with $G_i$ non-abelian, 
by sending the generator of the cyclic factor to a sequence of small cancellation elements,
see for instance \cite[Th.\ 1.9]{Andre_Tarski}).
Composing these retractions gives a discriminating sequence of retractions from $G $ to $H$, so $G$ is an $H$-limit group. 

If the tower is not simple, some $G*\F_r$ is an $H$-limit group by Corollary \ref{ett}, 
and $G$ is an $H$-limit group because it is a finitely generated subgroup of an $H$-limit group.
\end{proof}

\subsection{Applying Baumslag's lemma}

To 
prove Theorem \ref{eta}, we first need the following generalization of Baumslag's lemma \cite{Baumslag_residually}. For $g\ne1$ in $H$, we denote by $g^{\pm\infty}=\lim_{n\to\pm\infty}g^n$ the limit of $g^n$  in the Gromov boundary $\partial H$.

\begin{lem}[{\cite[Prop.\ 2.19]{Andre_hyperbolicity}}]\label{lem_baumslag}
Let    $a_0,a_1,\dots,a_k,c_1,\dots,c_k$ 
be elements of   a torsion-free hyperbolic group $H$, with  $c_i\ne1$. 
Assume that $a_i\, c_{i+1}^{+\infty}\neq  c_{i}^{-\infty}$ for $i=1,\dots,k-1$.
Then there exists $C\geq 0$ such that 
the element
$$a_0c_1^{p_1}a_1\cdots c_{k}^{p_{k}}a_k$$
is non-trivial for all $p_1,\dots,p_k\geq C$.\qed
\end{lem}

\begin{cor}\label{cor_twist} Suppose that a torsion-free hyperbolic group $G$ 
 is   the fundamental group of a graph of
  groups $\Gamma$ such that all edge groups are   infinite cyclic,  and maximal cyclic in adjacent vertex groups (hence in $G$).   Let   $\hat\tau\in\Out(G)$ be the product of the twists around all edges of $\Gamma$, and $\tau$ a representative of $\hat\tau$ in $\Aut(G)$.
  
If   $f:G\ra H$ is a homomorphism which is injective on each vertex group of $\Gamma$, with 
  $H$   a torsion-free hyperbolic group,
then the sequence of homomorphisms $f\circ\tau^n:G\ra H $ is discriminating.
\end{cor}

Edge groups $G_e$ being infinite cyclic, each edge $e$ carries a twist $\hat\tau_e$; the pair $\{\hat\tau_e,  \hat\tau_e\m\}$  is well-defined in $\Out(G)$, and it does  not matter  for the theorem whether we use $\hat\tau_e$ or $\hat\tau_e\m$ to define $\hat\tau$,  or which representative $\tau$ of $\hat\tau$ in $\Aut(G)$ we use.

The proof of Corollary \ref{cor_twist} is a standard   application of Baumslag's lemma (see for instance \cite[Prop.\ 4.11]{CG_compactifying}).
We prove it in the case of an amalgam $G=A*_{\grp{c}} B$ and leave the general case to the reader.

\begin{proof}[Proof in the case of an amalgam]  
  We choose the representative $\tau$ of $\hat\tau$ equal to the identity on $A$ and to conjugation by $c$ on $B$. It suffices to prove that,  given any $g\in G\setminus \{1\}$, we have  $(f\circ\tau^n)(g)\neq 1$ for all $n$ large enough.
This is obvious if $g$ is contained in $A$ or $B$ (up to conjugacy). 
Otherwise, up to conjugation, $g$ may be written as $g=a_1b_1\dots a_kb_k$ with $a_i\in A_i\setminus\grp{c}$ and $b_i\in B_i\setminus\grp{c}$, and $k\ge1$.

Now $$\tau^n(g)=a_1 c^n b_1 c^{-n} a_2 c^n\dots c^{-n}a_k c^n b_k c^{-n}$$
and 
$$(f\circ\tau^n)(g)=f(a_1) f(c)^n f(b_1) f(c)^{-n}  \cdots f(a_k) f(c)^n f(b_k) f(c)^{-n}.$$

 By assumption, $\grp{c}$ is maximal abelian in $A$ and in $B$, so
$a_i$ and $b_i$ do not commute with $c$.
Since $f$ is injective on $A$ and $B$, the elements $f(a_i)$ and $f(b_i)$ do not commute with $f(c)$,
so  none of the elements $f(a_i).f(c)^{\pm \infty},f(b_i).f(c)^{\pm \infty}$ belongs to $ \{f(c)^{\pm\infty}\}$.
Thus Lemma \ref{lem_baumslag} applies and concludes the proof.
\end{proof}

The idea to prove Theorem \ref{eta} is to apply Corollary \ref{cor_twist} to the   centered splitting $\Gamma$ associated to the \'etage, with $f$ a retraction $\rho:G\to H$, but there are two problems. First, the edge groups of $\Gamma$, though maximal cyclic subgroups of the central group, do not have to be maximal cyclic in bottom groups  (as in Example \ref{ctrex}). This is taken care of by introducing surfaces with sockets. The second problem is that the retraction $\rho:G\to H$ to the bottom group of $\Gamma$ does not have to be injective on the central group. To fix this we will have to refine $\Gamma$, using a multicurve on the surface associated to $\Gamma$, so that $\rho$ is injective on the new vertex groups.

\subsection{Surfaces with sockets}

A \emph{surface group with sockets} $\hat Q$ is the amalgam of the fundamental group of a compact  hyperbolic surface $\Sigma$
with cyclic groups over boundary subgroups of  $\pi_1(\Sigma)$. More precisely, it is the fundamental group
of a (possibly non-minimal) tree of groups $\Lambda$ with a central surface-type vertex $v$ (in the sense of Definition \ref{stype}) carrying the group $Q=\pi_1(\Sigma)$ for some hyperbolic surface   $\Sigma$,
and terminal vertices $v_1,\dots, v_n$ carrying $\bbZ$. We say that $\hat Q$ is \emph{non-exceptional} if the surface $\Sigma$ is non-exceptional.

This is similar to a centered splitting with all bottom groups $\Z$, with two differences. First,   there is a single edge $e_i$ joining $v$ to a given $v_i$. 
Second, $\Lambda$ does not have to be minimal: $G_{e_i}$ may be equal to $G_{v_i}$;  if this does not happen, $\hat Q$ is a 
 particular kind of a multi-parachute in the sense of Subsection \ref{parac}.

The group carried by $e_i$ is the fundamental group of a component $C_i$ of $\partial \Sigma$, and $C_i\ne C_j$ for $i\ne j$. It has finite index $d_i$ in the vertex group $G_{v_i}$. We   allow 
$d_i=1$ (i.e.\ $G_{e_i}=G_{v_i}$) and  $d_i=2$ 
(this corresponds to gluing a M\"obius band
to $C_i$).

In topological terms,  $\hat Q$ is the fundamental group of a graph of spaces $X_{\hat Q}$ obtained by gluing each boundary component $C_i$ to a circle by a covering of degree $d_i$.

The \emph{socket subgroups} of $\hat Q$ are the conjugates of    the vertex groups $G_{v_i}$.
  A homomorphism from $\hat Q$ to an arbitrary group is injective on each boundary subgroup of $\pi_1(\Sigma)$ if and only if  it is injective on each socket subgroup of $\hat Q$.

An \emph{essential multicurve} in $\Sigma$ is  a collection   $\calc$ of disjoint non-parallel 2-sided simple closed curves, each of which is not boundary parallel
and does not  bound a M\"obius band   or a disc. 
The cyclic splitting of $\pi_1(\Sigma)$ dual to $\calc$ induces a cyclic splitting of $\hat Q$. The  
vertex groups of this splitting
are  fundamental groups of  connected components of $X_{\hat Q}\setminus \calc$, they have 
 a natural structure of a surface with sockets (curves in   $\calc$ give rise to sockets of index 1). 

\subsection{A reduction}
Following Sela, the proof of Theorem \ref{eta} is based on the following result. 

\begin{thm} \label{thm_courbes}
Let $\hat Q$ be a non-exceptional surface group with sockets, with underlying surface $\Sigma$, and let $H$ be a torsion-free hyperbolic group.
Let $f:\hat Q\ra H$ be a homomorphism with non-abelian image, which is injective on each socket subgroup of $\hat Q$. 

There exists an essential multicurve $\calc$ in $\Sigma$ such that 
the restriction of $f$ 
to the fundamental group of each connected component of $X_{\hat Q}\setminus \calc$
  is injective.   
\end{thm}

\begin{rem}
Theorem \ref{thm_courbes} also applies if $\Sigma$ is  the closed non-orientable surface $N_3$ of genus $3$ (with no sockets), see Remark \ref{ntrois}.
\end{rem}

\begin{proof}[Proof of Theorem \ref{eta} from Theorem \ref{thm_courbes}]
 Let $\Gamma$ be the simple centered splitting  of $G$ associated to the \'etage, with bottom group $H$ and central group $Q$. 
  Perform  folds to transform $\Gamma$ into
  a splitting $\hat\Gamma$ over maximal cyclic subgroups; this replaces  $Q$  by a surface group with sockets $\hat Q$ without changing 
  $H$  (compare Remark \ref{indice}, where one obtains $\hat \Gamma$ from $\Gcan$ by collapsing the edges $vv_i$).
  
Let $\rho:G\ra H$ be a retraction associated to the étage,    as in Proposition \ref{equivsimple}. Its restriction to $\hat Q$ satisfies the assumptions of Theorem \ref{thm_courbes}, so we can 
refine the splitting $\hat \Gamma$ using the essential multicurve provided by the theorem.
This provides a new splitting $\hat \Gamma_c$  over maximal cyclic subgroups such that $\rho$ is injective in restriction
to every vertex group. 

We apply  Corollary \ref{cor_twist} to this splitting.
Since $H$ is a vertex group of $\hat\Gamma_c$, the outer automorphism $\hat\tau$ has a representative  $\tau$ in $\Aut(G)$ which
 is the identity on $H$ (the assumption that the \'etage is simple is used here). 
 Thus the maps $\rho_n:=\rho\circ \tau^n:G\ra H$ are retractions. By Corollary \ref{cor_twist}, this sequence of retractions is discriminating.
\end{proof}

The following subsections are devoted to the proof of Theorem \ref{thm_courbes}.
We first prove in subsection \ref{sec_nonab} that one can decompose the surface $\Sigma$  into pairs of pants and twice-punctured projective planes whose fundamental groups 
all  have non-abelian image. We then prove the theorem in Subsection \ref{sec_tore} by   using  the fact that $\Sigma$ is non-exceptional and
  reducing the general case to two special cases: 
the   once-punctured torus and the 4-punctured sphere. The proof in these special cases requires ping-pong lemmas which are proved in Subsection \ref{ping}.

\subsection{Decomposing into pieces with non-abelian image}\label{sec_nonab}

The first step in the proof of Theorem \ref{thm_courbes} is the following lemma. Recall (Definition \ref{pinc}) that a  two-sided simple closed curve is pinched under a homomorphism if   its fundamental group is contained in the kernel.

\begin{lem}\label{lem_nonab}
Let $H$ be a torsion-free CSA group,
and let $f:\hat Q\ra H$ be a homomorphism with non-abelian image, with $\hat Q$   a   surface group with sockets. Assume that no component of $\partial \Sigma$ is pinched.

There exists an essential multicurve     $\calc_0$ 
consisting of non-pinched curves
such that every connected component of $\Sigma\setminus \calc_0$
is a pair of pants or a twice-punctured projective plane whose fundamental group has non-abelian image under $f$.
\end{lem}

We do not assume that the underlying surface $\Sigma$ is non-exceptional.

\begin{rem}\label{rk_socket}
  As a preliminary remark, we note that, if $\hat Q$ is obtained from a
  group $Q=\pi_1(\Sigma)$ by adding sockets,
and $f:\hat Q \ra H$ takes values in a torsion-free commutative
  transitive group,
    then   $f(\hat Q)$ is non-abelian if and only if  $f(Q)$ is
  non-abelian. We give the argument when there is only one socket. Write $\hat Q=\grp{Q,u}$ with $u^d\in Q$, and assume that $f(Q)$ is abelian. 
  If $f(u^d)$
  is trivial,   so is $f(u)$ because $H$ is torsion-free, and $f(\hat Q)=f(Q)$ is abelian.
  Otherwise, both $f(Q)$ and $f(u)$ commute with $f(u^d)$, hence
  commute   by commutative transitivity, so $f(\hat Q)$ is abelian.
\end{rem}
This remark shows that we may forget about sockets in Lemma \ref{lem_nonab}: 
it suffices to prove the lemma for the fundamental group of a surface with boundary.
It is an immediate consequence of the following.

\begin{lem}[{\cite[Lemma 5.13]{Sela_diophantine1}, \cite[Lemma 1.32]{Sela_diophantine7}}]
Let $\Sigma$ be a 
compact surface 
which is not homeomorphic to a pair of pants or   a twice-punctured projective plane. Let $H$ be a torsion-free CSA group,
and $f:\pi_1(\Sigma)\ra H$   a homomorphism with non-abelian image.  Assume that no component of $\partial \Sigma$ is pinched.

There exists  a non-pinched essential two-sided simple closed curve $\gamma\subset \Sigma$ 
  such that the fundamental group of every connected component of $\Sigma\setminus \gamma $ has 
non-abelian image under $f$.
\end{lem}

\begin{proof}
Let $Q=\pi_1(\Sigma)$.  We start with two observations.

Obervation 1: \emph{if there exists a non-separating two-sided non-pinched curve $\gamma$,
then the fundamental group  $J$ of  
$\Sigma\setminus \gamma$ automatically has
non-abelian image.} Indeed, one has an HNN decomposition $Q=J*_{\pi_1(\gamma)}$ and, if $f(J)$ is abelian,
then the image of the stable letter has to commute with $f(J)$ by the CSA property.

The lemma follows in  the case when  $\Sigma$ is orientable with at most one boundary component, because $\pi_1(\Sigma)$ 
is generated by fundamental groups of non-separating essential simple closed curves.

Observation 2:  \emph{if $\Sigma$ contains a  two-sided non-separating simple closed curve $\gamma$ which is pinched, and $\partial \Sigma\neq\es$,
  one can find another such curve $\gamma'$ which is non-pinched} as follows:  
take an arc $I$ joining $\gamma$ to a boundary component $b$, and consider the curve $\gamma'$ obtained as the connected sum of $b$ and $\gamma$ along $I$.
Since $\gamma$ is pinched,  the elements of $\pi_1(\Sigma)$ corresponding to $b$ and $\gamma'$ have conjugate images under $f$, so $\gamma'$ is not pinched. 
Moreover, if $I$ is chosen disjoint from some simple closed curve $\delta$ with $i(\delta,\gamma)=1$,
then $i(\delta,\gamma')=1$ so $\gamma'$ is   non-separating.

Using these two observations repeatedly, we see that there remain three cases to analyze:
the $n$-punctured sphere ($n\ge4$), the $n$-punctured projective plane ($n\ge3$), and the   closed non-orientable surfaces.

$\bullet$ If $\Sigma$ is an $n$-punctured sphere with $n\geq 4$, then $Q=\grp{a_1,\dots,a_n\mid a_1\cdots a_n=1}$.
Since $H$ is CSA, there exists $i$ such that $f(a_i)$ does not commute with $f(a_{i+1})$,
and we may assume $i=1$ without loss of generality.
The subgroup $\grp{a_1,a_2}$ is the fundamental group of a pair of pants contained in $\Sigma$; it has  non-abelian image
and its third boundary curve, whose fundamental group is generated by   $a_1a_2$, is not pinched.

The fundamental group of the complement of this pair of pants is   generated by $a_3,\dots,a_n$  and contains $a_1a_2$.
We are done if it has non-abelian image under $f$, so assume otherwise.
Then  $f(a_3), \dots,f(a_n),f(a_1a_2)$ are non-trivial commuting elements. 
Let $A$ be the maximal abelian group containing $f(\grp{a_3,\dots,a_n})$ (hence also $f(a_1a_2)$).

We have $f(a_2)\notin A$ since otherwise   $f(Q)=A$, 
contradicting  the non-abelianity of $f(Q)$. 
By commutative transitivity, $f(a_2)$ does not commute with $f(a_3)$, and similarly  $f(a_n)$ does not commute with $f(a_1)$.
We then let $\gamma$ be a curve bounding 
a pair of pants with fundamental group
$\grp{a_2,a_3}$.

$\bullet$ If $\Sigma$ is an $n$-punctured projective plane with $n\geq 3$, we claim that there exists an embedded M\"obius band  $M\subset\Sigma$
whose fundamental group is not killed by $f$. 
If the claim is true, then one can write $\Sigma$ as an $(n+1)$-punctured sphere $\Sigma'$ with $M$ glued along a boundary component, which can be
viewed as $\Sigma'$ with a socket of index 2.
Since $n+1\geq 4$ and  
boundary components of $\Sigma'$ are not pinched, 
one reduces
to the previous case  using Remark \ref{rk_socket}.

To prove the claim, consider  a one-sided simple closed curve $\gamma$ (whose regular neighborhood is a M\"obius band),
and assume that $ \pi_1(\gamma)$ is killed.
Let $I$ be an arc joining $\gamma$ to a boundary component $b$. Then the concatenation $\gamma.I.b.\ol I$
is homotopic to a simple closed curve which is one-sided, and whose image under $f$ is conjugate to that of $b$ hence is non-trivial.

$\bullet$ If $\Sigma$ is a closed non-orientable surface   (of genus at least 3 because its  fundamental group  has non-abelian image in $H$), one can write it as the union of an orientable surface $\Sigma'$ with one boundary component
and  a M\"obius band or a punctured Klein bottle. If $\Sigma'$ contains a non-pinched non-separating curve, we are done.
Otherwise, being generated by such curves, $\pi_1(\Sigma')$ has trivial image under $f$, so $f$ factors through the fundamental group
of a projective plane or a Klein bottle. Since $H$ is CSA, this implies that the image of $f$ is abelian, a contradiction.
\end{proof}

\subsection{Proof of Theorem \ref{thm_courbes}}\label{sec_tore}

Consider the decomposition of the surface $\Sigma$  (assumed to be non-exceptional) into pairs of pants and twice-punctured projective planes 
given by Lemma \ref{lem_nonab}.
We first claim that, up to adding extra sockets, we may assume that all pieces of this decomposition are pairs of pants.

Indeed, 
one can view a twice-punctured projective plane $P$ as a pair of pants $P'$ with a M\"obius band
glued to one of its boundary components; this corresponds to adding a socket of index 2 to $P'$.
The fundamental group of the M\"obius band (and of the corresponding boundary component of $P'$) have non-trivial image by $f$ since
otherwise $\pi_1(P)$ would have abelian image.
By Remark \ref{rk_socket}, the image of the fundamental group of $P'$ is also non-abelian.
  The claim is then proved by replacing each twice-punctured projective plane by  a pair of pants with sockets.

We thus have a decomposition of the space $X_{\hat Q}$ into pairs of pants with sockets $P_i$ whose fundamental groups have non-abelian
images under $f$. We are done if $f$ is injective on every $\pi_1(P_i)$, so 
let $P=P_i$ be a pair of pants with sockets   such that $f$ is not injective on  $\pi_1(P)$.

  If  the Euler characteristic of $\Sigma$ is $-1$, then $\Sigma$ is a   once-punctured torus because it is assumed to be non-exceptional. This is a special case which will be treated below.
Otherwise, there exists a pair of pants $P'\neq P$ adjacent to $P$ in the decomposition,   and $P\cup P'$ is 
a 4-punctured sphere (possibly with boundary components identified). 
If we prove the theorem when $\Sigma$ is either a   once-punctured torus or a 4-punctured sphere, the general case follows by induction on the number of pairs of pants $P$ such that $f$ is not injective on  $\pi_1(P)$.

$\bullet$ The punctured torus.

Let $d\geq 1$ and let $\hat Q=\grp{a,b,c\mid [a,b]=c^d}$ be the fundamental group of a punctured torus with a socket of index $d$ (with $d=1$ if there is no socket),
where the fundamental group $Q$ of the punctured torus $\Sigma$ is the subgroup $\grp{a,b}$,    the curve $C$ provided by Lemma \ref{lem_nonab} 
has fundamental group $\grp{a}$, and 
cutting $\Sigma$ along $C$ produces a pair of pants $P$   whose fundamental  group $\grp{a,bab\m}=\grp{a,c^d}=\grp{bab\m,c^d}$ has non-abelian image.

Let $\gamma_n$ be the curve obtained by twisting the curve representing $b$ 
  $n$ times around $C$. 
  Cutting the graph of spaces $X_{\hat Q}$ representing $\hat Q$  along $\gamma_n$ creates a pair of pants with socket whose fundamental group is  a free group $\hat P_n=\grp{ba^n,c}$.

When applied to $(u,v,w)=(f(a),f(b),f(c))$, Corollary \ref{lem_twist_tore}, stated and proved  in the next subsection,  ensures that 
$f(\hat P_n)$ is free of rank 2   for all $n$ large enough, so the restriction of $f$ to $\hat P_n$ is injective. This proves Theorem \ref{thm_courbes} when $\Sigma$ is a   once-punctured torus.

$\bullet$ The 4-punctured sphere.

Using Lemma \ref{lem_nonab}, we may assume that 
  $G$ has a presentation of the form
  $\grp{a_1,\dots,a_4\mid a_1^{d_1}a_2^{d_2}a_3^{d_3}a_4^{d_4}=1}$, with   $\pi_1(\Sigma)=\grp{a_1^{d_1},\dots,a_4^{d_4}}$,  such that  the simple closed curve $\gamma$ represented by $g:=a_1^{d_1}a_2^{d_2}=(a_3^{d_3}a_4^{d_4})\m$ is not pinched by $f$, and the associated pairs of pants with sockets   $ \grp{a_1,a_2}$ and $ \grp{a_3,a_4}$ have non-abelian image by $f$.
 Note that $f(g)$ does not commute with any $f(a_i)$:  if for instance it commutes with $f(a_1)$,
then it   also commutes with $f(a_2^{d_2})=f(a_1^{-d_1} g)$ and $f(\grp{a_1,a_2})$ is abelian by commutative transitivity, 
a contradiction.

  We consider   a second    decomposition of $X_{\hat Q}$ into two pairs of pants with sockets, corresponding to $P_0=\grp{a_2,a_3}$ and $P'_0=\grp{a_4,a_1}$. Twisting it $n$ times around $\gamma$, we obtain a decomposition into   $P_n=\grp{a_2,g^n a_3 g^{-n}}$ and
  $P'_n=\grp{ g^n a_4 g^{-n},a_1}$.
 Since $f(g)$ commutes with no $f(a_i)$, Lemma \ref{lem_pingpong2} applies and ensures that $f$ is injective on $P_n$ and $P'_n$ for $n$ large.

\begin{rem} \label{ntrois}
    To prove the theorem when  $\hat Q$ is the    fundamental group of the closed exceptional surface   $N_3$, one views $\hat Q$ as a     once-punctured torus with a socket of  index $d=2$. One just has to check that $f$ is injective on 
the socket group. This is true  
since otherwise  $f$ would factor by $\bbZ^2$, the fundamental group of the closed torus,
and would therefore have abelian image. 
\end{rem}

\subsection{Ping-pong lemmas} \label{ping}

  Let $H$ be a 
hyperbolic group (torsion is allowed).  We fix a finite generating set and we let $\calg$ be the corresponding Cayley graph. Let $\bar\calg=\calg\cup\bo\calg$ be its compactification, obtained by adding the Gromov boundary $\bo\calg=\bo H$ (see \cite{GhHa}).   The action of $H$ on itself by left-translations extends to a continuous action on $\bar\calg$.

For $g$ of infinite order, we denote by  $g^{\pm\infty}=\lim_{g\to\pm\infty} g^n$
 the attracting and repelling fixed points of $g$ in   $\partial \calg$.  It is a standard fact that $g$ has North-South dynamics on 
$\bar\calg$: given neighborhoods $W^\pm$ of $g^{\pm\infty}$, one has $g^n(\bar\calg\setminus W^-)\inc W^+$ and $g^{-n}(\bar\calg\setminus W^+)\inc W^-$ for $n$ large.
If $g,h$ have infinite order and $\grp{g,h}$ is not virtually cyclic,
the four points $g^{\pm\infty},  h^{\pm\infty}$ are distinct  (this applies in particular if $H$ is torsion-free and $g,h$ do not commute).

   The following lemmas are proved by standard arguments, but we have not found them in the literature.

\begin{lem}\label{lem_vun}
Let   $u,v,w\in H$, with $u,w$ of infinite order (we allow $v=1$).  Let $\xi^+=vu^{+\infty}$ and $\xi^-=u^{-\infty}$. Assume that $\xi^{+},\xi^{-}, w^{+ \infty},w^{-\infty}$ are distinct, 
and  $w^i\xi^{-}\neq \xi^+$ for all integers $i$.
For $n$ large enough, $\grp{vu^n,w}$ is    a free subgroup freely generated by $vu^n$ and $w$.
\end{lem}

\begin{lem} \label{lem_pingpong2}
Let   $u,v, w\in H$ have infinite order. Assume    that 
  $\{w^{\pm \infty}\}\cap \{u^{\pm\infty} \}=\{w^{\pm \infty}\}\cap \{v^{\pm\infty} \}=\es$ (we allow $u^{\pm\infty} =v^{\pm\infty} $).
For $n$ large enough, $\grp{u,w^n vw^{-n}}$ is    a free subgroup freely generated by $u$ and $w^n vw^{-n}$.
\end{lem}

\begin{proof}[Proof of Lemma \ref{lem_vun}] 
Let $x$ be a basepoint in $\calg$. Using North-South dynamics of $w$, we can find  neighborhoods   $W^{\pm}$ of $\xi^{\pm}$ in $\bar\calg$  such that the neighborhoods $w^i W^{\pm}$ for $i\in\Z$ and the points $w^jx$ for $j\in\Z$ are all disjoint.
Denoting $v_n=vu^n$, we have $v_n(\bar\calg\setminus W^-)\subset W^+$ and  $v_n\m (\bar\calg\setminus W^+)\subset W^-$ for $n$ larger than some $n_0$: this follows from North-South dynamics of $u$, noting that $v\m W^+ $ is a neighborhood of $u^{+\infty}$.  
Note that, for $n>n_0$, one has  
 $v_n^i(\bar\calg\setminus W^-)\subset W^+$ and  $v_n ^{-i}(\bar\calg\setminus W^+)\subset W^-$ for all $i>0$.   
In particular, $v_n $  has infinite order.

We can now play ping-pong. For $n>n_0$, consider an element
$$g=v_n^{n_k}w^{p_{k-1}}\cdots v_n^{n_2}w^{p_1}v_n^{n_1}w^{p_0}$$
in $\grp{v_n,w}$, with all exponents nonzero and $k\ge1$. Let  $\varepsilon_i\in\{\pm\}$ be the sign of $n_i$. We show $gx\in W^{\varepsilon_k}$ (this implies $gx\ne x$, hence $g\ne 1$).
The point $w^{p_0}x$ is not in $W^{-\varepsilon_1}$, so $v_n^{n_1}w^{p_0}x\in W^{\varepsilon_1}$. We then have $w^{p_1}v_n^{n_1}w^{p_0}\in w^{p_1}W^{\eps_1}$, so 
$w^{p_1}v_n^{n_1}w^{p_0}\notin  W^{-\eps_2}$ and  $v_n^{n_2} w^{n_2}v_n^{n_1}w^{p_0}\in W^{\eps_2}$. Iterating gives the desired result.
\end{proof}

\begin{proof}[Proof of Lemma \ref{lem_pingpong2}]
The proof is similar.
We choose neighborhoods $W^{\pm}$ of $w^{\pm\infty}$ in $\bar\calg$ such that
  no $u^ix$ belongs to $W^{+}$,   the neighborhoods $u^i W^+$ for $i\in\Z$ are disjoint, and the neighborhoods $v^iW^-$ are disjoint. For $n$ larger than some $n_0$, we have $w^n(\bar\calg\setminus W_-)\subset W_+$
and $w^{-n}(\bar\calg\setminus W_+)\subset W_-$.
  
We consider an element  $$g=w^n v^{p_k}w^{-n} \cdots u^{n_2}w^n v^{p_1}w^{-n} u^{n_1}$$
in $\grp{u,w^nvw^{-n}}$ and we show $gx\in W^+$ (as above, $k\ge1$ and all exponents are nonzero).
We have $u^{n_1}x\notin W^+$, hence $w^{-n} u^{n_1}x\in W^-$, so $ v^{p_1}w^{-n} u^{n_1}x\notin  W^- $ and $w^n v^{p_1}w^{-n} u^{n_1}x\in  W^+ $. Applying $u^{n_2}$ takes the point outside of $W^+$ and we can iterate.
\end{proof}

\begin{cor}\label{lem_twist_tore}
Let  $H$ be  torsion-free. Let $u,v,w\in H$, with $u,w$   non-trivial.  Assume that $w$   commutes  with  none of  $u$ or $vuv\m$.
For $n$ large enough, $\grp{vu^n,w}$ is a free subgroup freely generated by $vu^n$ and $w$.
\end{cor}

\begin{proof} 
We check that Lemma \ref{lem_vun} applies. The points $w^{+ \infty},w^{-\infty}$ are distinct from $\xi^-= u^{-\infty}$ and $\xi^+ =(vuv\m)^{+\infty}$ because $w$ does not commute with 
$u$ and $vuv\m$.
  Since $H$ is torsion-free, no element exchanges the points $u^{+\infty}$ and $u^{-\infty}$,     
so  $\xi^+ $ and $\xi^- $ are not in the same $H$-orbit. 
 \end{proof}

\bibliography{published,unpublished} 

 \bigskip

\begin{flushleft}

Vincent Guirardel\\
  Univ Rennes, CNRS, IRMAR - UMR 6625, F-35000 Rennes, France\\
\emph{e-mail: }\texttt{vincent.guirardel@univ-rennes1.fr}\\[8mm]

Gilbert Levitt\\
Laboratoire de Math\'ematiques Nicolas Oresme, Universit\'e de Caen et CNRS (UMR 6139)\\
(Pour Shanghai : Normandie Univ, UNICAEN, CNRS, LMNO, 14000 Caen, France)\\
 \emph{e-mail: }\texttt{levitt@unicaen.fr}\\[8mm]

Rizos Sklinos\\ 
Department of Mathematical Sciences\\
Stevens Institute of Technology\\
1 Castle Point Terrace, Hoboken, 07030, NJ, USA\\
\emph{e-mail: }\texttt{rizozs@gmail.com}\\[8mm]

\end{flushleft}

\end{document}